\documentclass[12pt,reqno]{amsart}
\usepackage[T2A]{fontenc}
\usepackage[cp1251]{inputenc}

\usepackage{geometry} 
\usepackage{amsmath, amsthm, amssymb}
\usepackage{hyperref}
\usepackage{epigraph}
\usepackage{mathtools}
\usepackage{longtable}
\usepackage{pifont}
\usepackage{scalerel}
\usepackage{array}
\usepackage{lipsum}
\usepackage{scrextend}
\hypersetup{colorlinks=true,linkcolor= Mahogany, linktocpage, citecolor = OliveGreen}
\usepackage{graphicx} 
\usepackage{mathrsfs} 
\usepackage{color}
\usepackage{verbatim}
\usepackage[dvipsnames]{xcolor}
\usepackage{leftidx}

\usepackage{csquotes}

\usepackage[new]{old-arrows}

\usepackage{csquotes}

\usepackage{enumitem}
\setlist[enumerate,1]{label=(\arabic*),ref=\arabic*$^\circ$}

\newcommand{\sbullet}[1][.7]{\mathbin{\vcenter{\hbox{\scalebox{#1}{$\bullet$}}}}}

\usepackage{tikz-cd}
\usetikzlibrary{arrows}
\usepackage{graphicx}

\usepackage{geometry}

\usepackage[all]{xy}
\usepackage{hyperref}

\newcounter{eq}
\newcounter{intr}

\newtheorem{thmintr}[intr]{Theorem}
\newtheorem{propintr}[intr]{Proposition}

\newtheorem{prop}[eq]{Proposition}
\newtheorem{thm}[eq]{Theorem}
\newtheorem{cor}[eq]{Corollary}
\newtheorem{lemma}[eq]{Lemma}
\theoremstyle{definition}
\newtheorem{df}[eq]{Definition}
\newtheorem{rem}[eq]{Remark}

\newtheorem{example}[eq]{Example}

\numberwithin{eq}{subsection}
\numberwithin{equation}{subsection}

\DeclareMathOperator{\SheafHom}{\mathcal{H\kern -3pt o\kern -2pt m\kern -1pt}}

\newcommand{\A}{\mathbb A}

\newcommand{\Z}{\mathbb Z}

\address{\parbox{\linewidth}{
Department of Mathematics and Mechanics, Lomonosov Moscow State University, Moscow 119234, Russia
}}
\email{dmitrii.badulin@math.msu.ru, dbadulin28@gmail.com}

\let\oldtocsection=\tocsection

\let\oldtocsubsection=\tocsubsection

\let\oldtocsubsubsection=\tocsubsubsection

\renewcommand{\tocsection}[2]{\hspace{0em}\oldtocsection{#1}{#2}}
\renewcommand{\tocsubsection}[2]{\hspace{1em}\oldtocsubsection{#1}{#2}}
\renewcommand{\tocsubsubsection}[2]{\hspace{2em}\oldtocsubsubsection{#1}{#2}}

\begin{document}
\author{
Dmitry Badulin
}

\title{Embeddings and intersections of adelic groups 
}
\date{}

\maketitle

\begin{abstract}
We prove embeddings of adelic groups on an excellent scheme of special type and a flat quasicoherent sheaf on it. For a normal excellent scheme of special type we establish the equality $\A_I(X,\mathcal{F})\cap\A_J(X,\mathcal{F})=\A_{I\setminus0}(X,\mathcal{F})$ in the case $I\cap J=I\setminus0$. We show that the limit of restrictions of global sections of a locally free sheaf on a Cohen--Macaulay projective scheme to power thickenings of integral subschemes equals the group of global sections of this sheaf. Using this result, we deduce a theorem on intersections of adelic groups for normal projective surfaces. We also compute cohomology groups of a curtailed adelic complex and, as a consequence, show that on a three-dimensional regular projective variety over a countable field the intersection $\A_I(X,\mathcal{F})\cap\A_J(X,\mathcal{F})$ equals $\A_{I\cap J}(X,\mathcal{F})$ for any $I,J\subset\{0,1,2,3\}$ and any locally free sheaf $\mathcal{F}$ on $X$.

\end{abstract}

\tableofcontents

\section*{Introduction}\label{sect1}
A generalization of adeles on curves to the higher-dimensional case was first given by A.N. Parshin in \cite{Pa1} for the case of a surface. Later, A.A. Beilinson introduced a simplicial approach and generalized the notion of adeles for Noetherian schemes and quasicoherent sheaves on them in \cite{Be}. More precisely, a complex of abelian groups
$$\mathbb{A}^0(X, \mathcal{F})\overset{d^1}{\longrightarrow} \mathbb{A}^1(X, \mathcal{F})\overset{d^2}{\longrightarrow}\ldots \overset{d^{n}}{\longrightarrow}\mathbb{A}^n(X, \mathcal{F})\overset{d^{n + 1}}{\longrightarrow} \ldots$$
is associated with any Noetherian scheme $X$ and a quasicoherent sheaf $\mathcal{F}$ on $X$. This complex is called the {\it adelic complex}. By the theorem (see \cite{Be}, \cite[Theorem 4.2.3]{Hu}, or \cite[Theorem 1]{Pa1} for the case of a surface), the cohomology groups of this complex are equal to the cohomology groups $H^i(X, \mathcal{F})$ of the sheaf $\mathcal{F}$. Each term of this complex is a ``restricted product'' of adelic local factors $\A_\Delta(X, \mathcal{F})$ in simplexes $\Delta = (p_0, \ldots, p_n)$ of points of $X$. Considering non-degenerate simplexes, we obtain the complex of reduced adeles $\A_{\mathrm{red}}^{\sbullet}(X, \mathcal{F})$. If $X$ is finite-dimensional, then there is a decomposition
 $$\A^n_{\mathrm{red}}(X, \mathcal{F}) = \bigoplus\limits_{\substack{I\subset \{0, 1, \ldots, \dim X\}\\ |I| = n + 1}}\A_I(X, \mathcal{F}),$$
 where each term $\A_I(X, \mathcal{F})$ is a ``restricted product'' of local factors $\A_\Delta(X, \mathcal{F})$ such that the set of codimensions of points of $\Delta$ equals $I$.

On a regular surface over a field, A.N. Parshin noticed from the explicit description of local factors (see \cite[Section 2]{Pa1}, \cite[Section 3.3]{Os2}, \cite[Section 8.5]{Mo}) that
$$\A_I(X, \mathcal{O}_X)\longhookrightarrow \A_J(X, \mathcal{O}_X),$$
if $I\subset J\subset \{0, 1, 2\}$. He also showed in \cite{FP} that for a projective regular surface, the equality
\begin{equation}\label{eqParshinIntersectionAdeles}
\A_I(X, \mathcal{O}_X)\cap \A_J(X, \mathcal{O}_X) = \A_{I\cap J} (X, \mathcal{O}_X)  \tag{1}
\end{equation} 
holds for arbitrary $I, J\subset \{0, 1, 2\}$, where we assume $\A_{\varnothing}(X, \mathcal{F}) = H^0(X, \mathcal{F})$. R.Ya. Budylin and S.O. Gorchinskiy generalized \eqref{eqParshinIntersectionAdeles} for the case of a regular projective surface and a locally free sheaf on it in \cite{BG}. The equality \eqref{eqParshinIntersectionAdeles} was also used in \cite{OP} to prove the Riemann--Roch theorem on a surface by means of harmonic analysis on local fields and adelic spaces.

In view of the equality \eqref{eqParshinIntersectionAdeles}, A.N. Parshin formulated the following question in \cite[Remark~5]{Pa3}:

\vspace{0.25cm}

 \begin{center}
 \blockquote{\it Is it true that on a Cohen--Macaulay projective variety $X$, the equality 
 $$\A_I(X, \mathcal{F})\cap \A_J(X, \mathcal{F}) = \A_{I\cap J}(X, \mathcal{F})$$ 
 \it holds for every locally free sheaf $\mathcal{F}$ on $X$ and any $I, J\subset \{0, 1, \ldots, \dim X\}$?}
 \end{center}
 \vspace{0.25cm}
 Here we assume that $\A_{\varnothing}(X, \mathcal{F}) = H^0(X, \mathcal{F})$. Note that the projectivity assumption cannot be dropped, since R.Ya. Budylin and S.O. Gorchinskiy provided a counterexample to this question in \cite{BG} for the case of an affine surface over a countable field and the structure sheaf on it. Nevertheless, as they also showed, the question is true in the case of a regular (not necessarily projective) surface over an uncountable field. 

In order to solve the question posed by Parshin, it is necessary to correctly determine the intersections of adelic groups. In other words, we must prove the embeddings
\begin{equation}\label{eqIntrEmbed}
\A_I(X, \mathcal{F})\longhookrightarrow \A_J(X, \mathcal{F}) \tag{2}
\end{equation}
for a Cohen--Macaulay variety $X$ and flat quasicoherent sheaf $\mathcal{F}$ on $X$, when $I\subset J$. Then we can take the intersection of $\A_I(X, \mathcal{F})$ and $\A_J(X, \mathcal{F})$ for $I, J\subset \{0, 1, \ldots, \dim X\}$ in $\A_K(X, \mathcal{F})$ for any $K\subset \{0, 1, \ldots, \dim X\}$ such that $I\cup J\subset K$. 

In \cite{Os1} D.V. Osipov gave a generalization of the notion of {\it restricted adeles} associated with a flag of embedded ample Cartier divisors, first introduced by A.N. Parshin for the case of a surface in \cite{Pa3}. D.V. Osipov also proved analogues of the embedding property \eqref{eqIntrEmbed} and of the intersection property \eqref{eqParshinIntersectionAdeles} for the case of restricted adeles on an arbitrary Cohen--Macaulay variety and a locally free sheaf on it. This result allowed him to generalize the Krichever correspondence to the higher-dimensional case.

\subsection*{Main results}\label{sect1.1}
In this article we investigate the question posed by A.N. Parshin. We will first show that adelic groups on thickenings of integral excellent schemes of special type are embedded in each other, when $I\subset J$. See section \ref{sect2} for the notation used below.
\begin{thmintr}[Theorem \ref{MainThmembeddingAdeles}]\label{thm1}
        Let $Y$ be an excellent integral Noetherian scheme and let $X$ be a closed subscheme which is locally defined by primary ideals. 
        Assume that $Y$ is strongly biequidimensional.
        Let $\mathcal{F}$ be a subsheaf of some flat quasicoherent sheaf on $X$ and let $I\subset J$ be subsets of $\{0, 1, \ldots, \dim X\}$. Then the map
        $$\varphi_{IJ}\colon\A_I(X, \mathcal{F})\longrightarrow\A_{J}(X, \mathcal{F})$$
        is an embedding.
    \end{thmintr}
In particular, this theorem holds for strongly biequidimensional excellent integral Noetherian schemes, since we can set $X = Y$ in theorem \ref{thm1}. By remarks \ref{remkFinTypeScheme} and \ref{remZFinTypeScheme}, equidimensional schemes of finite type over a field and irreducible flat schemes of finite type over $\mathrm{Spec}\: \Z$ are strongly biequidimensional.
    
Since adelic groups are embedded into the product of its local factors, it is sufficient to show the embedding on local factors. We have the following result.
\begin{thmintr}[Theorem \ref{TheoremEmbeddingofAdelicGroups}]
        Let $Y$ be an excellent integral Noetherian scheme and let $X$ be a closed subscheme which is locally defined by primary ideals. Assume that $Y$ is strongly biequidimensional.
        Let $\mathcal{F}$ be a flat quasicoherent sheaf on $X$ and let $I, J\subset \{0, 1, \ldots, \dim X\}$ be non-empty subsets such that $I\subset J$. Let $\Delta \in S(X, I)$. Then there is an inclusion
        $$\A_\Delta(X, \mathcal{F})\longhookrightarrow\prod_{\substack{\Delta'\in S(X, J)\\ \Delta\subset \Delta'}}\A_{\Delta'}(X, \mathcal{F}).$$
    \end{thmintr}
In turn, this theorem can be reduced to the case of affine $X$. Consequently, the question on the embedding of adelic local factors reduces to commutative algebra. For adelic local factors on an affine scheme there is the following result.
\begin{thmintr}[Theorem \ref{LocalFactorEmbedding}]
        Let $R$ be an excellent biequidimensional domain and let $I, J\subset \{0, 1, \ldots, \dim R\}$ be subsets such that $I\subset J$. Let $\Delta \in S(R, I)$. Let $\mathfrak{n}$ be a primary ideal in $R$ such that $\mathrm{ht}\: \mathfrak{n}\leq j$ if $J = (j, \ldots)\neq \varnothing$, and $\mathfrak{n} \geq  \Delta$. Then 
        $$C_\Delta R\longhookrightarrow \prod_{\substack{\Delta'\in S(R, J)\\ \Delta\subset \Delta'}}C_{\Delta'}R, \quad C_\Delta R\cap \prod_{\substack{\Delta'\in S(R, J)\\ \Delta\subset \Delta'}}\mathfrak{n}C_{\Delta'}R = \mathfrak{n}C_\Delta R.$$
    \end{thmintr}

In \cite[Theorem 1]{BG} a positive answer was given to Parshin's question on a regular (not necessarily projective) surface in the case $I \cap J= I\setminus 0$. We generalize this result for the case of a normal excellent scheme and a flat quasicoherent sheaf on it.
\begin{thmintr}[Theorem \ref{TheoremIcapJI0}]
    Let $X$ be a normal integral excellent Noetherian scheme, let $I, J\subset \{0, 1, \ldots, \dim X\}$ be subsets with $I\cap J = I\setminus 0$, and let $\mathcal{F}$ be a flat quasicoherent sheaf on $X$. 
    Assume that $X$ is strongly biequidimensional.
    Then
    $$\A_I(X, \mathcal{F})\cap \A_J(X, \mathcal{F}) = \A_{I\setminus 0}(X, \mathcal{F}).$$
\end{thmintr}

In section \ref{sect2.2} we introduce different index categories of closed subschemes of a scheme $X$. In particular, we define index categories $\mathscr{P}$ and $\mathscr{S}$ of power and symbolic thickenings and define index categories $\mathscr{R}^{pow}$, $\mathscr{R}^{sym}$, and $\mathscr{E}$. Objects in $\mathscr{E}$ are locally equidimensional closed subschemes such that all irreducible components have the same codimension in $X$. We also show the cofinality of $\mathscr{P}$ and $\mathscr{S}$ on a regular scheme.

In \cite[Proof of Theorem 1(iii)]{BG} it is shown that the limit of restrictions of global sections of a locally free sheaf to closed subschemes is naturally isomorphic to the group of global sections of this sheaf on a regular surface. We generalize this result for the case of an arbitrary Cohen--Macaulay projective scheme and a locally free sheaf on it.
\begin{thmintr}[Theorem \ref{Hjlim}]\label{intrLimitofCohomology}
    Let $X$ be an irreducible Cohen--Macaulay projective scheme over a field $\Bbbk$, let $\mathcal{F}$ be a locally free sheaf on $X$, and let $\mathscr{C}$ be some index category of closed subschemes of $X$. Then there is a natural map
    $$H^k(X, \mathcal{F})\overset{\iota}{\longrightarrow} \lim_{Z\in \mathscr{C}}H^k(Z, \mathcal{F}|_Z).$$
    Then $\iota$ is an isomorphism, if one of the following is satisfied:
    \vspace{0.1cm}

    $1)$ $0\leq k < \dim X - 1$ and $\mathscr{C} = \mathscr{E}$,
    \vspace{0.1cm}

    $2)$ $k = 0$ and $\mathscr{C} = \mathscr{P}_{i, j}$ for $0\leq i < j \leq \dim X$,
    \vspace{0.1cm}

    $3)$ $X$ is regular, $k = 0$ and $\mathscr{C} = \mathscr{R}^{pow}_{i, j}$ for $0\leq i < j \leq \dim X$.
\end{thmintr}
From this theorem we show that $\A_{\dim X - 1}(X,\mathcal{F})\cap \A_{\dim X}(X,\mathcal{F})=H^0(X, \mathcal{F})$ for a Cohen--Macaulay variety $X$ and a locally free sheaf on it (see proposition \ref{IntersectiondimX}).

Since adelic groups can be represented as a limit of restrictions to power thickenings, one can generalize \cite[Theorem 1(iii)]{BG} using theorem \ref{intrLimitofCohomology} for the case of a normal projective surface.
\begin{thmintr}[Theorem \ref{ThmDimX2}]
    Let $X$ be a normal projective surface over a field $\Bbbk$. Let $\mathcal{F}$ be a locally free sheaf on $X$ and let $I, J\subset \{0, 1, 2\}$. Then there is an equality
    $$\A_I(X, \mathcal{F})\cap \A_J(X, \mathcal{F}) = \A_{I\cap J}(X, \mathcal{F}).$$
\end{thmintr}
Also, considering restrictions of adelic groups to closed subschemes, one can compute cohomology groups of a curtailed adelic complex, which we introduce in definition \ref{defCurtailedAdelicComplex}. Roughly speaking, a curtailed adelic complex is a factor of the adelic complex $\A^{\sbullet}_{\mathrm{red}}(X, \mathcal{F})$ by subgroups $\A_I(X, \mathcal{F})$ with $J\subset I$ for some fixed $J\subset \{0, 1, \ldots, \dim X\}$. 
\begin{propintr}[Proposition \ref{Complwithout0}]
    Let $X$ be an irreducible Cohen--Macaulay projective scheme over a countable field $\Bbbk$ and let $\mathcal{F}$ be a locally free sheaf on $X$. Consider a curtailed adelic complex
    $$\mathcal{C}(\mathcal{F})^{\sbullet}\colon \A_1(X, \mathcal{F})\oplus \ldots\oplus\A_{\dim X}(X, \mathcal{F})\longrightarrow \ldots\longrightarrow \A_{(1,2,\ldots, \dim X)}(X, \mathcal{F}),$$
    where 
    $$\mathcal{C}(\mathcal{F})^m =  \bigoplus_{\substack{I\subset \{1, \ldots, \dim X\}\\ |I| = m + 1}}\A_I(X, \mathcal{F}).$$
    Then $H^k\big(\mathcal{C}(\mathcal{F})^{\sbullet}\big) \cong H^k(X, \mathcal{F})$ for $0\leq k < \dim X - 1$. In particular, $H^k\big(\mathcal{C}(\mathcal{F}(n))^{\sbullet}\big) = 0$ for $0 < k < \dim X - 1$ and $n\gg 0$.
\end{propintr}

We show that on a regular projective variety $X$ there is an equality $\A_1(X, \mathcal{F})\cap \A_j(X,\mathcal{F}) = H^0(X, \mathcal{F})$ for any locally free sheaf $\mathcal{F}$ on $X$ and any $2\leq j \leq \dim X$ (see proposition \ref{A1capAj}). 
Hence considering the curtailed adelic complex on a regular three-dimensional variety, we obtain a theorem on the intersection of adelic groups.
\begin{thmintr}[Theorem \ref{ThmDimX3}]
    Let $X$ be a regular three-dimensional projective variety over a countable field $\Bbbk$, let $\mathcal{F}$ be a locally free sheaf on $X$, and let $I, J\subset \{0, 1, 2, 3\}$. Then there is an equality
    $$\A_I(X, \mathcal{F})\cap\A_J(X, \mathcal{F}) = \A_{I\cap J}(X, \mathcal{F}).$$
\end{thmintr}

\vspace{0.25cm}

\subsection*{Acknowledgments}\label{sect1.3}
The author thanks his advisor Denis V. Osipov for his advice and attention to this work. The author is also grateful to Vladislav Levashev for interesting discussions and useful remarks. The author was supported by the Theoretical Physics and Mathematics Advancement Foundation ``BASIS'' under grants no. 23-8-2-14-1 and no. 24-8-2-11-1.

\section{Preliminaries}\label{sect2}
In this section we introduce notions and construct objects which we will use in the text. Throughout the paper, we assume that all rings are commutative with identity. We denote by the sign $\subset$ a non-strict inclusion and by the sign $\subsetneq$ a strict inclusion. 

By a {\it variety} over a field $\Bbbk$ we mean a separated integral scheme of finite type over $\Bbbk$. A {\it surface} is a two-dimensional variety. By a quasicoherent sheaf on a scheme $X$ we mean a quasicoherent sheaf of $\mathcal{O}_X$-modules. All locally free sheaves on schemes are assumed to be of finite rank. If $Y\overset{j}{\longrightarrow} X$ is an open or closed embedding, then denote $j^*\mathcal{F}$ by $\mathcal{F}|_Y$ for any quasicoherent sheaf $\mathcal{F}$ on $X$.

\subsection{Commutative algebra}\label{sect2.1} 
For a ring $R$ and an ideal $\mathfrak{a}$ denote by $\sqrt{\mathfrak{a}}$ the radical of $\mathfrak{a}$, i.e. the set $\{x\in R\mid \exists n >~0\colon x^n\in \mathfrak{a}\}$. Denote by $\mathrm{rad}(R)$ the {\it Jackobson radical} of $R$, i.e. the intersection of maximal ideals of $R$. Recall that $$\mathrm{ht}\:\mathfrak{a} = \inf\limits_{\mathfrak{a}\subset \mathfrak{p}}\:\mathrm{ht}\:\mathfrak{p}$$
for a proper ideal $\mathfrak{a}\subset R$, where infimum is taken over all prime ideals of $R$ containing $\mathfrak{a}$. 

\begin{df}
    Let $R$ be a ring. An ideal $\mathfrak{n}$ is called {\it primary}, if all zero divisors in the ring $R/\mathfrak{n}$ are nilpotent. 
\end{df}

Note that for a primary ideal $\mathfrak{n}\subset R$ its radical $\sqrt{\mathfrak{n}}$ is a prime ideal. For a prime ideal $\mathfrak{p}$ we will call the ideal $\mathfrak{n}$ a $\mathfrak{p}$-primary ideal, if $\sqrt{\mathfrak{n}} = \mathfrak{p}$.

\begin{df}
Let $R$ be a ring and let $\mathfrak{p}$ be a prime ideal. For a submodule $N$ of an $R$-module $M$ set
    $$S_{\mathfrak{p}}(N) = \ker\big(M\longrightarrow S^{-1}_{\mathfrak{p}}(M/N)\big).$$
\end{df}
We do not include the module $M$ in the notation $S_{\mathfrak{p}}(N)$, since it will be clear from the context or will be mentioned explicitly.

We will denote the localization by a prime ideal $S^{-1}_\mathfrak{p}M$ also as $M_\mathfrak{p}$.

\begin{df}\label{DefSymbolicPower}
    Let $R$ be a Noetherian ring and let $\mathfrak{a}$ be a proper ideal. Then the ideal
    $$\mathfrak{a}^{(n)}  = \bigcap_{\mathfrak{p}\in \mathrm{ Ass}_R(R/\mathfrak{a})}S_{\mathfrak{p}}(\mathfrak{a}^n)$$
    is called the {\it $n$-th symbolic power} of the ideal $\mathfrak{a}$. We assume that $\mathfrak{a}^{(0)} = R$ and $R^{(n)} = R$ for any $n\geq 0$.
\end{df}
Note that for a prime ideal $\mathfrak{p}$ of $R$ the ideal $\mathfrak{p}^{(n)} = S_{\mathfrak{p}}(\mathfrak{p}^n)$ is primary with $\sqrt{\mathfrak{p}^{(n)}} = \mathfrak{p}$. In fact, $\mathfrak{p}^{(n)}$ is a $\mathfrak{p}$-primary part of the primary decomposition of the ideal $\mathfrak{p}^n$, see \cite[Theorem 6.8]{Ma1}. We obviously have inclusions $\mathfrak{a}^{(n+1)}\subset \mathfrak{a}^{(n)}$, $\mathfrak{a}^n\subset \mathfrak{a}^{(n)}$ for an ideal $\mathfrak{a}$ and any $n\geq 1$. Note that $\mathfrak{a}^{(1)} = \mathfrak{a}$.

Recall that embedded primes of a finitely generated module $M$ over a Noetherian ring $R$ are non-minimal elements of $\mathrm{Ass}_R(M)$. If $\mathfrak{a}$ is an ideal of $R$, then embedded primes of $\mathfrak{a}$ are embedded primes of the $R$-module $R/\mathfrak{a}$.

\begin{lemma}\label{LocalizationSymbPower}
    Let $R$ be a Noetherian ring and let $\mathfrak{a}$ be an ideal in $R$ without embedded primes. Let $T$ be a multiplicatively closed subset of $R$ and let $S = T^{-1}R$. Then for any $n > 0$ there is an equality 
    $$\mathfrak{a}^{(n)}S = (\mathfrak{a}S)^{(n)}$$
    of ideals in $S$.
\end{lemma}
\begin{proof}By $\mathrm{Ass}(M)$ for an $R$-module $M$ we mean $\mathrm{Ass}_R(M)$.
    By defintion, 
    $$\mathfrak{a}^{(n)} = \bigcap_{\mathfrak{p}\in \mathrm{Ass}(R/\mathfrak{a})}S_{\mathfrak{p}}(\mathfrak{a}^n).$$
    Since $S_{\mathfrak{p}}(\mathfrak{a}^n) = \ker\big(R\longrightarrow R_{\mathfrak{p}}/\mathfrak{a}^nR_\mathfrak{p}\big)$ and $R\rightarrow S$ is a flat morphism, then $S_{\mathfrak{p}}(\mathfrak{a}^n) S= \ker\big(S\longrightarrow S_\mathfrak{p}/\mathfrak{a}^nS_\mathfrak{p}\big)$.

    Let $\mathfrak{p}\in \mathrm{Ass}(R/\mathfrak{a})$. If $\mathfrak{p}\cap T = \varnothing$, then $S_{\mathfrak{p}S} = S_\mathfrak{p} = R_\mathfrak{p}$ by \cite[p. 24, Corollary 4]{Ma1}. Thus 
    $$S_{\mathfrak{p}}(\mathfrak{a}^n) S= \ker(S\longrightarrow S_{\mathfrak{p}S}/\mathfrak{a}^nS_{\mathfrak{p}S}) = S_{\mathfrak{p}S}(\mathfrak{a}^nS).$$
    If $\mathfrak{p}\cap T\neq \varnothing$, then $\mathfrak{p}R_\mathfrak{p}\cap \tilde{T}\neq \varnothing$, where $\tilde{T}$ is the image of $T$ in $R_\mathfrak{p}$. We have $\mathrm{Ass}_{R_\mathfrak{p}}(R_\mathfrak{p}/\mathfrak{a}R_\mathfrak{p}) = \{\mathfrak{p}R_\mathfrak{p}\}$ by \cite[Theorem 6.2]{Ma1}, since $\mathfrak{a}$ has no embedded primes, in particular, $\mathfrak{p}R_\mathfrak{p}$ is the unique minimal prime ideal over $\mathfrak{a}R_\mathfrak{p}$. Therefore $\sqrt{\mathfrak{a}R_\mathfrak{p}} = \mathfrak{p}R_\mathfrak{p}$. It follows that if $x\in \mathfrak{p}R_\mathfrak{p}\cap \tilde{T}$, then $x^m\in \mathfrak{a}R_\mathfrak{p}\cap \tilde{T}$ for some $m > 0$, since $\tilde{T}$ is multiplicatively closed. Thus $\mathfrak{a}S_\mathfrak{p} = \mathfrak{a}\tilde{T}^{-1}R_\mathfrak{p} = \tilde{T}^{-1}R_\mathfrak{p} = S_\mathfrak{p}.$
    Therefore $S_\mathfrak{p}/\mathfrak{a}^nS_\mathfrak{p} = 0$, whence $S_{\mathfrak{p}}(\mathfrak{a}^n) S= \ker\big(S\longrightarrow 0\big) = S.$
    From these observations we obtain
    $$\mathfrak{a}^{(n)}S = \bigcap_{\mathfrak{p}\in \mathrm{Ass}(R/\mathfrak{a})}S_{\mathfrak{p}}(\mathfrak{a}^n)S = 
    \bigcap_{\substack{\mathfrak{p}\in \mathrm{Ass}(R/\mathfrak{a})\\\mathfrak{p}\cap T = \varnothing}}S_{\mathfrak{p}}(\mathfrak{a}^n)S =
    \bigcap_{\substack{\mathfrak{p}\in \mathrm{Ass}(R/\mathfrak{a})\\\mathfrak{p}\cap T = \varnothing}}S_{\mathfrak{p}S}(\mathfrak{a}^nS) = (\mathfrak{a}S)^{(n)},$$
    where the last equality follows from the fact that $\mathrm{Ass}_S(S/\mathfrak{a}S) = \{\mathfrak{p}S\mid \mathfrak{p}\in \mathrm{Ass}(R/\mathfrak{a})\colon \mathfrak{p}\cap T = \varnothing\}$ by \cite[Theorem 6.2]{Ma1}.
\end{proof} 

\begin{rem}\label{remRadIdealNoEmbPrimes}
    If $\mathfrak{a}\subset R$ is a radical ideal, i.e. $\sqrt{\mathfrak{a}} = \mathfrak{a}$, then $\mathfrak{a}^{(n)}T^{-1}R = (\mathfrak{a}T^{-1}R)^{(n)}$ for any $n > 0$ and any multiplicative subset $T\subset R$, since radical ideals have no embedded primes.
\end{rem}

\begin{df}
    Let $R$ be a ring. A {\it flag of prime ideals} $\Delta$ is a non-empty set of prime ideals $\{\mathfrak{p}_0, \ldots, \mathfrak{p}_n\}$ of the ring $R$ for $n \geq 0$ such that $\mathfrak{p}_i\subsetneq \mathfrak{p}_{i + 1}$ for $0\leq i < n$. We will denote this by $\Delta = (\mathfrak{p}_0, \ldots, \mathfrak{p}_n)$.
\end{df}

  Let us remark that we require a strict inclusion of prime ideals in the definition of a flag, since we will consider this case for the most part of the paper. However, a more general definition of a simplex will be given below, see section \ref{sect2.3}.

\begin{df}
    Let $R\overset{\varphi}{\longrightarrow} S$ be a morphism of rings. For an ideal $\mathfrak{a}\subset S$ denote by $\mathfrak{a}\cap R$ the ideal $\varphi^{-1}(\mathfrak{a})$. The ideal $\mathfrak{a}\cap R$ is called a {\it contraction} of $\mathfrak{a}$ to $R$.

    For a flag $\Delta = (\mathfrak{p}_0, \ldots, \mathfrak{p}_n)$ of prime ideals of $S$ denote $\Delta \cap R = (\mathfrak{p}_0\cap R, \ldots, \mathfrak{p}_n\cap R)$ if $\mathfrak{p}_i\cap R\neq \mathfrak{p}_{i + 1}\cap R$ for $0\leq i < n$. If for a flag $\Gamma$ of prime ideals of $R$ we have $\Delta \cap R = \Gamma$, then denote this by $\Delta \mid \Gamma$.
\end{df}

\begin{df}
    Let $R$ be a ring. For a prime ideal $\mathfrak{p}$ and an $R$-module $M$ denote
    $$C_\mathfrak{p}M = \varprojlim_{n}M/\mathfrak{p}^nM.$$
    If $\Delta = (\mathfrak{p}_0, \ldots, \mathfrak{p}_n)$ is a flag of prime ideals of the ring $R$, then denote
    $$C_\Delta M = C_{\mathfrak{p}_0}S^{-1}_{\mathfrak{p}_0}\ldots C_{\mathfrak{p}_n}S^{-1}_{\mathfrak{p}_n}M.$$
    Also set $C_\varnothing M = M$. 
\end{df}
\begin{rem}
By definition, $\Delta = \varnothing$ is not a flag, but we will use this notation further in the text for convenience. 
\end{rem}
\begin{rem}
    Let us remark the difference between $C_{\mathfrak{p}}M$ and $C_{(\mathfrak{p})}M$. The first module is just completion $C_{\mathfrak{p}}M = \varprojlim_n M/\mathfrak{p}^n M$, and the second one is completion after localization $C_{(\mathfrak{p})}M = \varprojlim_n S^{-1}_{\mathfrak{p}}M/\mathfrak{p}^n S^{-1}_{\mathfrak{p}}M$.
\end{rem}

\begin{df}
    Let $R$ be a ring. For an ideal $\mathfrak{a}$ and a flag of prime ideals $\Delta = (\mathfrak{p}_0,\ldots, \mathfrak{p}_n)$ let $\mathfrak{a} \geq \Delta$ (resp. $\mathfrak{a} > \Delta$),  if $\mathfrak{a} \subset \mathfrak{p}_0$ (resp. $\mathfrak{a} \subsetneq  \mathfrak{p}_0$).

    If $\Delta = (\mathfrak{p}_0,\ldots, \mathfrak{p}_n)$, $\Delta' = (\mathfrak{p}'_0,\ldots, \mathfrak{p}'_n)$ is a pair of flags of prime ideals, then let $\Delta' \geq \Delta$ (resp. $\Delta' > \Delta$), if $\mathfrak{p}_n' \geq \Delta$ (resp. $\mathfrak{p}_n' > \Delta$). 

    Let us also assume that $\Delta > \varnothing$ for any flag of prime ideals $\Delta$.
\end{df}

\begin{df}
    Let $R$ be a ring. For a pair of flags of prime ideals $\Delta = (\mathfrak{p}_0,\ldots, \mathfrak{p}_n)$, $\Delta' = (\mathfrak{p}'_0,\ldots, \mathfrak{p}'_n)$ such that $\Delta' > \Delta$ define
    $$\Delta'\vee \Delta = (\mathfrak{p}_0',\ldots, \mathfrak{p}_n', \mathfrak{p}_0,\ldots, \mathfrak{p}_n).$$

    For a prime ideal $\mathfrak{p}$ such that $\mathfrak{p} > \Delta$ denote $(\mathfrak{p})\vee \Delta$ as $\mathfrak{p}\vee \Delta$.
Let us also assume that $\Delta \vee \varnothing = \Delta$.
\end{df}

\begin{df}
    A ring $R$ is called {\it catenary}, if for any pair of prime ideals $\mathfrak{p}, \mathfrak{q}$ such that $\mathfrak{p}\subset \mathfrak{q}$ every two maximal strictly increasing chains of prime ideals between $\mathfrak{p}$ and $\mathfrak{q}$ have the same finite length.  

    A ring $R$ is called {\it universally catenary}, if any finitely generated algebra over $R$ is catenary.
\end{df}

\begin{df}
    A ring $R$ is called {\it excellent}, if $R$ is a universally catenary $G$-ring and $J$-$2$ ring. By definition, $R$ is Noetherian.
\end{df}

We do not give the full definition of an excellent ring (including $G$-ring and $J$-$2$ ring) here, because of its complexity. See \cite[p. 260]{Ma1}, \cite[Section 34]{Ma2} for details. We will only need the properties of excellent rings such as analytic normality (see \cite[Theorem 79]{Ma2}). Examples of excellent rings are algebras of finite type over a field, over $\Z$, or over complete semi-local Noetherian rings.

\begin{df}
    Let $R$ be a ring. A chain $\mathfrak{p}_1\subsetneq \mathfrak{p}_2\subsetneq \ldots \subsetneq \mathfrak{p}_n$ of prime ideals of $R$ is called {\it saturated}, if there are no prime ideals between $\mathfrak{p}_i$ and $\mathfrak{p}_{i + 1}$ for all $1\leq i\leq n - 1$.
\end{df}

\begin{df}
    $1)$ A Noetherian ring $R$ is called {\it equidimensional}, if $\dim R/\mathfrak{p} = \dim R$ for every minimal prime ideal $\mathfrak{p}$ of $R$. 

    

    $2)$ A Noetherian ring $R$ is called {\it biequidimensional}, if $\dim R < \infty$ and all maximal chains of prime ideals have the same length $\dim R$. 

    $3)$ A Noetherian ring $R$ is called {\it locally equidimensional}, if the ring $R_\mathfrak{p}$ is equidimensional for every prime ideal $\mathfrak{p}$. 
\end{df}

\begin{rem}\label{remFiniteTypeEquidim}
    Note that every finitely generated algebra over a field $\Bbbk$ without zero divisors is  biequidimensional by \cite[Corollary 14.5]{Ma2}. Thus every equidimensional finitely generated algebra over a field is biequidimensional.
\end{rem}

\begin{rem}\label{remZFinType}
Let $R$ be a $\Z$-algebra of finity type. Assume that $\mathrm{Spec}\: R$ is irreducible and all non-empty fibers over $\mathrm{Spec}\: \Z$ are equidimensional and of equal dimension. Since every closed point of $\mathrm{Spec}\: R$ lies in a fiber over a closed point of $\mathrm{Spec}\: \Z$, one can see from the catenarity of $R$ that $R$ is biequidimensional.
\end{rem}

\begin{rem}\label{remBiequidim}
    Note that the biequidimensionality of $R$ implies that for any prime ideal $\mathfrak{p}$ of $R$ and an integer $k$ with $\mathrm{ht}\:\mathfrak{p}\leq k \leq \dim R$ there exists a prime ideal $\mathfrak{q}$ of height $k$ containing $\mathfrak{p}$. 
\end{rem}

It is a well-known fact that a tensor product with a flat module commutes with finite intersections. The next definition, given in \cite[p. 41]{HH} (see also \cite[p. 75]{HJ}), is a generalization of the flatness property for the case of infinite intersections.

\begin{df}
    Let $R$ be a ring and let $S$ be an $R$-module. Then $S$ is {\it intersection flat}, if $S$ is flat over $R$ and for any family of submodules $\{M_\lambda\}_{\lambda\in \Lambda}$ of any finitely generated $R$-module $M$ we have 
    $$\big(\bigcap_{\lambda\in \Lambda}M_\lambda\big)\otimes_R S = \bigcap_{\lambda\in \Lambda}(M_\lambda\otimes_R S),$$
    where both sides are identified with its images in $M\otimes_R S$.

    If $R\rightarrow S$ is a ring homomorphism, we say that $S$ is {\it intersection flat over} $R$ if $S$ is an intersection flat $R$-module.
\end{df}

\begin{lemma}\label{LemmaABCD}
    Let $A$, $B$, $C$ be subgroups in an abelian group $D$. Then there is an equality
    $$\dfrac{A}{A\cap C}\cap \dfrac{B}{B\cap C} = \dfrac{A\cap(B + C)}{A\cap C},$$
    of subgroups in $\dfrac{D}{C}$
\end{lemma}
\begin{proof}
    The proof is straightforward.
\end{proof}

\begin{lemma}\label{LemmaCompLocal}
    Let $R$ be a ring, and let $\mathfrak{a}$ be an ideal. For an $R/\mathfrak{a}$-module $M$ and a prime ideal $\mathfrak{p}$ such that $\mathfrak{a}\subset \mathfrak{p}$ there are equalities of $R/\mathfrak{a}$-modules
    $$C_\mathfrak{p}M = C_{\tilde{\mathfrak{p}}}M, \quad S^{-1}_\mathfrak{p}M = S^{-1}_{\tilde{\mathfrak{p}}}M,$$
    where $\tilde{\mathfrak{p}} = \mathfrak{p}\:\mathrm{mod}\: \mathfrak{a}$ and the completion and localization on the left side (resp. the right side) are considered as the completion and localization of $R$-modules (resp. $R/\mathfrak{a}$-modules).
\end{lemma}
\begin{proof}
 The proof is straightforward.
\end{proof}

\subsection{Notions on schemes} In this section we introduce some notions for schemes, which we will use in the text.

\begin{df}
    Let $X$ be a scheme. For a point $x\in X$ define its {\it codimension} 
    $$\mathrm{codim}_Xx = \dim \mathcal{O}_{X, x}.$$
    If $Z$ is a closed subscheme of $X$, then define $$\mathrm{codim}_XZ = \inf_{z\in Z}(\mathrm{codim}_X z).$$
    If $X$ is clear from the context, then denote $\mathrm{codim}_Xx$
    as $\mathrm{codim}\:x$ and denote $\mathrm{codim}_XZ$ as $\mathrm{codim}\: Z$.
\end{df}
Note that if $X$ is Noetherian, $\mathrm{codim}\: x$ is finite.

\begin{df}
    A scheme $X$ is called {\it catenary}, if for any $x\in X$ the local ring $\mathcal{O}_{X, x}$ is catenary.
\end{df}
\begin{df}
    A Noetherian scheme $X$ is called {\it locally equidimensional}, if for any $x\in X$ the local ring $\mathcal{O}_{X, x}$ is equidimensional.
\end{df}
\begin{df}
    A scheme $X$ is called {\it excellent}, if there is a cover $X = \bigcup_{i\in I}\mathrm{Spec}\: R_i$ by affine open subschemes, where each $R_i$ is an excellent ring.
\end{df}

Note that an integral Noetherian scheme is locally equidimensional, and an excellent scheme is catenary and locally Noetherian, but not Noetherian in general.

\begin{df} 
$1)$ A Noetherian affine scheme $X = \mathrm{Spec}\: R$ is called {\it  biequidimensional}, if $R$ is a  biequidimensional ring.

$2)$ A Noetherian scheme $X$ is called {\it strongly biequidimensional}, if $\dim X < \infty$ and every point of $X$ has a biequidimensional affine neighborhood of dimension $\dim X$.
\end{df}

\begin{rem}
    It is easy to see that a strongly biequidimensional scheme $X$ is biequidimensional in the sense of \cite[Definition 1.2]{He}, since every maximal chain of points of $X$ (ordered in the sense of definition \ref{defS(X)}) is contained in a biequidimensional affine scheme of dimension $\dim X$. We do not know an example of a biequidimensional, but not strongly biequidimensional scheme.
\end{rem}

\begin{rem}\label{remkFinTypeScheme}
    By remark \ref{remFiniteTypeEquidim}, every equidimensional scheme of finite type over a field is strongly biequidimensional.
\end{rem}

\begin{rem}\label{remZFinTypeScheme}
    By remark \ref{remZFinType}, every irreducible scheme of finite type over $\mathrm{Spec}\: \Z$ such that all non-empty fibers are equidimensional and of equal dimension is strongly biequidimensional. In particular, irreducible flat schemes of finite type over $\mathrm{Spec}\: \Z$ are strongly biequidimensional (see \cite[Theorem 14.116]{GW}). 
\end{rem}

\begin{rem}\label{remIrrSubschisBiequidim}
    Note that an equidimensional closed subscheme $X$ of a strongly biequidimensional scheme $Y$ is strongly biequidimensional of dimension $\dim Y - \mathrm{codim}\: X$. 
\end{rem}

\begin{df}
    Let $X$ be a scheme and let $\mathcal{J}$ be an ideal sheaf on $X$. Then we say that $\mathcal{J}$ is a {\it locally radical } ideal sheaf, if $\mathcal{J}_x$ is a radical ideal in $\mathcal{O}_{X, x}$ for any $x\in X$.
\end{df}

\begin{rem}
    Note that an ideal sheaf of an integral closed subscheme is locally radical.
\end{rem}

\begin{lemma}\label{lemmaRadicalSheaf}
    Let $X$ be a scheme and let $\mathcal{J}$ be an ideal sheaf on $X$. Assume that $\mathcal{J}$ is locally radical. Then for any affine open subset $U\subset X$ the ideal $\mathcal{J}(U)$ is radical in the ring $\mathcal{O}_X(U)$.
\end{lemma}
\begin{proof}
    The proof is straightforward.
\end{proof}

\begin{df}
    Let $X$ be a closed subscheme of a scheme $Y$ given by the ideal sheaf $\mathcal{J}$. Then we say that $X$ is {\it locally defined by primary ideals}, if $X$ is irreducible and $\mathcal{J}_y$ is a primary ideal in $\mathcal{O}_{Y, y}$ for any $y\in Y$.
\end{df}

\begin{rem}
    If $X$ is locally defined by primary ideals in $Y$, then $X$ is locally defined by primary ideals in itself.
\end{rem}

\begin{lemma}\label{lemmaPrimarySheaf}
    Let $Y$ be a Noetherian scheme and let $X$ be a closed subscheme given by the ideal sheaf $\mathcal{J}$. Assume that $X$ is locally defined by primary ideals in $Y$. Then for any affine open subset $U\subset Y$ the ideal $\mathcal{J}(U)$ is primary in the ring $\mathcal{O}_Y(U)$.
\end{lemma}
\begin{proof}
    Let $R = \mathcal{O}_Y(U)$, $\mathfrak{a} = \mathcal{J}(U)$. Since $X$ is irreducible, there is a unique minimal prime ideal $\mathfrak{p}$ over $\mathfrak{a}$.  Let $\mathfrak{n} = S_\mathfrak{p}(\mathfrak{a})$ be a $\mathfrak{p}$-primary part in the primary decomposition of $\mathfrak{a}$. By our assumptions, localizations of $\mathfrak{a}$ and $\mathfrak{n}$ at prime ideals of $R$ coincide. Hence $\mathfrak{a} = \mathfrak{n}$.
\end{proof}

\subsection{Thickenings and index categories of subschemes}\label{sect2.2}
\begin{df}
    Let $X$ be a closed subscheme of a sheme $Y$ defined by the ideal sheaf $\mathcal{J}_X$. A closed subscheme $\tilde{X}$ of $Y$ is called a {\it thickening of $X$ in $Y$} if there are inclusions of ideal sheaves $\mathcal{J}_X^n\subset \mathcal{J}_{\tilde{X}}\subset \mathcal{J}_X$ for some $n\geq 1$. 

    A subscheme of $Y$ defined by the ideal sheaf $\mathcal{J}_X^n$ is called the {\it $n$-th power thickening of $X$}. We will denote it by $X_{[n]}$.
\end{df}

Let $\mathcal{J}$ be an ideal sheaf on a Noetherian scheme $X$. For every point $x\in X$ and for every integer $n\geq 1$ there is the $n$-th symbolic power $\big(\mathcal{J}_x\big)^{(n)}$ of the ideal $\mathcal{J}_x$ in $\mathcal{O}_{X, x}$. For a point $x\in X$ define the ideal sheaf $\mathcal{J}_{n, x}$ on $X$ by the formula
$$\mathcal{J}_{n, x}(U) = \begin{cases}
    \mathcal{O}_X(U), \quad x\notin U, \\
    \big(\mathcal{J}_x\big)^{(n)}\cap \mathcal{O}_X(U), \quad x\in U
\end{cases}$$
on every open subset $U\subset X$. The stalk $(\mathcal{J}_{n, x})_x $ at $x$ is equal to $\big(\mathcal{J}_x\big)^{(n)}$ by definition. 
\begin{df}
    Let $\mathcal{J}$ be an ideal sheaf on a Noetherian scheme $X$. The ideal sheaf
    $$\mathcal{J}^{(n)} = \bigcap_{x\in X}\mathcal{J}_{n, x}$$
    is called the {\it $n$-th symbolic power of $\mathcal{J}$}.
\end{df}

Denote by $\mathcal{J}_x^{(n)}$ the stalk of the sheaf $\mathcal{J}^{(n)}$ at a point $x\in X$.

\begin{prop}\label{LemmaSymbIdealSheafProperties}
    Let $X$ be a Noetherian scheme and let $\mathcal{J}$ be an ideal sheaf on $X$. Assume that $\mathcal{J}$ is locally radical.


    $1)$ If $U$ is a non-empty open subset of $X$, then $\mathcal{J}^{(n)}(U) = \bigcap_{x\in U}\mathcal{J}_{n, x}(U)$.
    
    $2)$ For a point $x\in X$ we have $\mathcal{J}^{(n)}_x = \big(\mathcal{J}_x\big)^{(n)}$.

    $3)$ If $U$ is an affine open subscheme of $X$, then $\mathcal{J}^{(n)}(U) = \mathcal{J}(U)^{(n)}$.

    $4)$ For any $n\geq 0$ there are inclusions of ideal sheaves $\mathcal{J}^{(n + 1)}\subset \mathcal{J}^{(n)}$ and $\mathcal{J}^n\subset \mathcal{J}^{(n)}$. Moreover, $\mathcal{J}^{(0)} = \mathcal{O}_X$ and $\mathcal{J}^{(1)} = \mathcal{J}$.
\end{prop}
\begin{proof}
    By definition, 
    $$\mathcal{J}^{(n)}(U) = \bigcap_{x\in X}\mathcal{J}_{n, x}(U)$$
    on every open subset $U$. Since $\mathcal{J}_{n, x}(U) = \mathcal{O}_X(U)$ for any $x\notin U$, item $1$ is proved. 

    To prove item $2$, we can assume $X = \mathrm{Spec}\:R$ is affine. Then $\mathcal{J}$ is associated with the radical ideal $\mathfrak{a} = \mathcal{J}(X)$ in $R$ by lemma \ref{lemmaRadicalSheaf}. By lemma \ref{LocalizationSymbPower} and remark \ref{remRadIdealNoEmbPrimes}, 
    $\mathfrak{a}^{(n)}\mathcal{O}_{X, x} = (\mathfrak{a}\mathcal{O}_{X, x})^{(n)} = \big(\mathcal{J}_x\big)^{(n)}$ for any $x\in X$. Thus $\mathfrak{a}^{(n)}\subset \bigcap_{x\in X}\mathcal{J}_{n, x}(X) = \mathcal{J}^{(n)}(X)$, from what follows
    $$\big(\mathcal{J}_x\big)^{(n)} = \mathfrak{a}^{(n)}\mathcal{O}_{X, x}\subset \mathcal{J}_x^{(n)}\subset \big(\mathcal{J}_x\big)^{(n)},$$
    where the last inclusion follows from the definition of $\mathcal{J}^{(n)}$. 

    Let us prove item $3$. We can assume that $X = U$. From the proof of item $2$ we know that $\mathcal{J}(X)^{(n)}\subset \mathcal{J}^{(n)}(X)$ and $\mathcal{J}(X)^{(n)}\mathcal{O}_{X, x} = \big(\mathcal{J}_x\big)^{(n)} = \mathcal{J}^{(n)}_x$ for any point $x\in X$. Since stalks at every point are equal, the ideals are also equal. This proves item $3$.

    Inclusions of ideal sheaves from item $4$ are clear from the definition of $\mathcal{J}^{(n)}$. The equality $\mathcal{J}^{(0)} = \mathcal{O}_X$ is also clear. Since $\big(\mathcal{J}_x\big)^{(1)} = \mathcal{J}_x$, we obtain the inclusion of ideal sheaves $\mathcal{J}\subset \mathcal{J}^{(1)}$. By item $2$, the stalks of these sheaves are equal at every point, from what follows the equality of sheaves. 
\end{proof}

\begin{df}
    Let $X$ be a closed subscheme of a Noetherian sheme $Y$ defined by the ideal sheaf $\mathcal{J}_X$. Assume that $\mathcal{J}_X$ is locally radical.
    A subscheme of $Y$ defined by the ideal sheaf $\mathcal{J}_X^{(n)}$ is called the {\it $n$-th symbolic  thickening of $X$}. We will denote it by $X_{(n)}$.
\end{df}

\begin{rem}
    Note that if $X$ is an integral closed subscheme of a Noetherian scheme $Y$, then $X_{(n)}$ is locally defined by primary ideals. This follows from the remark after definition \ref{DefSymbolicPower} and from proposition \ref{LemmaSymbIdealSheafProperties}.
\end{rem}

Note that $X_{(n)}$ is a thickening of $X$ by item $4$ of proposition \ref{LemmaSymbIdealSheafProperties}. Also, it follows from this proposition that there exists a closed embedding $X_{(n)}\hookrightarrow X_{[n]}$.

It is natural to ask whether for $n\geq 1$ there exist $m\geq 1$ and a closed embedding $X_{[n]}\hookrightarrow X_{(m)}$. In general, this is not true, but this holds for regular Noetherian schemes.

\begin{prop}\label{PropCofinalSymbPow}
    Let $Y$ be a regular Noetherian scheme. Let $X$ be a closed subscheme of $Y$ defined by the ideal sheaf $\mathcal{J}$. Assume that $\mathcal{J}$ is locally radical. Then for any $n\geq 1$ there exist $m\geq 1$ and a natural closed embedding $X_{[n]}\hookrightarrow X_{(m)}$. 
\end{prop}
\begin{proof}
    
    Let $Y = \bigcup_{i\in I}U_i$ be a finite covering by open affine subschemes $U_i = \mathrm{Spec}\:R_i$ with each $R_i$ being regular and Noetherian. By \cite[Theorem A]{Mu}, for every $i\in I$ there exists $m_i\geq 1$ such that $\mathcal{J}(U_i)^{(m_i)}\subset \mathcal{J}(U_i)^n$. Let $m = \max_{i\in I} m_i$. Then $\mathcal{J}^{(m)}_y\subset \mathcal{J}^n_y$ for any point $y\in Y$. Therefore there is an embedding of ideal sheaves $\mathcal{J}^{(m)}\subset \mathcal{J}^n$, from what the statement follows.
\end{proof}

\begin{df}
    Let $X$ be a Noetherian scheme and let $0\leq i_1 < i_2 < \ldots < i_n\leq \dim X$ be a sequence of integers.
    The {\it index category of power thickenings}, denoted by $\mathscr{P}_{i_1, \ldots, i_n}(X)$, is defined as follows. Objects in $\mathscr{P}_{i_1, \ldots, i_n}(X)$ are power thickenings of integral closed subschemes of $X$ of codimensions $i_1, \ldots, i_n$, and morphisms are closed embeddings.

    The {\it index category of symbolic thickenings}, denoted by $\mathscr{S}_{i_1, \ldots, i_n}(X)$, is defined analogously, with objects being symbolic thickenings.

    If the scheme $X$ is clear from the context, we will omit $X$ in the notation.
\end{df}

\begin{df}\label{defIndCatE}
Let $X$ be a Noetherian finite-dimensional scheme.
Define a category $\mathscr{E}$ in the following way way. An object in $\mathscr{E}$ is a closed locally equidimensional subscheme $Z$ such that $\dim Z < \dim X$ and all irreducible components of $Z$ have the same codimension in $X$. Morphisms in $\mathscr{E}$ are closed embeddings. 
\end{df}

Note that if $X$ is an equidimensional scheme of finite type over a field $\Bbbk$, then $X$ is locally equidimensional by \cite[14.H, Corollary 3]{Ma2}. Every irreducible finite-dimensional Noetherian scheme is locally equidimensional, and thus $\mathscr{P}_i \subset \mathscr{E}$ for any $i$.

\begin{rem}
    For an index category $\mathscr{C}$ of closed subschemes of $X$ we will denote $Z\in \mathrm{Ob}(\mathscr{C})$ as $Z\in \mathscr{C}$.
\end{rem}

For an index category $\mathscr{C}$ of closed subschemes of a Notherian scheme $X$ consider a diagram $F\colon \mathscr{C}^{op}\rightarrow \mathrm{\bf Ab}$ to the category of abelian groups. Then the limit of this diagram can be identified with the group of elements $(a_Z)_{Z\in \mathscr{C}}\in \prod_{Z\in \mathscr{C}}F(Z)$ such that $F(f)(a_Z) = a_{Z'}$ for any morphism $Z'\xrightarrow{f} Z$ in $\mathscr{C}$. We will denote this limit by $\lim_{Z\in \mathscr{C}}F(Z)$.

\begin{example}
    If $\mathcal{F}$ is a quasicoherent sheaf of $\mathcal{O}_X$-modules, then $H^0(-, \mathcal{F}|_{-})$ and $\A_{I|_{-}}(-, \mathcal{F}|_{-})$ (see sections \ref{sect2.3} and \ref{sect2.5}) are $\mathscr{C}^{op}$-shaped diagrams. 
\end{example}

\begin{lemma}\label{LemmaEqualityLimitSymbPow}
Let $X$ be a regular Noetherian scheme and let $\mathcal{F}$ be a quasicoherent sheaf of $\mathcal{O}_X$-modules. Let $i_1, i_2, \ldots, i_n$ be non-negative integers with $i_1 < i_2 < \ldots < i_n\leq \dim X$. Let $F = H^0(-, \mathcal{F}|_{-})$ or $F = \A_{I|_{-}}(-, \mathcal{F}|_{-})$ (see sections \ref{sect2.3} and \ref{sect2.5}). Then
$$\lim_{Z\in \mathscr{P}_{i_1, \ldots, i_n}}F(Z) = \lim_{Z\in \mathscr{S}_{i_1, \ldots, i_n}}F(Z).$$
\end{lemma}
\begin{proof}
    For an integral closed subscheme $Z$ of $X$ of codimension $i_j$ for some $1\leq j \leq n$ we have $Z_{(m)}\hookrightarrow Z_{[m]}$ for any $m\geq 1$. By proposition \ref{PropCofinalSymbPow}, there exist $k\geq 1$ and a closed embedding $Z_{[m]}\hookrightarrow Z_{(k)}$. Thus the systems $\{Z_{[m]}\}_{m = 1}^\infty$ and $\{Z_{(m)}\}_{m = 1}^\infty$ are cofinal, from what the equality of limits follows.
\end{proof}



\begin{df}
    Let $X$ be a regular Noetherian scheme. Let $i, j$ be non-negative integers such that $i < j \leq \dim X$. Then define an index category $\mathscr{R}_{i, j}^{pow}(X)$ (resp. $\mathscr{R}_{i, j}^{sym}(X))$ in the following way. Objects in $\mathscr{R}_{i, j}^{pow}(X)$ (resp. $\mathscr{R}_{i, j}^{sym}(X)$) are either power (resp. symbolic) thickenings of regular integral closed subschemes of codimension $i$ in $X$ or power (resp. symbolic) thickenings of integral closed subschemes of codimension $j$ in $X$ which are embedded in regular integral closed subschemes of codimension $i$. Morphisms in $\mathscr{R}_{i, j}^{pow}(X)$ (resp. $\mathscr{R}_{i, j}^{sym}(X))$ are closed embeddings. If the scheme $X$ is clear from the context, we will omit $X$ in the notation.
\end{df}
\begin{rem}\label{RemarkSymPow}
    If $Z\in \mathscr{R}_{i, j}^{pow}$ is of codimension $j$ in $X$, then there exists a power thickening $Y$ of a regular integral closed subscheme of codimension $i$ in $X$, i.e. $Y\in \mathscr{R}_{i, j}^{pow}$, containing $Z$. The same holds for $\mathscr{R}_{i, j}^{sym}$, but we need to use the cofinality of symbolic and power thickenings of subschemes on a regular Noetherian scheme, see proposition \ref{PropCofinalSymbPow}.
\end{rem}

\begin{lemma}\label{EqualityRpowRsym}
    Let $X$ be a regular Noetherian scheme and let $\mathcal{F}$ be a quasicoherent sheaf of $\mathcal{O}_X$-modules. Let $i, j$ be non-negative integers with $i < j \leq \dim X$. Let $F = H^0(-, \mathcal{F}|_{-})$ or $F = \A_{I|_{-}}(-, \mathcal{F}|_{-})$ (see sections \ref{sect2.3} and \ref{sect2.5}). Then
$$\lim_{Z\in \mathscr{R}_{i, j}^{pow}}F(Z) = \lim_{Z\in \mathscr{R}_{i, j}^{sym}}F(Z).$$
\end{lemma}
\begin{proof}
    The proof is analogous to the proof of lemma \ref{LemmaEqualityLimitSymbPow}.
\end{proof}

\subsection{Parshin--Beilinson adeles on Noetherian schemes}\label{sect2.3}

In this section we recall some results on Parshin--Beilinson adeles. For details, see \cite{Be}, \cite{Hu}, \cite{Mo}, \cite{Os2}, \cite{Pa1}, \cite{Pa2}. 

\begin{df}\label{defS(X)}
Let $X$ be a Noetherian scheme. Let $p, q\in X$ be two points. Define partial order on the set of points of $X$
$$p \leq q \iff p\in \overline{\{q\}}.$$
Then we can define a simplicial set $S(X)$, where
$$S(X)_n = \{(p_0, \ldots, p_n)\mid p_j\in X,\: p_j\geq p_{j + 1}\}$$
is the set of $n$-simplices of $S(X)$ with usual boundary maps $\delta_i^n$ and degeneracy maps $\sigma_i^n$ for $n\in \mathbb{N}$, $0\leq i\leq n$.
\end{df}

For a subset $K\subset S(X)_n$ and a point $p\in X$ denote
$$\leftidx{_p}K = \{(p_1,\ldots, p_n)\in S(X)_{n - 1}\mid (p, p_1, \ldots, p_n)\in K\}.$$

Let $\text{\bf QCoh}(X)$ and $\text{\bf Coh}(X)$ be the categories of quasicoherent and coherent sheaves on $X$ respectively. Let $\text{\bf Ab}$ be the category of abelian groups. 
For any $K\subset S(X)_n$ the functor $\A(K, -)\colon \mathrm{\bf QCoh}(X)\longrightarrow \mathrm{\bf Ab}$ is defined, see \cite[Proposition 2.1.1]{Hu}, \cite{Be}. This functor is exact and commutes with filtered colimits.
Define the group of $n$-dimensional adeles with coefficients in a quasicoherent sheaf $\mathcal{F}$ as
$$\A^n(X, \mathcal{F}) = \A(S(X)_n, \mathcal{F}).$$
Let the local factor of $\A^n(X, \mathcal{F})$ in a flag $\Delta\in S(X)_n$ be
$$\A_\Delta(X, \mathcal{F}) = \A(\{\Delta\}, \mathcal{F}).$$
There is a natural inclusion of an adelic group in the product of its local factors, see \cite[Proposition 2.1.4]{Hu}.

\begin{prop}\label{LocalCondAdeleSubsheaf}
    Let $X$ be a Noetherian scheme and let $\mathcal{G}$ and $ \mathcal{F}$ be quasicoherent sheaves on $X$ such that $\mathcal{G}\subset \mathcal{F}$. Let $n\in \mathbb{Z}_{\geq 0}$ and let $K\subset S(X)_n$. Then there is an equality  
    $$\A(K, \mathcal{G}) = \A(K, \mathcal{F})\cap \prod_{\Delta\in K}\A_{\Delta}(X, \mathcal{G}),$$ where the intersection is taken in $\prod_{\Delta\in K}\A_{\Delta}(X, \mathcal{F})$.
\end{prop}
\begin{proof}
    It is easy to see that the proposition is equivalent to the fact that the functor 
    $$\prod_{\Delta\in K}\A_{\Delta}(X, -)\bigg/\A(K, -)\colon \mathrm{\bf QCoh}(X)\longrightarrow \mathrm{\bf Ab}$$
    is exact. The latter assertion follows from the fact that the quotient of two exact functors is exact.
\end{proof}

There are natural maps $d_i^{n + 1}(K, L, \mathcal{F})$ between adelic groups, called {\it boundary maps}, see \cite[Definition 2.2.1]{Hu}, \cite[Definition 2.2.2]{Hu}, and \cite{Be}. For $K = S(X)_{n + 1}$, $L = S(X)_n$ we call $d_i^{n + 1}(K, L, \mathcal{F})$ the global boundary map
$$d_i^{n + 1}\colon \A^{n}(X, \mathcal{F})\longrightarrow\A^{n + 1}(X, \mathcal{F}).$$
For $K = \{\Delta\}$, $L = \{\Delta' = \delta_i^{n + 1}\}$ we call $d_i^{n + 1}(K, L, \mathcal{F})$ the local boundary map
$$d_i^{n + 1}\colon \A_{\Delta'}(X, \mathcal{F})\longrightarrow\A_{\Delta}(X, \mathcal{F}).$$
For $K\subset S(X)_{n + 1}$, $L\subset S(X)_n$ with $\delta_{i}^{n + 1}K\subset L$ we define 
\begin{eqnarray*}D_i^{n + 1}(K, L, \mathcal{F})\colon \prod_{\Delta\in L}\A_\Delta(X, \mathcal{F})\longrightarrow \prod_{\Delta\in K}\A_\Delta(X, \mathcal{F}),\quad (x_\Delta)_{\Delta\in L}\longmapsto (y_\Delta)_{\Delta\in K}
\end{eqnarray*}
by $y_\Delta = d_i^{n + 1}(x_{\delta_i^{n + 1}\Delta})$. In fact, maps $d_i^{n + 1}$ are restrictions of maps $D_i^{n + 1}$, see \cite[Proposition 2.2.4]{Hu}.

There is a description of the local structure of adelic groups and boundary maps in terms of rings $C_\Delta R$, see \cite[Proposition 3.2.1]{Hu}, \cite[Proposition 3.2.2]{Hu}, and \cite[\S 3.2]{Os2}. Namely, if $X = \mathrm{Spec}\:R$ and $\mathcal{F}$ is a quasicoherent sheaf on $X$, then $\A_\Delta(X, \mathcal{F}) = C_\Delta R\otimes_R \mathcal{F}(X)$.
Note that for flags $\Delta$ and $ \Gamma$ of prime ideals of a Noetherian ring $R$ such that $\Delta \subset \Gamma$ local boundary maps induce a natural map
$$C_\Delta R\longrightarrow C_\Gamma R.$$

Define $S(X)_n^{\mathrm{red}}$ as a subset in $S(X)_n$ of non-degenerate $n$-dimensional simplices. Define the group of reduced $n$-dimensional adeles on a Noetherian scheme $X$ with coefficients in a quasicoherent sheaf $\mathcal{F}$
$$\A_{\mathrm{red}}^n(X, \mathcal{F}) = \A(S(X)_n^{\mathrm{red}}, \mathcal{F}).$$
In fact, adelic groups $\A^n(X,\mathcal{F})$ can be expressed in terms of reduced adelic groups, see \cite[Proposition 3.3.3]{Hu}.

Let $X$ be a Noetherian scheme. For a simplex $\Delta = (p_0, p_1, \ldots, p_n)\in S(X)_n$, let
$$\text{codim}\:\Delta = (\text{codim}\: p_0,\text{codim}\: p_1, \ldots, \text{codim}\: p_n).$$

\begin{df}\label{defSI}
    Let $I\subset \{0, 1, \ldots, \dim X\}$ be a non-empty finite subset. Define 
    $$S(X, I) = \{\Delta \in S(X)_{|I| - 1}\mid \mathrm{ codim}\: \Delta = I\},$$
    where by $|I|$ we denote the cardinality of $I$. 
    Define 
    $$\A_I(X, \mathcal{F}) = \A(S(X, I), \mathcal{F}).$$
    Also define $S(X, \varnothing) = \varnothing$ and $\A_{\varnothing}(X, \mathcal{F}) = H^0(X, \mathcal{F})$.
\end{df}

Note that $\A\big(S(X, \varnothing), \mathcal{F}\big) = 0 \neq \A_{\varnothing}(X, \mathcal{F})$.

\begin{df}\label{definitionPhiIJ}
    Let $X$ be a Noetherian scheme. Let $I, J\subset \{0, 1, \ldots, \dim X\}$ be finite subsets such that $I\subset J$. Then there is a canonical map
    $$\varphi_{IJ}\colon \A_I(X, \mathcal{F})\longrightarrow \A_J(X, \mathcal{F})$$ 
    induced by boundary maps.
\end{df}

\begin{df}
    For $X = \mathrm{Spec}\: R$, where $R$ is a Noetherian ring, let $S(R, I)$ denote $S(X, I)$ for a finite subset $I\subset \{0, 1, \ldots, \dim X\}$.
\end{df}

\begin{df}
    For $I = (i_0, \ldots, i_n)\subset \{0, 1, \dim X\}$ and an integer $k$ denote $I\pm k = (i_0\pm k,\ldots, i_n\pm k)$.
\end{df}

If $I = (i)$, we will denote $\A_{(i)}(X, \mathcal{F})$ as $\A_i(X, \mathcal{F})$.

\begin{prop}\label{Decomposition}
    Let $X$ be a finite-dimensional Noetherian scheme and let $\mathcal{F}$ be a quasicoherent sheaf on $X$. Then for $n\in \Z_{\geq 0}$ there is a decomposition
    $$\A^n_{\mathrm{red}}(X, \mathcal{F}) = \bigoplus\limits_{\substack{I\subset \{0, 1, \ldots, \dim X\}\\ |I| = n + 1}}\A_I(X, \mathcal{F}).$$
\end{prop}
\begin{proof}
    This proposition follows from \cite[Lemma 3.3.1]{Hu}, since $S(X)_n^{\mathrm{red}}= \bigsqcup\limits_{\substack{I\subset \{0, 1, \ldots, \dim X\}\\ |I| = n + 1}}S(X, I)$.
\end{proof}


The proof of the following proposition is similar to the proof of \cite[Proposition 1.5]{Os1}.

\begin{prop}\label{AffineOxtimesF}
    Let $X$ be an affine Noetherian scheme and let $\mathcal{F}$ be a quasicoherent sheaf on $X$. Let $I\subset \{0, 1, \ldots, \dim X\}$ be a finite subset. Then there is a natural isomorphism
    $$\A_I(X, \mathcal{O}_X)\otimes_{\mathcal{O}_X(X)}\mathcal{F}(X)\overset{\sim}{\longrightarrow}\A_I(X, \mathcal{F}).$$
\end{prop}



\begin{lemma}\label{IandI0}
    Let $X$ be an irreducible Noetherian scheme with the generic point $\eta$, let $\mathcal{F}$ be a quasicoherent sheaf on $X$, and let $I\subset \{0, 1, \ldots, \dim X\}$ be a finite subset with $0\in I$. Then 
    $$\A_I(X, \mathcal{F}) = \A_{I\setminus 0}(X, \mathcal{F}_\eta).$$
\end{lemma}
\begin{proof}
Since localizations commute with direct limits, we can assume that $\mathcal{F}$ is a coherent sheaf on $X$.
    By \cite[Proposition 2.1.1]{Hu},
    $$\A_I(X, \mathcal{F}) = \varprojlim_l\: \A(S(I\setminus 0), \mathcal{F}_\eta/\mathfrak{m}_\eta^l\mathcal{F}_\eta),$$
    and since $\mathfrak{m}_\eta$ is nilpotent as the maximal ideal in the local Artinian ring $\mathcal{O}_{X, \eta}$, we have
    $$\A_I(X, \mathcal{F}) = \A_{I\setminus 0}(X, \mathcal{F}_\eta).$$
\end{proof}

\subsection{Adelic complexes}\label{sect2.4}

Let $X$ be a Noetherian scheme and let $\mathcal{F}$ be a quasicoherent sheaf on $X$. Define for $n\in \Z_{\geq 0}$
$$d^{n + 1}= \sum_{j = 0}^{n + 1}(-1)^jd_j^{n + 1}\colon \A^n(X, \mathcal{F})\longrightarrow \A^{n + 1}(X, \mathcal{F}).$$
The complex
$$0\longrightarrow \A^0(X, \mathcal{F})\overset{d^1}{\longrightarrow} \A^1(X, \mathcal{F})\overset{d^2}{\longrightarrow} \A^2(X, \mathcal{F})\overset{d^3}{\longrightarrow} \ldots$$
is called the {\it adelic complex}. The complex of reduced adeles $\A_{\mathrm{red}}^{\sbullet}(X, \mathcal{F})$ is a direct summand of the adelic complex.

\begin{thm}[{\cite[Proposition 5.1.3]{Hu}}]\label{CohomologyRedAdelicComp}
    Let $X$ be a Noetherian scheme and let $\mathcal{F}$ be a quasicoherent sheaf on $X$. There are isomorphisms for all $i\geq 0$
    $$H^i(\A_{\mathrm{red}}^{\sbullet}(X, \mathcal{F})) \cong H^i(\A^{\sbullet}(X, \mathcal{F})) \cong H^i(X, \mathcal{F}).$$
\end{thm}

\begin{df}\label{defCurtailedAdelicComplex}
    Let $X$ be a finite-dimensional Noetherian scheme and let $\mathcal{F}$ be a quasicoherent sheaf on $X$. Let $J$ be a non-empty subset of $\{0, 1, \ldots, \dim X\}$. We call the complex
    $$\bigoplus_{i\in \{0, 1, \ldots, \dim X\}\setminus J}\A_i(X, \mathcal{F})\rightarrow \bigoplus_{\substack{i, j\in \{0, 1, \ldots, \dim X\}\setminus J\\ i < j}}\A_{(i, j)}(X, \mathcal{F})\rightarrow \ldots \rightarrow  \bigoplus_{\substack{I\subset  \{0, 1, \ldots, \dim X\}\setminus J\\ |I| = n + 1}}\A_I(X, \mathcal{F})\rightarrow \ldots$$
    a {\it curtailed adelic complex}. It is a quotient complex of $\A_{\mathrm{red}}^{\sbullet}(X, \mathcal{F})$ which depends on $J$.
\end{df}

In section \ref{sect6.1} we compute cohomology groups of a special case of a curtailed adelic complex.

\subsection{Inverse image maps on adelic groups}\label{sect2.5}

In this section we construct a pullback morphism on adelic groups. Its construction on local adelic factors can be found in \cite[p. 178]{Pa2}. 

Note that a morphisn $f\colon X\rightarrow Y$ of Noetherian schemes induces a morphism $f\colon S(X)_n\rightarrow ~S(Y)_n$ for each $n\geq 0$.
\begin{df}\label{PullbackDef}
    Let $f\colon X\rightarrow Y$ be a morphism of Noetherian schemes and let $\mathcal{F}$ be a quasicoherent sheaf on $Y$. Let $n\geq 0$ and $K\subset S(X)_n$, $L\subset S(Y)_n$ such that $f(K)\subset L$. Then there exists a pullback morphism
    $$f^*(K, L, \mathcal{F})\colon \A(L, \mathcal{F})\longrightarrow \A(K, f^*\mathcal{F}),$$ 
    which is defined by the following properties.
    
    $1)$ $f^*(K, L, -)$ is a natural transformation of functors and commutes with direct limits.

    $2)$ If $\mathcal{F}$ is coherent on $Y$ and $n = 0$, then $f^*\mathcal{F}$ is coherent on $X$ and $$\A(L, \mathcal{F}) = \prod_{y\in L}\mathcal{F}_y\otimes_{\mathcal{O}_{Y, y}}\widehat{\mathcal{O}}_{Y, y},\quad \A(K, f^*\mathcal{F}) = \prod_{x\in K}(f^*\mathcal{F})_x\otimes_{\mathcal{O}_{X, x}}\widehat{\mathcal{O}}_{X, x}.$$ 
    If $f(x) = y$ for some $x\in K$, $y\in L$, then there is the local morphism of local rings $\mathcal{O}_{Y, y}\rightarrow\mathcal{O}_{X, x}$ which induces morphism of their completions $\widehat{\mathcal{O}}_{Y, y}\rightarrow\widehat{\mathcal{O}}_{X, x}$. Since $(f^*\mathcal{F})_x \cong \mathcal{F}_y\otimes_{\mathcal{O}_{Y, y}}\mathcal{O}_{X, x}$, we obtain a natural morphism $\mathcal{F}_y\otimes_{\mathcal{O}_{Y, y}}\widehat{\mathcal{O}}_{Y, y}\rightarrow \mathcal{F}_y\otimes_{\mathcal{O}_{Y, y}}\widehat{\mathcal{O}}_{X, x} = (f^*\mathcal{F})_x\otimes_{\mathcal{O}_{X, x}}\widehat{\mathcal{O}}_{X, x}$. Then define $f^*(K, L, \mathcal{F})$ as a composition 
    $$\A(L, \mathcal{F})\longrightarrow \A(f(K), \mathcal{F}) = \prod_{y\in f(K)}\mathcal{F}_y\otimes_{\mathcal{O}_{Y, y}}\widehat{\mathcal{O}}_{Y, y}\longrightarrow \prod_{y\in f(K)}\prod_{\substack{x\in K\\ f(x) = y}}(f^*\mathcal{F})_x\otimes_{\mathcal{O}_{X, x}}\widehat{\mathcal{O}}_{X, x} = \A(K, f^*\mathcal{F}).$$

    $3)$ If $\mathcal{F}$ is coherent on $Y$ and $n > 0$, then $f^*\mathcal{F}$ is coherent on $X$ and $$\A(L, \mathcal{F}) = \prod_{y\in Y}\varprojlim_l\A(\leftidx{_y}{}L, \mathcal{F}_y/\mathfrak{m}^l_y\mathcal{F}_y),\quad \A(K, f^*\mathcal{F}) = \prod_{x\in X}\varprojlim_l\A(\leftidx{_x}{}K, (f^*\mathcal{F})_x/\mathfrak{m}^l_x(f^*\mathcal{F})_x).$$
    If $f(x) = y$ for some $x\in K$, $y\in L$, then there is the 
    morphism 
    $$f^*(\leftidx{_x}{}K, \leftidx{_y}{}L, \mathcal{F}_y/\mathfrak{m}^l_y\mathcal{F}_y)\colon \A(\leftidx{_y}{}L, \mathcal{F}_y/\mathfrak{m}^l_y\mathcal{F}_y)\longrightarrow \A\left(\leftidx{_x}{}K, f^*\left(\mathcal{F}_y/\mathfrak{m}^l_y\mathcal{F}_y\right)\right).$$
    The local morphism $\mathcal{O}_{Y, y}\rightarrow\mathcal{O}_{X, x}$ induces a commutative diagram
    \begin{center}
        \begin{tikzcd}
            \mathrm{Spec}\left(\mathcal{O}_{X, x}/\mathfrak{m}_x^l\right)\arrow{r}\arrow{d} & X\arrow{d}\\
             \mathrm{Spec}\left(\mathcal{O}_{Y, y}/\mathfrak{m}_y^l\right)\arrow{r} & Y,
        \end{tikzcd}
    \end{center}
    which gives rise to a natural morphism of sheaves $f^*\left(\mathcal{F}_y/\mathfrak{m}^l_y\mathcal{F}_y\right)\rightarrow (f^*\mathcal{F})_x/\mathfrak{m}_x^l(f^*\mathcal{F})_x$ on $X$, which comes from the natural morphism of 
   $\mathcal{O}_{Y, y}/\mathfrak{m}_y^l$-modules
   $$\mathcal{F}_y/\mathfrak{m}^l_y\mathcal{F}_y\longrightarrow \left(\mathcal{F}_y\otimes_{\mathcal{O}_{Y, y}}\mathcal{O}_{X, x}\right)/\mathfrak{m}^l_x\left(\mathcal{F}_y\otimes_{\mathcal{O}_{Y, y}}\mathcal{O}_{X, x}\right)\cong (f^*\mathcal{F})_x/\mathfrak{m}^l_x(f^*\mathcal{F})_x$$
    Thus we obtain a composition morphism 
    $$\A(\leftidx{_y}{}L, \mathcal{F}_y/\mathfrak{m}^l_y\mathcal{F}_y)\longrightarrow \A\left(\leftidx{_x}{}K, f^*\left(\mathcal{F}_y/\mathfrak{m}^l_y\mathcal{F}_y\right)\right)\longrightarrow\A(\leftidx{_x}{}K, (f^*\mathcal{F})_x/\mathfrak{m}_x^l(f^*\mathcal{F})_x).$$
    Then define $f^*(K, L, \mathcal{F})$ as a composition
    $$\A(L, \mathcal{F}) = \prod_{y\in Y}\varprojlim_l\A(\leftidx{_y}{}L, \mathcal{F}_y/\mathfrak{m}^l_y\mathcal{F}_y)\longrightarrow \prod_{y\in Y}\prod_{\substack{x\in X\\ f(x) = y}}\varprojlim_l\A(\leftidx{_x}{}K, (f^*\mathcal{F})_x/\mathfrak{m}^l_x(f^*\mathcal{F})_x) = \A(K, f^*\mathcal{F}).$$
\end{df}

\begin{rem}
    Since $\A(K, -)$ is exact on $\mathrm{\bf QCoh}(X)$ and $f^*$ is right exact on $\mathrm{\bf QCoh}(Y)$, we see that $\A(K,f^*-)$ is right exact on $\mathrm{\bf QCoh}(Y)$. Therefore $f^*(K, L, -)$ transforms short exact sequences of adelic groups corresponding to short exact sequences of quasicoherent sheaves on $Y$ to right exact sequences of adelic groups on $X$.
\end{rem}

\begin{rem}\label{remGenPullback}
    We can define $f^*$ for arbitrary $K$ and $L$ without the assumption $f(K)\subset L$. Indeed, define for arbitrary $K\subset S(X)_n$, $L\subset S(Y)_n$ the set $M = f^{-1}(f(K)\cap L)\cap K\subset K$ and then define $f^*(K, L, \mathcal{F})$ as a composition
    $$\A(L, \mathcal{F})\longrightarrow \A(M, f^*\mathcal{F})\longrightarrow \A(K, f^*\mathcal{F}),$$
    where the first morphism is equal to $f^*(M, L, \mathcal{F})$ defined in definition \ref{PullbackDef} and the second morphism is from \cite[Proposition 2.1.5]{Hu}.
\end{rem}


If sets $K$ and $L$ are clear from the context, we will denote $f^*(K, L, \mathcal{F})$ as $f^*$.

If $K = \{\Delta\}$, $L = \{f(\Delta)\}$, then there is a morphism 
$$f^*\colon \A_{f(\Delta)}(Y, \mathcal{F})\longrightarrow \A_{\Delta}(X, f^*\mathcal{F}),$$
which agrees with the definition from \cite[p. 178]{Pa2}. Therefore if $K\subset S(X)_n$, $L\subset S(Y)_n$ are arbitrary subsets with $f(K)\subset L$, then there is a morphism
$$f^*\colon \prod_{\Delta'\in L}\A_{\Delta'} (Y, \mathcal{F})\longrightarrow\prod_{\Delta\in K}\A_\Delta(X, f^*\mathcal{F}),\quad (y_{\Delta'})\mapsto (x_\Delta),$$
where $x_\Delta = f^*y_{f(\Delta)}$.


\begin{prop}\label{commdiagpullback}
    Let $f\colon X\rightarrow Y$ be a morphism of Noetherian schemes and let $\mathcal{F}$ be a quasicoherent sheaf on $Y$. Let $n\geq 0$ and let $K\subset S(X)_n$, $L\subset S(Y)_n$ such that $f(K)\subset L$. Then the diagram
    \begin{center}
        \begin{tikzcd}
            \A(L, \mathcal{F})\arrow{r}{f^*}\arrow{d}&\A(K, f^*\mathcal{F})\arrow{d}\\
            \prod\limits_{\Delta'\in L}\A_{\Delta'} (Y, \mathcal{F})\arrow{r}{f^*}& \prod\limits_{\Delta\in K}\A_\Delta(X, f^*\mathcal{F})
        \end{tikzcd}
    \end{center}
    commutes.
\end{prop}
\begin{proof}
   This follows from definition \ref{PullbackDef} by induction on $n$.
\end{proof}

\begin{prop}
    Let $X$ and $Y$ be Noetherian schemes and let $\mathcal{F}$ be a quasicoherent sheaf on $Y$. Then the pullback map
    $$f^*\colon \A^{\sbullet}(Y, \mathcal{F})\longrightarrow\A^{\sbullet}(X, f^*\mathcal{F})$$
    is a morphism of complexes. Therefore it induces the map on cohomology groups for every $i\geq 0$
    $$f^*\colon H^i(Y, \mathcal{F}) \cong H^i(\A^{\sbullet}(Y, \mathcal{F}))\longrightarrow H^i(\A^{\sbullet}(X, f^*\mathcal{F}))\cong H^i(X, f^*\mathcal{F}),$$
    and this map coincides with the ordinary pullback map between cohomology groups of sheaves.
\end{prop}
\begin{proof}
    By propositions \ref{commdiagpullback} and \cite[Proposition 2.2.4]{Hu}, in order to prove that $f^*$ is a morphism of complexes, it suffices to consider only the case $K = \{\Delta\}\subset S(X)_{n + 1}$, $L = \{\Delta' = f(\Delta)\}\subset S(Y)_{n + 1}$ for some integer $n\geq 0$ and one boundary map $d_i^{n + 1}$ for $0\leq i\leq n + 1$, and this case is clear.

    Therefore $f^*$ induces morphism on cohomology groups 
    $$f^*\colon H^i(Y, \mathcal{F})\longrightarrow H^i(X, f^*\mathcal{F}).$$
    This map is equal to the ordinary pullback on cohomology groups by \cite[Scolium 5.2]{Iv}, since the diagram
    \begin{center}
        \begin{tikzcd}
            \mathcal{F}\arrow{r}\arrow{d} & \underline{\A}^{\sbullet}(Y, \mathcal{F})\arrow{d} \\
            f_*f^*\mathcal{F}\arrow{r} & f_*\underline{\A}^{\sbullet}(X, f^*\mathcal{F}).
        \end{tikzcd}
    \end{center}
    is commutative, where the right vertical morphism of sheaves is equal to $$\A^{\sbullet}(U, \mathcal{F}|_U)\overset{f^*}{\longrightarrow}\A^{\sbullet}(f^{-1}(U), f^*\mathcal{F}|_{f^{-1}(U)})$$ on every open subset $U\subset Y$ (these are acyclic resolutions of the sheaves $\mathcal{F}$ and $f^*\mathcal{F}$ by \cite[Theorem 4.2.3]{Hu}).
\end{proof}

 Let $X$ be a finite-dimensional catenary locally equidimensional Noetherian scheme and let $Z$ be a locally equidimensional closed subscheme in $X$ such that all irreducible components of $Z$ have the same codimension in $X$. 
 Denote $Z\overset{\iota}{\hookrightarrow} X$ the closed embedding. Since $X$ is catenary, $Z$ is also catenary. Let $I \subset \{0, 1, \ldots, \dim X\}$. 
Denote
 $$I|_Z = I - \mathrm{codim}\:Z.$$

 \begin{lemma}\label{LemmaIandIZ}
     In the above notation, $\iota\big( S(Z, I|_Z)\big) \subset S(X, I)$. 
 \end{lemma}
 \begin{proof}
 Let $(p_0, \ldots, p_n)\in S(Z, I|_Z)$. By definition, $\dim \mathcal{O}_{Z, p_j} = i_j - \mathrm{codim}\: Z$ for $0\leq j\leq n$. If we consider $p_j$ as a point of $X$, its codimension equals $\mathrm{codim}\: Z + \dim \mathcal{O}_{Z, p_j} = i_j$ by our assumptions on equidimensionality and catenarity of $Z$ and $X$. On the other hand, the codimension of $p_j$ in $X$ is equal to $\mathrm{codim}_Xp_j$ by definition. Hence $i_j = \mathrm{codim}_Xp_j$, and the flag $(p_0, \ldots, p_n)$, considered as an element of $S(X)_n$, lies in $S(X, I)$.
 \end{proof}

 \begin{df}
     Let $X$ be a finite-dimensional catenary locally equidimensional Noetherian scheme and let $Z$ be a locally equidimensional closed subscheme in $X$ such that all irreducible components of $Z$ have the same codimension in $X$. Let $\mathcal{F}$ be a quasicoherent sheaf on $X$. Let $I = (i_0, \ldots )\subset \{0, 1, \ldots, \dim X\}$. If $i_0 < \mathrm{codim}\: Z$, then define $\A_{I|_Z}(Z, \mathcal{F}|_Z) = 0$. Otherwise, let $\A_{I|_Z}(Z, \mathcal{F}|_Z) = \A(S(Z, I|_Z), \mathcal{F}|_Z)$.
     Define a {\it restriction map}
     $$\A_I(X, \mathcal{F})\longrightarrow \A_{I|_Z}(Z, \mathcal{F}|_Z)$$
     as the pullback map with $\iota\big( S(Z, I|_Z)\big) \subset S(X, I)$ in case $i_0\geq \mathrm{codim }\: Z$, and let it be the zero map otherwise.
 \end{df}

\section{Embeddings of adelic groups}\label{sect3}

\subsection{Embeddings of adelic local factors on affine schemes}\label{sect3.1}
In this section we examine embeddings of adelic local factors.

\begin{prop}\label{Cp0CDeltaflat}
    Let $R$ be a Noetherian ring and let $\Delta = (\mathfrak{p}_0,\ldots, \mathfrak{p}_n)$ be a flag of prime ideals of $R$. Then the natural map of rings
    $$C_{\mathfrak{p}_0}S^{-1}_{\mathfrak{p}_0}R\longrightarrow C_\Delta R$$
    is faithfully flat. In particular, there is an embedding $C_{\mathfrak{p}_0}S^{-1}_{\mathfrak{p}_0}R\otimes_R M\hookrightarrow C_\Delta R\otimes_R M$ for any $R$-module $M$.
\end{prop}
\begin{proof}
    By \cite[Proposition 3.2.1]{Hu}, $C_\Delta R$ is a 
    Noetherian $C_{\mathfrak{p}_0}S^{-1}_{\mathfrak{p}_0}R$-algebra. Since the ring $C_{\Delta}R$ is $\mathfrak{p}_0C_\Delta R$-complete, we have $\mathfrak{p}_0C_\Delta R\subset \mathrm{ rad}(C_\Delta R)$ by \cite[Theorem 8.2]{Ma1}. We also have that $C_{\mathfrak{p}_0}S^{-1}_{\mathfrak{p}_0}R/\mathfrak{p}_0C_{\mathfrak{p}_0}S^{-1}_{\mathfrak{p}_0}R$ is a field, so $C_\Delta R/\mathfrak{p}_0C_\Delta R$ is flat over $C_{\mathfrak{p}_0}S^{-1}_{\mathfrak{p}_0}R/\mathfrak{p}_0C_{\mathfrak{p}_0}S^{-1}_{\mathfrak{p}_0}R$.

    Observe that 
    \begin{eqnarray*}
        \mathfrak{p}_0C_{\mathfrak{p}_0}S^{-1}_{\mathfrak{p}_0}R\otimes_{C_{\mathfrak{p}_0}S^{-1}_{\mathfrak{p}_0}R}C_\Delta R &=& \left(\mathfrak{p}_0\otimes_R C_{\mathfrak{p}_0}S^{-1}_{\mathfrak{p}_0}R\right)\otimes_{C_{\mathfrak{p}_0}S^{-1}_{\mathfrak{p}_0}R}C_\Delta R\\
        &=& \mathfrak{p}_0\otimes_R C_\Delta R = \mathfrak{p}_0C_\Delta R,
    \end{eqnarray*}
    where the first and the third equalities follow from the fact that $C_{\mathfrak{p}_0}S^{-1}_{\mathfrak{p}_0}R$ and $C_\Delta R$ are flat $R$-algebras.
Therefore, by \cite[Theorem 22.3]{Ma1}, we obtain that $C_\Delta R$ is flat over $C_{\mathfrak{p}_0}S^{-1}_{\mathfrak{p}_0}R$. 

    For faithful flatness, it is sufficient to prove that the induced morphism on spectra is surjective. But $C_{\mathfrak{p}_0}S^{-1}_{\mathfrak{p}_0}R$ is local and $\varphi^{-1}(\mathfrak{p}_0C_\Delta R)= \mathfrak{p}_0 C_{\mathfrak{p}_0}S^{-1}_{\mathfrak{p}_0}R$, where $\varphi\colon C_{\mathfrak{p}_0}S^{-1}_{\mathfrak{p}_0}R\rightarrow C_\Delta R$. Thus, by the going-down property for flat morphisms, 
    we obtain that this morphism is surjective, from what we get the faithful flatness.

    By \cite[Theorem 7.5]{Ma1}, faithfully flat maps are universally injective, from what follows the second part of the statement.
\end{proof}

\begin{cor}
    Let $R$ be a Noetherian ring and let $\Delta = (\mathfrak{p}_0,\ldots, \mathfrak{p}_n)$ be a flag of prime ideals of $R$. Let $i\in \{0, 1, \ldots, n\}$. Then the natural map of rings
    $$C_{\mathfrak{p}_i}S^{-1}_{\mathfrak{p}_i}R\longrightarrow C_\Delta R$$
    is flat.
\end{cor}
\begin{proof}
    Let us prove this 
    by induction on the number of ideals of the flag $\Delta$ before $\mathfrak{p}_i$. By proposition \ref{Cp0CDeltaflat}, the map $C_{\mathfrak{p}_i}S^{-1}_{\mathfrak{p}_i}R\rightarrow C_{(\mathfrak{p}_i, \ldots, \mathfrak{p}_n)} R$ is flat. Assume that for some $j$ with $0 < j \leq i$ the map 
    $$C_{\mathfrak{p}_i}S^{-1}_{\mathfrak{p}_i}R\longrightarrow C_{(\mathfrak{p}_j, \ldots, \mathfrak{p}_n)} R$$
    is flat. Then $S^{-1}_{\mathfrak{p}_{j - 1}}C_{(\mathfrak{p}_j, \ldots, \mathfrak{p}_n)} R$ is flat over $C_{\mathfrak{p}_i}S^{-1}_{\mathfrak{p}_i}R$. 
    Then the $\mathfrak{p}_{j-1}$-adic completion $C_{(\mathfrak{p}_{j - 1}, \ldots, \mathfrak{p}_n)} R = C_{\mathfrak{p}_{j - 1}}S^{-1}_{\mathfrak{p}_{j - 1}}C_{(\mathfrak{p}_j, \ldots, \mathfrak{p}_n)} R$ is flat over $C_{\mathfrak{p}_i}S^{-1}_{\mathfrak{p}_i}R$ by \cite[Theorem 0.1]{Ye3}. Thus the induction hypothesis is proved, and we obtain that the map $C_{\mathfrak{p}_i}S^{-1}_{\mathfrak{p}_i}R\rightarrow C_\Delta R$ is flat. 
\end{proof}

\begin{rem}
    If $i > 0$, the map $C_{\mathfrak{p}_i}S^{-1}_{\mathfrak{p}_i}R\longrightarrow C_\Delta R$ is not faithfully flat, since $\mathfrak{p}_iC_\Delta R = (S^{-1}_{\mathfrak{p}_0}\mathfrak{p}_i)\cdot C_\Delta R = C_\Delta R$ and its contraction to $C_{\mathfrak{p}_i}S^{-1}_{\mathfrak{p}_i}R$ equals the ring itself.
\end{rem}

\begin{lemma}\label{annExactSequence}
    Let $R$ be a Noetherian ring and let $a\in R$. Let $\Delta = (\mathfrak{p}_0, \ldots, \mathfrak{p}_n)$ be a flag of prime ideals of $R$. Then there is an exact sequence of $R$-modules
    $$0\longrightarrow \mathrm{ ann}_R(a)C_\Delta R\longrightarrow C_\Delta R\overset{\cdot a}{\longrightarrow}C_\Delta R.$$
\end{lemma}
\begin{proof}
    It follows from the fact that $C_\Delta R$ is a flat $R$-algebra.
\end{proof}

\begin{prop}\label{REmbeddsInCDeltaR}
    Let $R$ be a Noetherian ring and let $\Delta = (\mathfrak{p}_0,\ldots, \mathfrak{p}_n)$ be a flag of prime ideals of $R$. Assume that $(0)$ is a primary ideal. Then $R\longhookrightarrow C_\Delta R.$
\end{prop}
\begin{proof}
    Let $\mathfrak{n} = \sqrt{(0)}$ be the nilradical of $R$. We have $R\hookrightarrow S^{-1}_\mathfrak{p}R$ for every prime ideal $\mathfrak{p}$, since every zero divisor in $R$ is nilpotent, i.e. lies in $\mathfrak{n}$. Therefore
    $$R\longhookrightarrow S^{-1}_{\mathfrak{p}_0}R\longhookrightarrow C_\Delta R,$$
    where the second inclusion follows from proposition \ref{Cp0CDeltaflat}.
\end{proof}

\begin{lemma}\label{lemmaEmbSheafPrimThik}
    Let $X$ be an irreducible Noetherian scheme locally defined by primary ideals in itself. Then for any non-empty open subsets $V\subset U$ in $X$ the restriction map $\mathcal{O}_X(U)\longrightarrow \mathcal{O}_X(V)$ is injective.
\end{lemma}
\begin{proof}
    Let $\eta$ be the generic point of $X$. It suffices to prove that for any open subset $U\subset X$ the natural map $\mathcal{O}_X(U)\longrightarrow \mathcal{O}_{X, \eta}$ is injective. We can assume that $U$ is affine, otherwise take a cover of $U$ by affine open subsets. Then $U = \mathrm{Spec}\: R$ and $\eta$ corresonds to the minimal prime ideal $\mathfrak{p}$ of $R$. Since $X$ is locally defined by primary ideals in itself,  the ideal $(0)$ is primary in $R$ by lemma \ref{lemmaPrimarySheaf}, and thus $\mathfrak{p} = \sqrt{(0)}$. Hence every zero divisor in $R$ is nilpotent, from what follows the inclusion 
    $$\mathcal{O}_X(U) = R\longhookrightarrow S^{-1}_\mathfrak{p}R = \mathcal{O}_{X, \eta}.$$
\end{proof}

\begin{lemma}\label{SpCDelta=CDeltaSp}
    Let $R$ be a Noetherian ring, let $\mathfrak{a}$ be an ideal of $R$, let $\mathfrak{p}$ be a prime ideal of $R$, and let $\Delta = (\mathfrak{p}_0, \ldots, \mathfrak{p}_n)$ be a flag of prime ideals of $R$. Then 
    $$S_{\mathfrak{p}}\left(C_\Delta(\mathfrak{a})\right) = C_\Delta\left(S_{\mathfrak{p}}(\mathfrak{a})\right),$$
    where $S_{\mathfrak{p}}(\mathfrak{a}) = \ker\left(R\longrightarrow S^{-1}_{\mathfrak{p}}\left(R/\mathfrak{a}\right)\right)$, $S_{\mathfrak{p}}\left( C_{\Delta}(\mathfrak{a})\right) = \ker\left(C_\Delta R\longrightarrow S^{-1}_{\mathfrak{p}}C_\Delta\left(R/\mathfrak{a}\right)\right)$.
\end{lemma}
\begin{proof}
    By definition, $S^{-1}_\mathfrak{p}\big(S_\mathfrak{p}(\mathfrak{a})\big) = S^{-1}_\mathfrak{p}\mathfrak{a}$. Thus $S^{-1}_{\mathfrak{p}}C_\Delta\left(R/\mathfrak{a}\right) = S^{-1}_{\mathfrak{p}}C_\Delta\left(R/S_{\mathfrak{p}}(\mathfrak{a})\right)$. Again, by definition of $S_\mathfrak{p}(\mathfrak{a})$, there is an inclusion 
    $$R/S_{\mathfrak{p}}(\mathfrak{a})\overset{\cdot f}{\longhookrightarrow}R/S_{\mathfrak{p}}(\mathfrak{a})$$
    for any $f\in R\setminus \mathfrak{p}$. Hence 
    $$C_\Delta\left(R/S_{\mathfrak{p}}(\mathfrak{a})\right)\overset{\cdot f}{\longhookrightarrow}C_\Delta\left(R/S_{\mathfrak{p}}(\mathfrak{a})\right),$$
    from what follows $C_\Delta\left(R/S_{\mathfrak{p}}(\mathfrak{a})\right)\longhookrightarrow S^{-1}_\mathfrak{p}C_\Delta\left(R/S_{\mathfrak{p}}(\mathfrak{a})\right)$. Therefore
    \begin{eqnarray*}
    S_{\mathfrak{p}}\left( C_{\Delta}(\mathfrak{a})\right) &=& \ker\left(C_\Delta R\longrightarrow S^{-1}_{\mathfrak{p}}C_\Delta\left(R/\mathfrak{a}\right)\right)= \ker\left(C_\Delta R\longrightarrow S^{-1}_{\mathfrak{p}}C_\Delta\left(R/S_{\mathfrak{p}}(\mathfrak{a})\right)\right) \\
    &=& \ker\left(C_\Delta R\longrightarrow C_\Delta\left(R/S_{\mathfrak{p}}(\mathfrak{a})\right)\right) = C_\Delta(S_{\mathfrak{p}}(\mathfrak{a})).
    \end{eqnarray*}
\end{proof}

\begin{lemma}\label{NormalLocalFactorEmbedding}
    Let $R$ be a normal excellent domain, let $\Delta = (\mathfrak{p}_0, \ldots, \mathfrak{p}_n)$ be a flag of prime ideals of $R$, and let $\mathfrak{p}$ be a prime ideal such that $\mathfrak{p} > \Delta$.
    Then
    $$C_\Delta R\longhookrightarrow C_{\mathfrak{p}\vee\Delta}R = C_{\mathfrak{p}}S^{-1}_{\mathfrak{p}}C_\Delta R.$$
\end{lemma}
\begin{proof}
Observe that
    $$\ker\left(C_\Delta R\longrightarrow C_{\mathfrak{p}}S^{-1}_{\mathfrak{p}}C_\Delta R\right) = \bigcap\limits_{n\geq 1}\ker\left(C_\Delta R\longrightarrow S^{-1}_{\mathfrak{p}}C_\Delta \left(R/\mathfrak{p}^n\right)\right) = \bigcap\limits_{n\geq 1}S_{\mathfrak{p}}\left(C_\Delta (\mathfrak{p}^n)\right).$$
    By lemma \ref{SpCDelta=CDeltaSp},
    $$\ker\left(C_\Delta R\longrightarrow C_{\mathfrak{p}}S^{-1}_{\mathfrak{p}}C_\Delta R\right) = \bigcap\limits_{n\geq 1}S_{\mathfrak{p}}\left(C_\Delta (\mathfrak{p}^n)\right) = \bigcap\limits_{n\geq 1}S_{\mathfrak{p}}\left(\mathfrak{p}^n\right)C_\Delta R.$$
    By proposition \ref{Cp0CDeltaflat}, the morphism  $C_{\mathfrak{p}_0}S^{-1}_{\mathfrak{p}_0}R\longrightarrow C_\Delta R$ is flat. Note that $C_{\mathfrak{p}_0}S^{-1}_{\mathfrak{p}_0}R$ is a complete local Noetherian ring, $C_\Delta R$ is a Noetherian ring, and moreover $\mathfrak{p}_0C_\Delta R \subset \mathrm{ rad}(C_\Delta R)$ by \cite[Theorem 8.2]{Ma1}. 
    Then, by \cite[Proposition 5.7(e)]{HJ}), $C_\Delta R$ is intersection flat over $C_{\mathfrak{p}_0}S^{-1}_{\mathfrak{p}_0}R$ and
    thus
    $$\bigcap\limits_{n\geq 1}S_{\mathfrak{p}}\left(\mathfrak{p}^n\right)C_\Delta R = \left(\bigcap\limits_{n\geq 1}S_{\mathfrak{p}}\left(\mathfrak{p}^n\right)C_{\mathfrak{p}_0}S^{-1}_{\mathfrak{p}_0} R\right)\cdot C_\Delta R.$$
    Observe that 
    $$\bigcap\limits_{n\geq 1}S_{\mathfrak{p}}\left(\mathfrak{p}^n\right) C_{\mathfrak{p}_0}S^{-1}_{\mathfrak{p}_0} R = \ker\left( C_{(\mathfrak{p}_0)}R\longrightarrow C_{(\mathfrak{p}, \mathfrak{p}_0)}R\right).$$ 
    Thus we obtain an equality
\begin{equation}\label{eq9}
\ker\left(C_\Delta R\longrightarrow C_{\mathfrak{p}}S^{-1}_{\mathfrak{p}}C_\Delta R\right) = \ker\left( C_{(\mathfrak{p}_0)}R\longrightarrow C_{(\mathfrak{p}, \mathfrak{p}_0)}R\right)\cdot C_\Delta R. 
    \end{equation}

    Denote $\widehat{R}_{\mathfrak{p}_0} = C_{(\mathfrak{p}_0)}R$. Note that $\mathfrak{p}S^{-1}_{\mathfrak{p}}\widehat{R}_{\mathfrak{p}_0}$ is a proper ideal in $S^{-1}_{\mathfrak{p}}\widehat{R}_{\mathfrak{p}_0}$, since the map $S^{-1}_\mathfrak{p}R\longrightarrow ~S^{-1}_{\mathfrak{p}}\widehat{R}_{\mathfrak{p}_0}$ is faithfully flat as a localization of a faithfully flat map $S^{-1}_{\mathfrak{p}_0}R\longrightarrow \widehat{R}_{\mathfrak{p}_0}$. 
    
    Since $R_{\mathfrak{p}_0}$ is a normal excellent local domain, $\widehat{R}_{\mathfrak{p}_0}$ is also a domain (see \cite[Theorem 79]{Ma2}). Therefore the ring $S^{-1}_{\mathfrak{p}}\widehat{R}_{\mathfrak{p}_0}$ is a domain, and since $\mathfrak{p}S^{-1}_{\mathfrak{p}}\widehat{R}_{\mathfrak{p}_0}$ is a proper ideal in $S^{-1}_{\mathfrak{p}}\widehat{R}_{\mathfrak{p}_0}$, we get by Krull's intersection theorem
    $$\bigcap_{n\geq 1}\mathfrak{p}^nS^{-1}_{\mathfrak{p}}\widehat{R}_{\mathfrak{p}_0} = 0.$$
    Note that $\bigcap_{n\geq 1}\mathfrak{p}^nS^{-1}_{\mathfrak{p}}\widehat{R}_{\mathfrak{p}_0} = \ker\left(S^{-1}_{\mathfrak{p}}C_{(\mathfrak{p}_0)}R\longrightarrow C_{(\mathfrak{p}, \mathfrak{p}_0)}R\right)$. Thus we obtain
    $$\ker\left(S^{-1}_{\mathfrak{p}}C_{(\mathfrak{p}_0)}R\longrightarrow C_{(\mathfrak{p}, \mathfrak{p}_0)}R\right) = 0.$$
    Also note that $C_{(\mathfrak{p}_0)}R\hookrightarrow S^{-1}_{\mathfrak{p}}C_{(\mathfrak{p}_0)}R$, since $C_{(\mathfrak{p}_0)}R\overset{\cdot a}{\hookrightarrow}C_{(\mathfrak{p}_0)}R$ for any $a\in R\setminus 0$ by lemma \ref{annExactSequence}. Therefore 
    $$\ker\left(C_{(\mathfrak{p}_0)}R\longrightarrow C_{(\mathfrak{p}, \mathfrak{p}_0)}R\right) = 0,$$
    and, from \eqref{eq9}, we obtain
    $$\ker\left(C_\Delta R\longrightarrow C_{\mathfrak{p}}S^{-1}_{\mathfrak{p}}C_\Delta R\right) = 0.$$
\end{proof}

\begin{lemma}\label{ExcDomainLocalFactorEmb}
        Let $R$ be an excellent domain, let $\Delta = (\mathfrak{p}_0, \ldots, \mathfrak{p}_n)$ be a flag of prime ideals of $R$ with $\mathrm{ht}\:\mathfrak{p}_0 > 0$, and let $k$ be a non-negative integer with $k <  \mathrm{ht}\:\mathfrak{p}_0$.
    Then there exists a finite set $\{\mathfrak{q}_j\}_{j = 1}^{m}$ of prime ideals of $R$ with $\mathfrak{q}_j > \Delta$ and $\mathrm{ht}\:\mathfrak{q}_j = k$ 
    such that there is an inclusion
    $$C_\Delta R\longhookrightarrow \prod_{j = 1}^{m}C_{\mathfrak{q}_j\vee \Delta }R.$$
    \end{lemma}
    \begin{proof}
    Let $S$ be the normalization of $R$ in its field of fractions. Then $S$ is finite over $R$ and excellent by \cite[pp. 260-264]{Ma1}. In particular, $R$ and $S$ are catenary. By \cite[Proposition 3.1.7]{Ye1}, there is an equality
        \begin{equation}\label{eqExcDom1}
        C_\Delta S = \prod_{\Gamma\mid \Delta} C_\Gamma S,
        \end{equation}
        where $C_\Delta S = C_\Delta R \otimes_R S$ (see \cite[Proposition 3.2.1]{Hu}). 
        
        For a flag $\Gamma = (\mathfrak{r}, \ldots)$ of prime ideals of $S$ such that $\Gamma\mid \Delta$ choose a saturated chain of prime ideals $\{\mathfrak{r}_l\}$ between $(0)$ and $\mathfrak{r}$. Since $S$ is catenary, this chain has length $\mathrm{ht}\: \mathfrak{r} + 1$. Since $S$ is finite over $R$, $\mathfrak{r}_l\cap R \subsetneq \mathfrak{r}_{l + 1}\cap R$ for any $0\leq l  < \mathrm{ht}\: \mathfrak{r}$. By \cite[Proposition 2.2]{DF}, $\{\mathfrak{r}_l\cap R\}$ is a saturated chain of prime ideals of $R$ between $(0)$ and $\mathfrak{p}_0$ of length $\mathrm{ht}\:\mathfrak{r} + 1$. Hence we obtain $\mathrm{ht}\:\mathfrak{p}_0 = \mathrm{ht}\: \mathfrak{r}$ and $\mathrm{ht}( \mathfrak{r}_k\cap R) = k$, because $R$ is catenary. Let $\mathfrak{r}_\Gamma = \mathfrak{r}_k$, $\mathfrak{q}_\Gamma = \mathfrak{r}_k\cap R$.

         Since for any $\Gamma$ such that $\Gamma\mid \Delta$ there is an embedding $C_\Gamma S\hookrightarrow C_{\mathfrak{r}_\Gamma\vee \Gamma}S$ by lemma \ref{NormalLocalFactorEmbedding}, we have an embedding
        $$C_\Delta S = \prod_{\Gamma\mid \Delta} C_\Gamma S\longhookrightarrow \prod_{\Gamma\mid \Delta } C_{\mathfrak{q}_\Gamma\vee\Delta}S = \prod_{\Gamma\mid \Delta}C_{\mathfrak{r}_\Gamma\vee \Gamma}S\times \prod_{\Gamma\mid \Delta}\prod_{\substack{\Gamma'\mid \mathfrak{q}_\Gamma\vee \Delta \\ \Gamma' \neq \mathfrak{r}_\Gamma \vee \Gamma}}C_{\Gamma'}S$$
        Let $\{\mathfrak{q}_j\}$ be the set $\{\mathfrak{q}_\Gamma\:\colon\: \Gamma\mid\Delta\}$. Then we obtain a commutative diagram 
        \begin{center}
            \begin{tikzcd}
                C_\Delta S \arrow[hook]{r} & \prod_{j = 1}^mC_{\mathfrak{q}_j\vee\Delta}S \\
                C_\Delta R\arrow[hook]{u}\arrow{r} & \prod_{j = 1}^mC_{\mathfrak{q}_j\vee\Delta}R\arrow[hook]{u},
            \end{tikzcd}
        \end{center}
        where vertical inclusions follow from the fact that $C_\Delta R$, $C_{\mathfrak{q}_j\vee \Delta}R$ are flat over $R$ and $R\hookrightarrow S$. Then the bottom horizontal arrow is injective, from what the assertion follows.
    \end{proof}

     \begin{prop}\label{SCDeltaNinclusion}
       Let $R$ be an excellent domain, let $\Delta = (\mathfrak{p}_0, \ldots, \mathfrak{p}_n)$ be a flag of prime ideals of $R$ with $\mathrm{ht}\:\mathfrak{p}_0 > 0$, and let $k < \mathrm{ht}\:\mathfrak{p}_0$ be a non-negative integer. Let $M$ be a finitely generated $R$-module.    
       Then there exists a finite set $\{\mathfrak{q}_j\}_{j = 1}^{m}$ of prime ideals of $R$ with $\mathfrak{q}_j > \Delta$ and $\mathrm{ht}\:\mathfrak{q}_j = k$ for $1\leq j \leq m$ such that there is an inclusion
    $$T^{-1}C_\Delta M\longhookrightarrow \prod_{j = 1}^{m}C_{\mathfrak{q}_j\vee \Delta }M,$$
    where $T = R\setminus \bigcup_{j = 1}^m\mathfrak{q}_j$ is a multiplicative set.
   \end{prop}
   \begin{proof}
   Assume $M\neq 0$. By \cite[Theorem 6.4]{Ma1}, there exists a filtration
   $$0 = M_l \subset M_{l - 1}\subset\ldots\subset M_1\subset M_0 = M$$
   such that $M_i/M_{i + 1}\cong R/\mathfrak{r}_i$, $\mathfrak{r}_i\in \mathrm{Spec}\:R$, for $0\leq i < l$.

   Let $I = \big\{i\mid 0\leq i < l, \: \mathfrak{r}_i\subset \mathfrak{p}_0, \: \mathrm{ht}\: \mathfrak{r}_i \leq k\big\}$. For each $i\in I$ fix a set $\{\mathfrak{q}_{i, s}\}$ of prime ideals of $R$ such that $\{\mathfrak{q}_{i, s}/\mathfrak{r}_i\}$ satisfies lemma \ref{ExcDomainLocalFactorEmb} for the ring $R/\mathfrak{r}_i$, the flag $\Delta\:\mathrm{mod}\: \mathfrak{r}_i = (\mathfrak{p}_0/\mathfrak{r}_i, \ldots, \mathfrak{p}_n/\mathfrak{r}_i)$, and the integer $k - \mathrm{ht}\: \mathfrak{r}_i$. Since $R$ is a excellent and, in particular, catenary domain, $\mathrm{ht}\:\mathfrak{q}_{i, s} = k$ for any $s$.  
   Let $\{\mathfrak{q}'_d\}$ be a set of prime ideals of $R$ which satisfies lemma \ref{ExcDomainLocalFactorEmb} for the ring $R$, the flag $\Delta$, and the integer $k$. Note that this set can be chosen up to adding a finite number of prime ideals $\{\mathfrak{t}_j\}$ with $\mathrm{ht}\:\mathfrak{t}_j = k$ and $\mathfrak{t}_j > \Delta$ such that lemma \ref{ExcDomainLocalFactorEmb} is still satisfied (cf. subsequent remark \ref{RemAddFiniteSet}).
   Let 
   $$\{\mathfrak{q}_j\}_{j = 1}^m = \bigcup_{i\in I}\{\mathfrak{q}_{i, s}\}\cup\{\mathfrak{q}_d'\}, \quad T = R\setminus \bigcup_{j = 1}^m\mathfrak{q}_j.$$
By construction, $\mathrm{ht}\: \mathfrak{q}_j = k$ and $\mathfrak{q}_j\subset \mathfrak{p}_0$ for $1\leq j \leq m$.

Consider an exact sequence of $R$-modules
        \begin{equation}\label{eqSqj}
        0\longrightarrow T^{-1}C_\Delta R \longrightarrow \prod_{j = 1}^mC_{\mathfrak{q}_j\vee \Delta}R\longrightarrow \prod_{j = 1}^mC_{\mathfrak{q}_j\vee \Delta}R\bigg/T^{-1}C_\Delta R\longrightarrow 0.
        \end{equation}
   The inclusion follows from lemma \ref{ExcDomainLocalFactorEmb} and from the fact that $R$ is a domain. 

   Let us denote $P = \prod_{j = 1}^mC_{\mathfrak{q}_j\vee \Delta}R\bigg/T^{-1}C_\Delta R$. Since $T^{-1}C_\Delta R$, $C_{\mathfrak{q}_j\vee \Delta}R$ are flat $R$-modules, we obtain for any $R$-module $N$ (after tensoring \eqref{eqSqj} by $N$) that 
   $$\mathrm{Tor}^R_q(P, N) = 0, \quad q\geq 2,$$
   and the sequence
   \begin{equation}\label{eqTor}        
        0\longrightarrow \mathrm{Tor}_1^R(P, N)\longrightarrow T^{-1}C_\Delta R\otimes_R N\longrightarrow \prod_{j = 1}^m\big(C_{\mathfrak{q}_j\vee \Delta}R\otimes_R N\big)\longrightarrow P\otimes_R N\longrightarrow 0
        \end{equation}
        is exact. 

        Let us prove that for any $0\leq i < l$ we have $\mathrm{Tor}_1^R(P, R/\mathfrak{r}_i) = 0$. First, assume that $\mathrm{ht}\: \mathfrak{r}_i > k$ or $\mathfrak{r}_i\not\subset \mathfrak{p}_0$. In both cases, $\mathfrak{r}_i\not\subset \bigcup_{j = 1}^m\mathfrak{q}_j$, otherwise $\mathfrak{r}_i\subset \mathfrak{q}_j\subset \mathfrak{p}_0$ for some $j$ by prime avoidance, which is a contradiction, since then $\mathfrak{r}_i\subset \mathfrak{p}_0$ and $\mathrm{ht}\:\mathfrak{r}_i \leq \mathrm{ht}\:\mathfrak{q}_j = k$. Therefore $T^{-1}\mathfrak{r}_i = T^{-1}R$, and thus $$T^{-1}C_\Delta R\otimes_R R/\mathfrak{r}_i = \big( T^{-1}R\otimes_R R/\mathfrak{r}_i\big)\otimes_R C_\Delta R = \big(T^{-1}R\big/T^{-1}\mathfrak{r}_i\big)\otimes_R C_\Delta R = 0.$$
        From \eqref{eqTor} we obtain $\mathrm{Tor}_1^R(P, R/\mathfrak{r}_i) = 0$.

        Consider the remaining case $\mathrm{ht}\: \mathfrak{r}_i \leq k$ and $\mathfrak{r}_i\subset \mathfrak{p}_0$. By construction of $\{\mathfrak{q}_j\}_{j = 1}^m$, there is an inclusion 
        $$C_{\tilde{\Delta}}(R/\mathfrak{r}_i)\longhookrightarrow \prod_{j = 1}^m C_{\tilde{\mathfrak{q}}_j\vee \tilde{\Delta}}(R/\mathfrak{r}_i),$$
        where $\:\tilde{}\:$ means $\mathrm{mod}\: \mathfrak{r}_i$ (we assume that $C_{\tilde{\mathfrak{q}}_j\vee \tilde{\Delta}}(R/\mathfrak{r}_i) = 0$ whenever $\mathfrak{r}_i\not\subset \mathfrak{q}_j$). 
        Thus we have $\tilde{T}^{-1}C_{\tilde{\Delta}}(R/\mathfrak{r}_i)\hookrightarrow \prod_{j = 1}^m C_{\tilde{\mathfrak{q}}_j\vee \tilde{\Delta}}(R/\mathfrak{r}_i)$, where $\tilde{T} = T \: \mathrm{mod}\: \mathfrak{r}_i$, since $R/\mathfrak{r}_i$ is a domain. By lemma \ref{LemmaCompLocal}, we obtain
        $${T}^{-1}C_{{\Delta}}(R/\mathfrak{r}_i)\longhookrightarrow \prod_{j = 1}^m C_{{\mathfrak{q}}_j\vee {\Delta}}(R/\mathfrak{r}_i).$$
        Therefore, from \eqref{eqTor}, $\mathrm{Tor}_1^R(P, R/\mathfrak{r}_i) = 0$.

        Consider the sequence 
        $$0\longrightarrow M_{i + 1}\longrightarrow M_i \longrightarrow R/\mathfrak{r}_i\longrightarrow 0$$
        for $0\leq i < l$. After tensoring this sequence with $P$, we obtain an exact sequence
        $$\ldots \longrightarrow \mathrm{Tor}_2^R(R/\mathfrak{r}_i, P)\longrightarrow \mathrm{Tor}_1^R(M_{i + 1}, P) \longrightarrow \mathrm{Tor}_1^R(M_i, P) \longrightarrow \mathrm{Tor}_1^R(R/\mathfrak{r}_i, P) \longrightarrow \ldots.$$
        Since $\mathrm{Tor}_2^R(R/\mathfrak{r}_i, P) = \mathrm{Tor}_1^R(R/\mathfrak{r}_i, P) = 0$, we obtain an isomorphism 
        $$\mathrm{Tor}_1^R(M_{i + 1}, P) \overset{\sim}{\longrightarrow} \mathrm{Tor}_1^R(M_i, P).$$
        Thus, by inductive argument, 
        $\mathrm{Tor}_1^R(M, P) = \mathrm{Tor}_1^R(M_0, P) \cong \mathrm{Tor}_1^R(M_l, P) = 0$. From this fact and from \eqref{eqTor} we obtain the required inclusion.
   \end{proof}

   \begin{lemma}\label{Corideala}
       Let $R$ be an excellent domain, let $\Delta = (\mathfrak{p}_0, \ldots, \mathfrak{p}_n)$ be a flag of prime ideals of $R$ with $\mathrm{ht}\:\mathfrak{p}_0 > 0$, and let $k < \mathrm{ht}\:\mathfrak{p}_0$ be a non-negative integer. Let $\mathfrak{a}$ be an ideal in $R$. Then there exists a finite set $\{\mathfrak{q}_j\}_{j = 1}^m$ of prime ideals of $R$ with $\mathfrak{q}_j > \Delta$ and $\mathrm{ht}\:\mathfrak{q}_j = k$ for $1\leq j \leq m$ such that 
       $$C_\Delta R\longhookrightarrow \prod_{j = 1}^mC_{\mathfrak{q}_j\vee \Delta}R, \quad C_\Delta R\cap \prod_{j = 1}^m\mathfrak{a}C_{\mathfrak{q}_j\vee \Delta}R = \mathfrak{b}C_\Delta R,$$
       where $\mathfrak{b} = \ker\big(R\rightarrow T^{-1}(R/\mathfrak{a})\big)$ for $T = R\setminus \bigcup_{j = 1}^m\mathfrak{q}_j,$ and the intersection is taken in $\prod_{j = 1}^mC_{\mathfrak{q}_j\vee \Delta}R$.
   \end{lemma}
   \begin{proof}
       Let $\{\mathfrak{q}_j\}_{j = 1}^m$ be a set of prime ideals from proposition \ref{SCDeltaNinclusion} for $M = R/\mathfrak{a}$. From the proof of proposition \ref{SCDeltaNinclusion} we can choose this set in such a way that lemma \ref{ExcDomainLocalFactorEmb} holds for $\{\mathfrak{q}_j\}_{j = 1}^m$ (see remark \ref{RemAddFiniteSet}). Thus
       \begin{equation}\label{eqInclCDelta}
            C_\Delta R\longhookrightarrow \prod_{j = 1}^mC_{\mathfrak{q}_j\vee\Delta}R, \quad T^{-1}C_\Delta (R/\mathfrak{a})\longhookrightarrow \prod_{j = 1}^mC_{\mathfrak{q}_j\vee\Delta}(R/\mathfrak{a}),
        \end{equation}
        where $T = R\setminus \bigcup_{j = 1}^m\mathfrak{q}_j$. The latter inclusion is equivalent to the equality
        $$T^{-1}C_\Delta R\cap \prod_{j = 1}^m\mathfrak{a}C_{\mathfrak{q}_j\vee\Delta}R = \mathfrak{a}T^{-1}C_\Delta R,$$
        since the left side is equal to $\ker\big(T^{-1}C_\Delta R\rightarrow\prod_{j = 1}^mC_{\mathfrak{q}_j\vee\Delta}(R/\mathfrak{a})\big)$. Then
        \begin{eqnarray*}
            C_\Delta R \cap \prod_{j = 1}^m\mathfrak{a}C_{\mathfrak{q}_j\vee\Delta}R &=& C_\Delta R\cap T^{-1}C_\Delta R\cap\prod_{j = 1}^m\mathfrak{a}C_{\mathfrak{q}_j\vee\Delta}R  \\
            &=& C_\Delta R\cap \mathfrak{a}T^{-1}C_\Delta R = (R\cap \mathfrak{a}T^{-1}R)\cdot C_\Delta R,
        \end{eqnarray*}
        where the last equality follows from the flatness of $C_\Delta R$ over $R$.
        It remains to note that $R\cap \mathfrak{a}T^{-1}R = \ker\big(R\rightarrow T^{-1}(R/\mathfrak{a})\big)$.
   \end{proof}

   \begin{rem}\label{RemAddFiniteSet}
    Note that a set of prime ideals in proposition \ref{SCDeltaNinclusion}, and therefore in lemma \ref{Corideala}, can be chosen up to adding a finite set of prime ideals satisfying certain conditions. Indeed, when we consider a set $\{\mathfrak{q}'_d\}$ 
    in the proof of proposition \ref{SCDeltaNinclusion} that satisfies lemma \ref{ExcDomainLocalFactorEmb} for the ring $R$, we can add to it an arbitrary set $\{\mathfrak{r}_i\}$ of prime ideals of $R$ with $\mathrm{ht}\:\mathfrak{r}_i = k$ and $\mathfrak{r}_i > \Delta$ such that lemma \ref{ExcDomainLocalFactorEmb} still holds. 
   
   \end{rem}

    \begin{prop}\label{Coridealn}
        Let $R$ be an excellent domain, let $\Delta = (\mathfrak{p}_0, \ldots, \mathfrak{p}_n)$ be a flag of prime ideals of $R$ with $\mathrm{ht}\:\mathfrak{p}_0 > 0$, and let $k < \mathrm{ht}\:\mathfrak{p}_0$ be a non-negative integer. Let $\mathfrak{n}$ be a primary ideal in $R$ such that $\mathrm{ht}\:\mathfrak{n}\leq k$ and $\mathfrak{n} > \Delta$. Then there exists a finite set $\{\mathfrak{q}_j\}_{j = 1}^m$ of prime ideals of $R$ with $\mathfrak{q}_j > \Delta$ and $\mathrm{ht}\:\mathfrak{q}_j = k$ for $1\leq j \leq m$ such that $\mathfrak{n}\subset \bigcup_{j =1}^m\mathfrak{q}_j$ and
       $$C_\Delta R\longhookrightarrow \prod_{j = 1}^mC_{\mathfrak{q}_j\vee \Delta}R, \quad C_\Delta R\cap \prod_{j = 1}^m\mathfrak{n}C_{\mathfrak{q}_j\vee \Delta}R = \mathfrak{n}C_\Delta R,$$
       where 
       the intersection is taken in $\prod_{j = 1}^mC_{\mathfrak{q}_j\vee \Delta}R$.
    \end{prop}
    \begin{proof}
    Since $\mathfrak{n}$ is a primary ideal, $\mathfrak{p} = \sqrt{\mathfrak{n}}$ is a prime ideal in $R$ which is minimal over $\mathfrak{n}$. Thus $\mathrm{ht}\: \mathfrak{p} = \mathrm{ht}\: \mathfrak{n}\leq k$. Since $R$ is an excellent and, in particular, catenary domain, there is a saturated chain of prime ideals between $\mathfrak{p}$ and $\mathfrak{p}_0$ of length $\mathrm{ht}\: \mathfrak{p}_0 - \mathrm{ht}\: \mathfrak{p}$; let $\mathfrak{q}$ be a prime ideal in this chain with $\mathrm{ht}\: \mathfrak{q} = k$. By construction, $\mathfrak{q}>\Delta$.

    Let $\{\mathfrak{q}_j\}_{j = 1}^m$ be a set of prime ideals from lemma \ref{Corideala} for the ideal $\mathfrak{n}$. Observe that we can assume $\mathfrak{q}\in \{\mathfrak{q}_j\}_{j = 1}^m$ by remark \ref{RemAddFiniteSet}. 
    Thus 
    $$\mathfrak{n}\subset \mathfrak{q}\subset \bigcup_{j = 1}^m\mathfrak{q}_j.$$
    Note that $R\setminus \bigcup_{j = 1}^m\mathfrak{q}_j \subset R\setminus\mathfrak{q}\subset R\setminus \mathfrak{p}$. Hence
    $$\ker\big(R\longrightarrow T^{-1}(R/\mathfrak{n})\big)\subset \ker\big(R\longrightarrow S_{\mathfrak{p}}^{-1}(R/\mathfrak{n})\big) = S_{\mathfrak{p}}(\mathfrak{n}) =\mathfrak{n},$$
    where $T = R\setminus \bigcup_{j = 1}^m\mathfrak{q}_j$, and the last equality follows from the primarity of $\mathfrak{n}$.
    \end{proof}

\begin{cor}\label{LocalAdeleInterPrimary}
       Let $R$ be an excellent domain, let $\Delta = (\mathfrak{p}_0, \ldots, \mathfrak{p}_n)$ be a flag of prime ideals of $R$ with $\mathrm{ht}\: \mathfrak{p}_0 > 0$, and let $k$ be a non-negative integer with $ k < \mathrm{ht}\: \mathfrak{p}_0$. Let $\mathfrak{n}$ be a primary ideal in $R$ such that $\mathrm{ht}\:\mathfrak{n}\leq k$ and $\mathfrak{n} > \Delta$.
       Then 
       $$C_\Delta R \cap \prod_{\substack{\mathfrak{q}> \Delta\\ \mathrm{ht}\: \mathfrak{q} = k}}\mathfrak{n}C_{\mathfrak{q}\vee\Delta}R = \mathfrak{n}C_\Delta R.$$
   \end{cor}
   \begin{proof}
       The assertion clearly follows from proposition \ref{Coridealn}.
   \end{proof}

   \begin{lemma}\label{aSCDeltaR}
       Let $R$ be a Noetherian domain, let $\mathfrak{a}$ be an ideal in $R$, and let $\Delta$ be a flag of prime ideals in $R$. Let $\Lambda$ be an arbitrary set of flags of prime ideals such that $\Delta\subset \Delta'$ for any $\Delta'\in \Lambda$ and $C_\Delta R\longhookrightarrow \prod\limits_{\Delta'\in \Lambda}C_{\Delta'}R$. Let $T$ be a multiplicative set in $R$. Then there is an equality
       $$T^{-1}C_\Delta R\cap \prod_{\Delta'\in \Lambda}\mathfrak{a}T^{-1}C_{\Delta'}R = T^{-1}C_\Delta R\cap T^{-1}\left(\prod_{\Delta'\in \Lambda}\mathfrak{a}C_{\Delta'}R\right)$$
       in $\prod_{\Delta'\in \Lambda}T^{-1}C_{\Delta'}R$.
   \end{lemma}
   \begin{proof}
    The proof is straightforward.
   \end{proof}

   Observe that if $\Delta\in S(R, I)$ is a flag of prime ideals of a biequidimensional ring $R$, then for any $J\subset \{0, 1, \ldots, \dim R\}$ such that $I\subset J$ there exists $\Delta'\in S(R, J)$ such that $\Delta\subset \Delta'$. This follows from remark \ref{remBiequidim}.

   \begin{thm}\label{LocalFactorEmbedding}
        Let $R$ be an excellent  biequidimensional domain and let $I, J\subset \{0, 1, \ldots, \dim R\}$ be subsets such that $I\subset J$. Let $\Delta \in S(R, I)$. Let $\mathfrak{n}$ be a primary ideal in $R$ such that $\mathrm{ht}\: \mathfrak{n}\leq j$ if $J = (j, \ldots)\neq \varnothing$, and $\mathfrak{n} \geq  \Delta$. Then 
        $$C_\Delta R\longhookrightarrow \prod_{\substack{\Delta'\in S(R, J)\\ \Delta\subset \Delta'}}C_{\Delta'}R, \quad C_\Delta R\cap \prod_{\substack{\Delta'\in S(R, J)\\ \Delta\subset \Delta'}}\mathfrak{n}C_{\Delta'}R = \mathfrak{n}C_\Delta R.$$
    \end{thm}
    \begin{proof}
        If $I = J$, the assertion is clear, so we can assume $I\subsetneq J$. In particular, $J\neq \varnothing$.

        If $I = \varnothing$, then $\Delta = \varnothing$. By proposition \ref{REmbeddsInCDeltaR},
        $$R\longhookrightarrow \prod\limits_{\Delta'\in S(R, J)}C_{\Delta'}R, \quad R/\mathfrak{n} \longhookrightarrow \prod_{\tilde{\Delta}'\in S(R/\mathfrak{n},\: J - \mathrm{ht}\: \mathfrak{n})}C_{\tilde{\Delta}'}(R/\mathfrak{n}).$$
        Note that there is a natural bijection between sets $S(R/\mathfrak{n},\: J - \mathrm{ht}\: \mathfrak{n})$ and $\{\Delta'\in S(R, J)\mid \mathfrak{n} \geq \Delta'\}$, since $R$ is an excellent and, in particular, catenary domain.
        Note that the latter set is non-empty, since $R$ is biequidimensional.
        Hence the second inclusion is equivalent to the equality 
        $$R\cap \prod_{\substack{\Delta'\in S(R, J)\\ \mathfrak{n} \geq \Delta'}}\mathfrak{n}C_{\Delta'}R = \mathfrak{n},$$
        and therefore
        $$R\cap \prod_{\Delta'\in S(R, J)}\mathfrak{n}C_{\Delta'}R = \mathfrak{n}.$$

        Let us proceed by induction on $|J|$. If $|J| = 1$, then $I = \varnothing$ by our assumption, and this case is examined above. Hence we can assume $|J| > 1$ and $I\neq \varnothing$. Then let $I = (i, \ldots)$, $J = (j, \ldots)$. Since $I\subset J$, $j \leq i$. Let also $\Delta = (\mathfrak{p}_0, \ldots, \mathfrak{p}_n)$, where $n = |I| - 1$. By definition, $\mathrm{ht}\:\mathfrak{p}_0 = i$. Note that the set $\big\{\Delta'\in S(R, J)\mid \Delta\subset \Delta'\big\}$ is non-empty, since $R$ is biequidimensional.

        First, assume $j < i$. By induction hypothesis, 
        $$C_\Delta R\longhookrightarrow \prod_{\substack{\Gamma'\in S(R, J\setminus j)\\ \Delta\subset \Gamma'}}C_{\Gamma'}R, \quad C_\Delta R\cap \prod_{\substack{\Gamma'\in S(R, J\setminus j)\\ \Delta\subset \Gamma'}}\mathfrak{n}C_{\Gamma'}R = \mathfrak{n}C_\Delta R.$$
        Fix any $\Gamma' \in S(R, J\setminus j)$. Let $\Gamma' = (\mathfrak{r},\ldots)$. By lemma \ref{ExcDomainLocalFactorEmb}, 
    $$C_{\Gamma'} R\longhookrightarrow \prod_{\substack{\mathfrak{q} > \Gamma'\\ \mathrm{ht}\:\mathfrak{q} = j}}C_{\mathfrak{q}\vee\Gamma'}R.$$
    If $\mathfrak{n}\geq \Gamma'$ (in fact, $\mathfrak{n} > \Gamma'$, since $\mathrm{ht}\:\mathfrak{n} = j$), then, by proposition \ref{LocalAdeleInterPrimary}, 
    \begin{equation}\label{eqInterPrimary}
    C_{\Gamma'} R\cap \prod_{\substack{\mathfrak{q} > \Gamma'\\ \mathrm{ht}\:\mathfrak{q} = j}}\mathfrak{n}C_{\mathfrak{q}\vee\Gamma'}R = \mathfrak{n}C_{\Gamma'} R.
    \end{equation}
    If $\mathfrak{n}\not\geq \Gamma'$ or, in other words, $\mathfrak{n}\not\subset \mathfrak{r}$,  
    then $\mathfrak{n}C_{\Gamma'}R = S^{-1}_\mathfrak{r}\mathfrak{n}C_{\Gamma'}R= C_{\Gamma'}R$, since $S^{-1}_\mathfrak{r}\mathfrak{n} = R_\mathfrak{r}$.
    Thus the equality \eqref{eqInterPrimary} is also satisfied in this case.
    Consequently,
    $$C_\Delta R \longhookrightarrow \prod_{\substack{\Gamma'\in S(R, J\setminus j)\\ \Delta\subset \Gamma'}}C_{\Gamma'}R\longhookrightarrow \prod_{\substack{\Gamma'\in S(R, J\setminus j)\\ \Delta\subset \Gamma'}}\prod_{\substack{\mathfrak{q} > \Gamma'\\ \mathrm{ht}\:\mathfrak{q} = j}}C_{\mathfrak{q}\vee\Gamma'}R = \prod_{\substack{\Delta'\in S(R, J)\\ \Delta\subset \Delta'}}C_{\Delta'}R$$
    and
    \begin{eqnarray*}
    C_\Delta R\cap \prod_{\substack{\Delta'\in S(R, J)\\ \Delta\subset \Delta'}}\mathfrak{n}C_{\Delta'}R &=& C_\Delta R\cap \prod_{\substack{\Gamma'\in S(R, J\setminus j)\\ \Delta\subset \Gamma'}}\prod_{\substack{\mathfrak{q} > \Gamma'\\ \mathrm{ht}\:\mathfrak{q} = j}}\mathfrak{n}C_{\mathfrak{q}\vee\Gamma'}R  \\
    &=& C_\Delta R\cap \prod_{\substack{\Gamma'\in S(R, J\setminus j)\\ \Delta\subset \Gamma'}}C_{\Gamma'} R\cap \prod_{\substack{\Gamma'\in S(R, J\setminus j)\\ \Delta\subset \Gamma'}}\prod_{\substack{\mathfrak{q} > \Gamma'\\ \mathrm{ht}\:\mathfrak{q} = j}}\mathfrak{n}C_{\mathfrak{q}\vee\Gamma'}R \\
    &=& C_\Delta R\cap \prod_{\substack{\Gamma'\in S(R, J\setminus j)\\ \Delta\subset \Gamma'}}\mathfrak{n}C_{\Gamma'} R = \mathfrak{n}C_\Delta R.
    \end{eqnarray*}

    Let us consider the remaining case $j = i$. Then $\Delta' = (\mathfrak{p}_0, \ldots)$ for any $\Delta'\in S(R, J)$ such that $\Delta\subset \Delta'$. Let $\Gamma = \Delta\setminus \mathfrak{p}_0$. Define $\Lambda = \{\Gamma'\in S(R, J\setminus i)\mid \Gamma\subset \Gamma'\}$. Note that $\mathfrak{p}_0^{(m)}$ is a primary ideal by the remark after definition \ref{DefSymbolicPower}. Since $\sqrt{\mathfrak{p}_0^{(m)}} = \mathfrak{p}_0$, then $\mathrm{ht}\: \mathfrak{p}_0^{(m)} = \mathrm{ht}\:\mathfrak{p}_0$ and $\mathfrak{p}_0^{(m)} > \Gamma$. By induction hypothesis, 
    $$C_\Gamma R\longhookrightarrow \prod_{\Gamma'\in \Lambda}C_{\Gamma'}R, \quad C_\Gamma R\cap \prod_{\Gamma'\in \Lambda}\mathfrak{n}C_{\Gamma'}R = \mathfrak{n}C_\Gamma R,$$
    \begin{equation}\label{eqInterpm}
    C_\Gamma R\cap \prod_{\Gamma'\in \Lambda}\mathfrak{p}_0^{(m)}C_{\Gamma'}R = \mathfrak{p}_0^{(m)}C_\Gamma R
    \end{equation}
    for any $m\geq 1$. Since $S_{\mathfrak{p}_0}^{-1}\mathfrak{p}_0^{(m)} = S_{\mathfrak{p}_0}^{-1}\mathfrak{p}_0^m$, we obtain
    \begin{eqnarray*}
    S^{-1}_{\mathfrak{p}_0}C_\Gamma R\cap \prod_{\Gamma'\in \Lambda}\mathfrak{p}_0^mS^{-1}_{\mathfrak{p}_0}C_{\Gamma'}R &=& S^{-1}_{\mathfrak{p}_0}C_\Gamma R\cap \prod_{\Gamma'\in \Lambda}\mathfrak{p}_0^{(m)}S^{-1}_{\mathfrak{p}_0}C_{\Gamma'}R \\
    &=& S^{-1}_{\mathfrak{p}_0}\bigg(C_\Gamma R\cap \prod_{\Gamma'\in \Lambda}\mathfrak{p}_0^{(m)}C_{\Gamma'}R\bigg)\\
    &=& \mathfrak{p}_0^{(m)}S^{-1}_{\mathfrak{p}_0}C_{\Gamma}R = \mathfrak{p}_0^{m}S^{-1}_{\mathfrak{p}_0}C_{\Gamma}R,
    \end{eqnarray*}
    where the second equality follows from lemma \ref{aSCDeltaR} and the third equality follows from \eqref{eqInterpm}; this is equivalent to the embedding
    \begin{equation}\label{eqInclGamma}S^{-1}_{\mathfrak{p}_0}C_{\Gamma}R\big/\mathfrak{p}_0^mS^{-1}_{\mathfrak{p}_0}C_{\Gamma}R\longhookrightarrow \prod_{\Gamma'\in \Lambda}S^{-1}_{\mathfrak{p}_0}C_{\Gamma'}R\big/\mathfrak{p}_0^mS^{-1}_{\mathfrak{p}_0}C_{\Gamma'}R.
    \end{equation}
    Since $\mathfrak{n}/(\mathfrak{n}\cap\mathfrak{p}_0^m)\hookrightarrow R/\mathfrak{p}_0^m$, then 
    \begin{equation*}
        \prod_{\Gamma'\in \Lambda}\mathfrak{n}S^{-1}_{\mathfrak{p}_0}C_{\Gamma'}R\big/(\mathfrak{n}\cap\mathfrak{p}_0^m)S^{-1}_{\mathfrak{p}_0}C_{\Gamma'}R\longhookrightarrow \prod_{\Gamma'\in \Lambda}S^{-1}_{\mathfrak{p}_0}C_{\Gamma'}R\big/\mathfrak{p}_0^mS^{-1}_{\mathfrak{p}_0}C_{\Gamma'}R
    \end{equation*}
    by the flatness of $S^{-1}_{\mathfrak{p}_0}C_{\Gamma'}R$ over $R$.
    Since $R$ is Noetherian, 
    $$\prod_{\Gamma'\in \Lambda}\mathfrak{n}S^{-1}_{\mathfrak{p}_0}C_{\Gamma'}R\big/(\mathfrak{n}\cap\mathfrak{p}_0^m)S^{-1}_{\mathfrak{p}_0}C_{\Gamma'}R = \frac{\mathfrak{n}\prod_{\Gamma'\in \Lambda}S^{-1}_{\mathfrak{p}_0}C_{\Gamma'}R}{(\mathfrak{n}\cap\mathfrak{p}_0^m)\prod_{\Gamma'\in \Lambda}S^{-1}_{\mathfrak{p}_0}C_{\Gamma'}R}.$$
    By lemma \ref{LemmaABCD},
    \begin{eqnarray*}
        \frac{S^{-1}_{\mathfrak{p}_0}C_{\Gamma}R}{\mathfrak{p}_0^mS^{-1}_{\mathfrak{p}_0}C_{\Gamma}R}\cap \frac{\mathfrak{n}\prod_{\Gamma'\in \Lambda}S^{-1}_{\mathfrak{p}_0}C_{\Gamma'}R}{(\mathfrak{n}\cap\mathfrak{p}_0^m)\prod_{\Gamma'\in \Lambda}S^{-1}_{\mathfrak{p}_0}C_{\Gamma'}R} &=& \frac{S^{-1}_{\mathfrak{p}_0}C_{\Gamma}R\cap\big(\mathfrak{n}\prod_{\Gamma'\in \Lambda}S^{-1}_{\mathfrak{p}_0}C_{\Gamma'}R + \mathfrak{p}_0^m\prod_{\Gamma'\in \Lambda}S^{-1}_{\mathfrak{p}_0}C_{\Gamma'}R\big)}{\mathfrak{p}_0^mS^{-1}_{\mathfrak{p}_0}C_{\Gamma}R}\\
        &=& \frac{S^{-1}_{\mathfrak{p}_0}C_{\Gamma}R\cap(\mathfrak{n}+\mathfrak{p}_0^m)\prod_{\Gamma'\in \Lambda}S^{-1}_{\mathfrak{p}_0}C_{\Gamma'}R}{\mathfrak{p}_0^mS^{-1}_{\mathfrak{p}_0}C_{\Gamma}R}.
    \end{eqnarray*}
    Note that $\sqrt{\mathfrak{n} + \mathfrak{p}_0^m} = \mathfrak{p}_0$, since $\mathfrak{n}\subset \mathfrak{p}_0$. Thus $\mathfrak{p}_0$ is a minimal prime ideal in the primary decomposition of the ideal $\mathfrak{n} + \mathfrak{p}_0^m$, and $S_{\mathfrak{p}_0}(\mathfrak{n} + \mathfrak{p}_0^m)$ is its $\mathfrak{p}_0$-primary part and, moreover, $S_{\mathfrak{p}_0}(\mathfrak{n} + \mathfrak{p}_0^m) > \Gamma$. Hence, by induction hypothesis,
    $$C_\Gamma R\cap \prod_{\Gamma'\in \Lambda}S_{\mathfrak{p}_0}(\mathfrak{n} + \mathfrak{p}_0^m)C_{\Gamma'}R = S_{\mathfrak{p}_0}(\mathfrak{n} + \mathfrak{p}_0^m)C_\Gamma R.$$
    Therefore, by lemma \ref{aSCDeltaR},
    \begin{eqnarray*}
    S^{-1}_{\mathfrak{p}_0}C_{\Gamma}R\cap(\mathfrak{n}+\mathfrak{p}_0^m)\prod_{\Gamma'\in \Lambda}S^{-1}_{\mathfrak{p}_0}C_{\Gamma'}R &=& S^{-1}_{\mathfrak{p}_0}C_{\Gamma}R\cap S_{\mathfrak{p}_0}(\mathfrak{n} + \mathfrak{p}_0^m)\prod_{\Gamma'\in \Lambda}S^{-1}_{\mathfrak{p}_0}C_{\Gamma'}R \\
        &=& S_{\mathfrak{p}_0}^{-1}\left(C_{\Gamma}R\cap S_{\mathfrak{p}_0}(\mathfrak{n} + \mathfrak{p}_0^m)\prod_{\Gamma'\in \Lambda}C_{\Gamma'}R\right)\\
        &=&  S_{\mathfrak{p}_0}(\mathfrak{n} + \mathfrak{p}_0^m)S^{-1}_{\mathfrak{p}_0}C_{\Gamma}R = (\mathfrak{n} + \mathfrak{p}_0^m)S^{-1}_{\mathfrak{p}_0}C_{\Gamma}R.
    \end{eqnarray*}
    Thereby
    $$\frac{S^{-1}_{\mathfrak{p}_0}C_{\Gamma}R\cap(\mathfrak{n}+\mathfrak{p}_0^m)\prod_{\Gamma'\in \Lambda}S^{-1}_{\mathfrak{p}_0}C_{\Gamma'}R}{\mathfrak{p}_0^mS^{-1}_{\mathfrak{p}_0}C_{\Gamma}R} = \frac{(\mathfrak{n} + \mathfrak{p}_0^m)S^{-1}_{\mathfrak{p}_0}C_{\Gamma}R}{\mathfrak{p}_0^mS^{-1}_{\mathfrak{p}_0}C_{\Gamma}R} = \frac{\mathfrak{n}S^{-1}_{\mathfrak{p}_0}C_{\Gamma}R}{(\mathfrak{n}\cap\mathfrak{p}_0^m)S^{-1}_{\mathfrak{p}_0}C_{\Gamma}R}.$$
    Finally, we obtain the equality
    \begin{equation}\label{eqInterPrimaryn}
    \frac{S^{-1}_{\mathfrak{p}_0}C_{\Gamma}R}{\mathfrak{p}_0^mS^{-1}_{\mathfrak{p}_0}C_{\Gamma}R}\cap \frac{\mathfrak{n}\prod_{\Gamma'\in \Lambda}S^{-1}_{\mathfrak{p}_0}C_{\Gamma'}R}{(\mathfrak{n}\cap\mathfrak{p}_0^m)\prod_{\Gamma'\in \Lambda}S^{-1}_{\mathfrak{p}_0}C_{\Gamma'}R} = \frac{\mathfrak{n}S^{-1}_{\mathfrak{p}_0}C_{\Gamma}R}{(\mathfrak{n}\cap\mathfrak{p}_0^m)S^{-1}_{\mathfrak{p}_0}C_{\Gamma}R}
    \end{equation}
    for $m\geq 1$.

    Observe that $C_{\Gamma'}R = \mathfrak{p}_0C_{\Gamma'}R$, if $\mathfrak{p}_0\not\geq \Gamma'$, from what follows the equality $S^{-1}_{\mathfrak{p}_0}C_{\Gamma'}R = \mathfrak{p}_0^mS^{-1}_{\mathfrak{p}_0}C_{\Gamma'}R$ for any $m\geq 1$.
    From \eqref{eqInclGamma} we obtain
    $$C_\Delta R = \varprojlim_{m}\:S^{-1}_{\mathfrak{p}_0}C_{\Gamma}R\big/\mathfrak{p}_0^mS^{-1}_{\mathfrak{p}_0}C_{\Gamma}R\longhookrightarrow \varprojlim_m\prod_{\Gamma'\in \Lambda}S^{-1}_{\mathfrak{p}_0}C_{\Gamma'}R\big/\mathfrak{p}_0^mS^{-1}_{\mathfrak{p}_0}C_{\Gamma'}R,$$
    and 
    \begin{eqnarray*}
        \varprojlim_m\prod_{\Gamma'\in \Lambda}S^{-1}_{\mathfrak{p}_0}C_{\Gamma'}R\big/\mathfrak{p}_0^mS^{-1}_{\mathfrak{p}_0}C_{\Gamma'}R &=& \prod_{\Gamma'\in \Lambda}\varprojlim_m\:S^{-1}_{\mathfrak{p}_0}C_{\Gamma'}R\big/\mathfrak{p}_0^mS^{-1}_{\mathfrak{p}_0}C_{\Gamma'}R\\
        &=& \prod_{\substack{\Gamma'\in \Lambda\\ \mathfrak{p}_0\geq \Gamma'}}\varprojlim_m\:S^{-1}_{\mathfrak{p}_0}C_{\Gamma'}R\big/\mathfrak{p}_0^mS^{-1}_{\mathfrak{p}_0}C_{\Gamma'}R\\
        &= &\prod_{\substack{\Delta'\in S(R, J)\\ \Delta\subset \Delta'}}C_{\Delta'}R,
    \end{eqnarray*}
    where the second equality follows from the above observation. 

    From \eqref{eqInterPrimaryn}, Artin--Rees lemma, and from the fact that limits commute with intersections, we obtain
    \begin{eqnarray*}
    C_\Delta R\cap \prod_{\substack{\Delta'\in S(R, J)\\ \Delta\subset \Delta'}}\mathfrak{n}C_{\Delta'}R &=& \varprojlim_m\:\frac{S^{-1}_{\mathfrak{p}_0}C_{\Gamma}R}{\mathfrak{p}_0^mS^{-1}_{\mathfrak{p}_0}C_{\Gamma}R}\cap \frac{\prod_{\Gamma'\in \Lambda}\mathfrak{n}S^{-1}_{\mathfrak{p}_0}C_{\Gamma'}R}{\prod_{\Gamma'\in \Lambda}(\mathfrak{n}\cap\mathfrak{p}_0^m)S^{-1}_{\mathfrak{p}_0}C_{\Gamma'}R} \\
    &=& \varprojlim_m\:\frac{S^{-1}_{\mathfrak{p}_0}C_{\Gamma}R}{\mathfrak{p}_0^mS^{-1}_{\mathfrak{p}_0}C_{\Gamma}R}\cap \frac{\mathfrak{n}\prod_{\Gamma'\in \Lambda}S^{-1}_{\mathfrak{p}_0}C_{\Gamma'}R}{(\mathfrak{n}\cap\mathfrak{p}_0^m)\prod_{\Gamma'\in \Lambda}S^{-1}_{\mathfrak{p}_0}C_{\Gamma'}R} \\
    &=& \varprojlim_m\: \frac{\mathfrak{n}S^{-1}_{\mathfrak{p}_0}C_{\Gamma}R}{(\mathfrak{n}\cap\mathfrak{p}_0^m)S^{-1}_{\mathfrak{p}_0}C_{\Gamma}R} = \mathfrak{n}C_\Delta R.
    \end{eqnarray*}
    \end{proof}

\subsection{Embeddings of adelic groups on schemes}\label{sect3.2} In this section we establish theorems on embeddings of adelic groups on thickenings of special type of excellent integral schemes. The following theorem is the local version of embeddings of adelic groups.
\begin{thm}\label{TheoremEmbeddingofAdelicGroups}
        Let $Y$ be an excellent integral Noetherian scheme and let $X$ be a closed subscheme which is locally defined by primary ideals. Assume that $Y$ is strongly biequidimensional.
        Let $\mathcal{F}$ be a flat quasicoherent sheaf on $X$ and let $I, J\subset \{0, 1, \ldots, \dim X\}$ be non-empty subsets such that $I\subset J$. Let $\Delta \in S(X, I)$. Then there is an inclusion
        $$\A_\Delta(X, \mathcal{F})\longhookrightarrow\prod_{\substack{\Delta'\in S(X, J)\\ \Delta\subset \Delta'}}\A_{\Delta'}(X, \mathcal{F}).$$
    \end{thm}
    \begin{proof} 
        Let $\Delta = (p_0, \ldots, p_n)$. Let $\Psi = \iota(\Delta) \in S(Y)$, where $\iota\colon X\longrightarrow Y$ is the closed embedding. Let $\mathrm{Spec}\: R\subset Y$ be a biequidimnesional affine neighborhood of $p_n$ of dimension $\dim Y$. Then, according to the theorem assumptions, $R$ is an excellent biequidimensional domain with $\dim R = \dim Y$, and  $X\cap \mathrm{Spec}\: R$ is given by a primary ideal $\mathfrak{n}$ of $R$ by lemma \ref{lemmaPrimarySheaf}. By remark \ref{remIrrSubschisBiequidim}, $X\cap \mathrm{Spec}\: R$ is a biequidimensional affine scheme of dimension $\dim Y - \mathrm{ht}\: \mathfrak{n} = \dim X$. We can consider $\Delta$ as an element of $S(R/\mathfrak{n}, I)$. 
        By lemma \ref{LemmaIandIZ}, $\Psi\in S(R, I + \mathrm{ht}\:\mathfrak{n})$. By \cite[Proposition 3.2.1]{Hu}, $\A_\Psi(Y, \mathcal{O}_Y) = C_\Psi R$. By theorem \ref{LocalFactorEmbedding},
        $$C_\Psi R \longhookrightarrow \prod_{\substack{\Psi'\in S(R, J + \mathrm{ht}\:\mathfrak{n})\\ \Psi\subset \Psi'}}C_{\Psi'}R, \quad C_\Psi R \cap \prod_{\substack{\Psi'\in S(R, J + \mathrm{ht}\:\mathfrak{n})\\ \Psi\subset \Psi'}}\mathfrak{n}C_{\Psi'}R = \mathfrak{n}C_\Psi R.$$
        The second equality is equivalent to the embedding
        $$C_\Psi(R/\mathfrak{n})\longhookrightarrow \prod_{\substack{\Psi'\in S(R, J + \mathrm{ht}\:\mathfrak{n})\\ \Psi\subset \Psi'}}C_{\Psi'}(R/\mathfrak{n}).$$
        If $\mathfrak{n}\not\geq \Psi'$, then $C_{\Psi'}R = \mathfrak{n}C_{\Psi'}R$, and thus $C_{\Psi'}(R/\mathfrak{n}) = 0$. 
        There is also a bijection between sets $\{\Psi'\in S(R, J + \mathrm{ht}\: \mathfrak{n})\mid \mathfrak{n} \geq \Psi'\}$ and $S(R/\mathfrak{n}, J)$. Hence, by lemma \ref{LemmaCompLocal}, we have an embedding
    $$C_{\Delta} (R/\mathfrak{n}) \longhookrightarrow \prod_{\substack{\Delta'\in S(R/\mathfrak{n}, J)\\ \Delta\subset \Delta'}}C_{\Delta'}(R/\mathfrak{n}).$$
    By assumption, $M = \mathcal{F}\big({\mathrm{Spec}(R/\mathfrak{n})}\big)$ is a flat $R/\mathfrak{n}$-module, and thus 
    $$C_\Delta (R/\mathfrak{n})\otimes_{R/\mathfrak{n}} M \longhookrightarrow \prod_{\substack{\Delta'\in S(R/\mathfrak{n}, J)\\ \Delta\subset \Delta'}}C_{\Delta'}(R/\mathfrak{n})\otimes_{R/\mathfrak{n}} M.$$
    By \cite[Proposition 3.2.1]{Hu}, $\A_\Delta(X, \mathcal{F}) = C_\Delta (R/\mathfrak{n})\otimes_{R/\mathfrak{n}} M$, $\A_{\Delta'}(X, \mathcal{F}) = C_{\Delta'}(R/\mathfrak{n})\otimes_{R/\mathfrak{n}} M$, from what the theorem follows.
    \end{proof}

    \begin{thm}\label{MainThmembeddingAdeles}
        Let $Y$ be an excellent integral Noetherian scheme and let $X$ be a closed subscheme which is locally defined by primary ideals. 
        Assume that $Y$ is strongly biequidimensional.
        Let $\mathcal{F}$ be a subsheaf of some flat quasicoherent sheaf on $X$ and let $I\subset J$ be subsets of $\{0, 1, \ldots, \dim X\}$. Then the map
        $$\varphi_{IJ}\colon\A_I(X, \mathcal{F})\longrightarrow\A_{J}(X, \mathcal{F})$$
        from definition \ref{definitionPhiIJ} is an embedding.
    \end{thm}
    \begin{proof}
        Since functors $\A_I(X, -)$, $\A_J(X, -)$ are left exact, we can assume that $\mathcal{F}$ is flat. Note that the case $I = \varnothing$ follows from proposition \ref{REmbeddsInCDeltaR} and from the flatness of $\mathcal{F}$. Thus we can assume $I \neq \varnothing$.

        By \cite[Proposition 2.1.4]{Hu},
        $$\A_I(X, \mathcal{F})\longhookrightarrow \prod_{\Delta\in S(X, I)}\A_\Delta(X, \mathcal{F}).$$
        By theorem \ref{TheoremEmbeddingofAdelicGroups},
        $$\A_\Delta(X, \mathcal{F})\longhookrightarrow \prod_{\substack{\Delta'\in S(X, J)\\ \Delta\subset \Delta'}}\A_{\Delta'}(X, \mathcal{F})$$
        for any $\Delta\in S(X, I)$. Therefore we have an embedding
        $$\Phi_{IJ}\colon\prod_{\Delta\in S(X, I)}\A_\Delta(X, \mathcal{F})\longhookrightarrow \prod_{\Delta'\in S(X, J)}\A_{\Delta'}(X, \mathcal{F}).$$
        Observe that $\Phi_{IJ}|_{\A_I(X, \mathcal{F})} = \varphi_{IJ}$ by construction, since both maps are induced by boundary maps.
        Thus the map 
        $$\varphi_{IJ}\colon \A_I(X, \mathcal{F})\longrightarrow\A_J(X, \mathcal{F})$$
        is injective.
    \end{proof}

    Let $f\colon Y \longrightarrow X$ be a surjective morphism of schemes of finite type over a field $\Bbbk$ and let $\mathcal{F}$ be a quasicoherent sheaf on $X$. By remark \ref{remGenPullback}, there is a pullback map
$$f^*\colon \A_I(X, \mathcal{F})\longrightarrow \A_I(Y, f^*\mathcal{F}).$$
More precisely, define $S(Y/X, I) = \{\Delta\in S(Y, I)\mid f(\Delta)\in S(X, I)\}\subset S(Y, I)$, and let $f^*$ be a composition of maps
$$\A_I(X, \mathcal{F})\overset{f^*(S(Y/X, I),\: S(X, I), \mathcal{F})}{\xrightarrow{\hspace*{3cm}}} \A\big(S(Y/X, I), f^*\mathcal{F}\big)\longrightarrow \A_I(Y, f^*\mathcal{F}),$$
where the first map is from definition \ref{PullbackDef} and the second map is from \cite[Proposition 2.1.5]{Hu}.  
\begin{prop}\label{InjPullback}
    Let $f\colon  Y\longrightarrow X$ be a proper surjective morphism of irreducible schemes of finite type over a perfect field $\Bbbk$ with $\dim X = \dim Y$. Assume that they are locally defined by primary ideals in integral schemes of finite type over $\Bbbk$. Also, assume that $f$ induces an inclusion $\mathcal{O}_{X, x_0}\hookrightarrow\mathcal{O}_{Y, y_0}$, where $x_0$, $y_0$ are generic points of $X$, $Y$ respectively. Let $\mathcal{F}$ be a subsheaf of some flat quasicoherent sheaf on $X$ and let $I\subset \{0, 1, \ldots, \dim X\}$. Then the pullback map
    $$f^*\colon \A_I(X, \mathcal{F})\longrightarrow \A_I(Y, f^*\mathcal{F})$$
    is injective.
\end{prop}
\begin{proof}
Since the functor $\A_I(X, -)$ is left exact, we can assume that $\mathcal{F}$ is flat.

Denote $d = \dim X = \dim Y$. For $\Delta \in S\big(X, (0, \ldots, d)\big)$ define $\Lambda_\Delta = \big\{\Psi \in S\big(Y, (0, \ldots, d)\big)\mid f(\Psi) = \Delta\big\}$. If $\Delta = (x_0, \ldots, x_d)$ and $\Psi = (y_0, \ldots, y_d)$, then the condition $f(\Psi) = \Delta$ implies 
that $y_i$ is a closed point in the fiber over $x_i$ for $0\leq i \leq d$.  Note that $\Lambda_\Delta$ is finite, since the fiber over $x_{i + 1}$ in $\overline{\{y_i\}}$ is zero-dimensional, otherwise $y_i$ lies in this fiber, which is a contradiction. Since $\mathcal{O}_{X, x_0}$ and $\mathcal{O}_{Y, y_0}$ are local Artinian rings, $\varprojlim_l \mathcal{O}_{X, x_0}/\mathfrak{m}_{x_0}^l = \mathcal{O}_{X, x_0}$ and $\varprojlim_l \mathcal{O}_{Y, y_0}/\mathfrak{m}_{y_0}^l = \mathcal{O}_{Y, y_0}$. By \cite[Proposition 3.1]{Ye2},
$$\prod_{\Psi\in \Lambda_\Delta}\A_{\Psi}(Y, \mathcal{O}_Y) = \mathcal{O}_{Y, y_0}\otimes_{\mathcal{O}_{X, x_0}}\A_\Delta(X, \mathcal{O}_X).$$
Since $\A_\Delta(X, \mathcal{O}_X)$ is flat over $\mathcal{O}_{X, x_0}$, we obtain 
$$\A_\Delta(X, \mathcal{O}_X) \longhookrightarrow \mathcal{O}_{Y, y_0}\otimes_{\mathcal{O}_{X, x_0}}\A_\Delta(X, \mathcal{O}_X) = \prod_{\Psi\in \Lambda_\Delta}\A_{\Psi}(Y, \mathcal{O}_Y).$$

Let $\mathrm{Spec}\: R$ be some affine neighborhood of $x_d$. Then $M = \mathcal{F}({\mathrm{Spec}\: R})$ is a flat $R$-module. By \cite[Proposition 3.2.1]{Hu}, $\A_\Delta(X, \mathcal{F}) = \A_\Delta(X, \mathcal{O}_X)\otimes_{R} M$. Similarly, $\A_\Psi(Y, f^*\mathcal{F}) = \A_\Psi(Y, \mathcal{O}_Y)\otimes_{R} M$. Since $M$ is flat, we obtain an inclusion
$$\A_\Delta(X, \mathcal{F}) \longhookrightarrow \prod_{\Psi\in \Lambda_\Delta}\A_{\Psi}(Y, f^*\mathcal{F}).$$
Since the composition of maps
$$ \A_I(X, \mathcal{F})\longhookrightarrow\prod\limits_{\Gamma\in S(X, I)}\A_\Gamma(X, \mathcal{F})\longhookrightarrow  \prod\limits_{\Delta\in S\left(X, (0, \ldots, d)\right)}\A_\Delta(X, \mathcal{F})\overset{f^*}{\longhookrightarrow} \prod\limits_{\Psi\in S\left(Y/X, (0, \ldots, d)\right)}\A_\Psi(Y, f^*\mathcal{F}),$$
where the second inclusion follows from theorem \ref{TheoremEmbeddingofAdelicGroups}, is injective and factors through $\A_I(Y, f^*\mathcal{F})$, we obtain the assertion.
\end{proof}

\section{Intersections of adelic groups with $I\cap J = I\setminus 0$ on a normal excellent scheme}\label{sect4}

\subsection{Local computations} The following proposition is a generalization of \cite[Theorem 11.5]{Ma1}.

\begin{prop}\label{IntersectionofHeight1}
    Let $R$ be a normal excellent domain and let $\Delta = (\mathfrak{p}_0, \ldots, \mathfrak{p}_n)$ be a flag of prime ideals or $\Delta = \varnothing$. Then 
    $$\bigcap_{\substack{\mathfrak{p}\subset R\\ \mathrm{ ht}\:\mathfrak{p} = 1}}R_\mathfrak{p}\cdot C_\Delta R = C_\Delta R,$$
    where the intersection is taken in $\mathrm{ Frac}(R)\cdot C_\Delta R$.
\end{prop}
\begin{proof}
The case $\Delta = \varnothing$ follows from \cite[Theorem 11.5]{Ma1}. Thus assume that $\Delta\neq \varnothing$. We can also assume that $\mathrm{ht}\: \mathfrak{p}_0 \geq 1$, otherwise $\mathfrak{p}_0 = (0)$ and $R_\mathfrak{p}\cdot C_\Delta R = C_\Delta R$ for any prime ideal $\mathfrak{p}$ with $\mathrm{ht}\:\mathfrak{p} = 1$, so the assertion is clear. 

Let 
$$\Lambda = \prod_{\mathrm{ht}\:\mathfrak{p} = 1} R\setminus \mathfrak{p}, \quad\Lambda_{f} = \big\{{\bf f} = (f_\mathfrak{p})_{\mathrm{ht}\:\mathfrak{p} = 1}\in \Lambda \mid \sum_{{\mathrm{ht}\:\mathfrak{p} = 1}} f^{-1}_\mathfrak{p}R\: \text{ is a finitely generated }R\text{-module}\big\}.$$
Define partial order on sets $\Lambda$ and $\Lambda_f$:
$${\bf f} \leq {\bf g} \iff f^{-1}_\mathfrak{p}R \subset g^{-1}_\mathfrak{p}R \quad \forall \mathfrak{p},\: \mathrm{ht}\:\mathfrak{p} = 1.$$
Note that $\Lambda$ and $\Lambda_f$ are directed sets according to this partial order. Observe that
\begin{eqnarray*}
        \bigcap_{\mathrm{ ht}\:\mathfrak{p} = 1}R_\mathfrak{p}\cdot C_\Delta R = \varinjlim_{{\bf f}\in \Lambda}\:\bigg(\bigcap_{\mathrm{ ht}\:\mathfrak{p} = 1}f_\mathfrak{p}^{-1}C_\Delta R\bigg),
    \end{eqnarray*}
    since $R_\mathfrak{p} = \varinjlim\limits_{f_{\mathfrak{p}}\in R\setminus \mathfrak{p}}f^{-1}_\mathfrak{p}R$.
    Note that every element $s\in R\setminus 0$ lies only in a finite number of prime ideals $\mathfrak{q}_1,\ldots, \mathfrak{q}_m$ with $\mathrm{ ht}\:\mathfrak{q}_i = 1$. Thus if $x/s\in\bigcap_{\mathrm{ ht}\:\mathfrak{p} = 1}f_\mathfrak{p}^{-1}C_\Delta R$ for some $x\in C_\Delta R$, then 
    $$x/s\in \bigcap_{i = 1}^mf_{\mathfrak{q}_i}^{-1} C_\Delta R\cap \bigcap_{\substack{\mathrm{ ht}\:\mathfrak{p} = 1\\ \mathfrak{p}\neq \mathfrak{q}_i}}s^{-1}C_\Delta R,$$
    and $R$-module $\sum_{i = 1}^mf_{\mathfrak{q}_i}^{-1}R+s^{-1} R$ is finitely generated. Hence
    $$\varinjlim_{{\bf f}\in \Lambda}\:\bigg(\bigcap_{\mathrm{ ht}\:\mathfrak{p} = 1}f_\mathfrak{p}^{-1}C_\Delta R\bigg) = \varinjlim_{{\bf f}\in \Lambda_f}\:\bigg(\bigcap_{\mathrm{ ht}\:\mathfrak{p} = 1}f_\mathfrak{p}^{-1}C_\Delta R\bigg).$$
    Since the map $C_{\mathfrak{p}_0}S^{-1}_{\mathfrak{p}_0}R\rightarrow C_\Delta R$ is flat by proposition \ref{Cp0CDeltaflat} and $\mathfrak{p}_0C_\Delta R\subset \mathrm{ rad}(C_\Delta R)$, because $C_\Delta R$ is $\mathfrak{p}_0C_\Delta R$-complete, we obtain that $C_\Delta R$ is intersection flat over $C_{\mathfrak{p}_0}S^{-1}_{\mathfrak{p}_0}R$ by \cite[Proposition 5.7(e)]{HJ}. Thus 
    $$\bigcap_{\mathrm{ ht}\:\mathfrak{p} = 1}f_\mathfrak{p}^{-1}C_\Delta R = \left(\bigcap_{\mathrm{ ht}\:\mathfrak{p} = 1}f_\mathfrak{p}^{-1}C_{\mathfrak{p}_0}S^{-1}_{\mathfrak{p}_0}R\right)\cdot C_\Delta R$$
    for any ${\bf f} = (f_{\mathfrak{p}})_{\mathrm{ht}\: \mathfrak{p} = 1}\in \Lambda_f$, since $\sum_{\mathrm{ht}\: \mathfrak{p} = 1}f^{-1}_{\mathfrak{p}}C_{\mathfrak{p}_0}S^{-1}_{\mathfrak{p}_0}R$ is a finitely generated $C_{\mathfrak{p}_0}S^{-1}_{\mathfrak{p}_0}R$-module. Hence it is sufficient to prove the equality 
    $$\bigcap_{\mathrm{ ht}\:\mathfrak{p} = 1}f_\mathfrak{p}^{-1}C_{\mathfrak{p}_0}S^{-1}_{\mathfrak{p}_0}R = C_{\mathfrak{p}_0}S^{-1}_{\mathfrak{p}_0}R.$$

    Let ${\bf f} = (f_{\mathfrak{p}})_{\mathrm{ht}\: \mathfrak{p} = 1}\in \Lambda_f$, and assume that $x/s\in \bigcap_{\mathrm{ ht}\:\mathfrak{p} = 1}f_\mathfrak{p}^{-1}C_{\mathfrak{p}_0}S^{-1}_{\mathfrak{p}_0}R$, where $x\in C_{\mathfrak{p}_0}S^{-1}_{\mathfrak{p}_0}R$ and $s\in R\setminus 0$. Let $\mathfrak{P}$ be an ideal of height $1$ in $C_{\mathfrak{p}_0}S^{-1}_{\mathfrak{p}_0}R$. Since $C_{\mathfrak{p}_0}S^{-1}_{\mathfrak{p}_0}R$ is flat over $R$, we have that $\mathrm{ht}\:(\mathfrak{P}\cap R) \leq \mathrm{ht}\:\mathfrak{P} = 1$ by the going-down property for flat ring morphisms (see \cite[Theorem 9.5]{Ma1}). Define $\hat{f}_\mathfrak{P} = f_{\mathfrak{P}\cap R}$, if $\mathrm{ht}\:(\mathfrak{P}\cap R) = 1$, and $\hat{f}_\mathfrak{P} = s$, if $\mathrm{ht}\:(\mathfrak{P}\cap R) = 0$. By construction, $\hat{f}_\mathfrak{P}\notin \mathfrak{P}$ and
    $$x/s \in \bigcap_{\substack{\mathfrak{P}\subset C_{\mathfrak{p}_0}S^{-1}_{\mathfrak{p}_0}R\\ \mathrm{ ht}\:\mathfrak{P} = 1}}\hat{f}_\mathfrak{P}^{-1}C_{\mathfrak{p}_0}S^{-1}_{\mathfrak{p}_0}R$$
    Since $S^{-1}_{\mathfrak{p}_0}R$ is a normal excellent domain, we obtain that $C_{\mathfrak{p}_0}S^{-1}_{\mathfrak{p}_0}R$ is a normal domain by \cite[Theorem 79]{Ma2}. Thus
    $$\bigcap_{\substack{\mathfrak{P}\subset C_{\mathfrak{p}_0}S^{-1}_{\mathfrak{p}_0}R\\ \mathrm{ ht}\:\mathfrak{P} = 1}}\hat{f}_\mathfrak{P}^{-1}C_{\mathfrak{p}_0}S^{-1}_{\mathfrak{p}_0}R = C_{\mathfrak{p}_0}S^{-1}_{\mathfrak{p}_0}R$$
    by \cite[Theorem 11.5]{Ma1}. Hence $x/s \in C_{\mathfrak{p}_0}S^{-1}_{\mathfrak{p}_0}R$, and we obtain an inclusion
    $$\bigcap_{\mathrm{ ht}\:\mathfrak{p} = 1}f_\mathfrak{p}^{-1}C_{\mathfrak{p}_0}S^{-1}_{\mathfrak{p}_0}R \subset C_{\mathfrak{p}_0}S^{-1}_{\mathfrak{p}_0}R.$$
    Since the opposite inclusion is clear, the assertion follows.
\end{proof}

\subsection{Intersections of adelic groups with $I\cap J= I\setminus 0$}
\begin{lemma}\label{ProponInter}
    Let $X$ be an integral Noetherian scheme and let $I\subset \{0, 1, \ldots, \dim X\}$ be a finite subset such that $0\in I$ and $I\setminus 0\neq \varnothing$. Then 
    $$\A_I(X, \mathcal{O}_X)\cap \prod_{\Delta\in S(X, I)}\A_{\Delta\setminus \eta}(X, \mathcal{O}_X) = \A_{I\setminus 0}(X, \mathcal{O}_X),$$
    where $\eta$ is the generic point of $X$.
\end{lemma}
\begin{proof}
    Follows from lemma \ref{IandI0} and lemma \ref{LocalCondAdeleSubsheaf}.
\end{proof}

The next theorem is a generalization of \cite[Theorem 1(i)]{BG}.

\begin{thm}\label{TheoremIcapJI0}
    Let $X$ be a normal integral excellent Noetherian scheme, let $I, J\subset \{0, 1, \ldots, \dim X\}$ be subsets with $I\cap J = I\setminus 0$, and let $\mathcal{F}$ be a flat quasicoherent sheaf on $X$. 
    Assume that $X$ is strongly biequidimensional.
    Then
    $$\A_I(X, \mathcal{F})\cap \A_J(X, \mathcal{F}) = \A_{I\setminus 0}(X, \mathcal{F}).$$
\end{thm}
\begin{proof}
    First, let us assume that $\mathcal{F} = \mathcal{O}_X$.
    If $0\notin I$ or $I = (0)$, $J = \varnothing$, then the assertion is clear. If $I = (0)$ and $J\neq \varnothing$, then 
    \begin{equation}\label{eqA0capAJ}
    \A_0(X, \mathcal{O}_X)\cap \A_J(X, \mathcal{O}_X) = \big\{s\in  \A_0(X, \mathcal{O}_X) = \mathcal{O}_{X, \eta}\mid \forall\:\Gamma\in S(X, J)\colon s\in \A_\Gamma(X, \mathcal{O}_X)\big\},
    \end{equation}
    where $\eta$ is the generic point of $X$. Let $\Gamma = (p_0, \ldots)\in S(X, J)$. Then the map 
    \begin{equation}\label{eqFaithFlatp0}
    \mathcal{O}_{X, p_0}\longrightarrow \A_\Gamma(X, \mathcal{O}_X)
    \end{equation}
    is faithfully flat by proposition \ref{Cp0CDeltaflat} and the fact that the map $\mathcal{O}_{X, p_0}\rightarrow\widehat{\mathcal{O}}_{X, p_0} = \A_{(p_0)}(X, \mathcal{O}_X)$ is faithfully flat. Thus if $f/g\in \mathcal{O}_{X, \eta}\cap \A_\Gamma(X, \mathcal{O}_X)$, where $f, g\in \mathcal{O}_{X, p_0}$, then 
    $$f\in \mathcal{O}_{X, p_0}\cap g\A_\Gamma(X, \mathcal{O}_X) = g\mathcal{O}_{X, p_0}$$
    by the faithfull flatness of the map \eqref{eqFaithFlatp0}, i.e. $f/g \in \mathcal{O}_{X, p_0}$, from what follows the equality 
    \begin{equation}\label{eqLocalIntersection}
    \mathcal{O}_{X, \eta}\cap \A_\Gamma(X, \mathcal{O}_X) = \mathcal{O}_{X, p_0}.
    \end{equation}
    Let $U$ be a biequidimensional affine open subscheme. 
    Then 
    $$\bigcap_{\substack{(p_0, \ldots)\in S(X, J)\\ p_0\in U}}\mathcal{O}_{X, p_0} \subset \bigcap_{\substack{p \in U \\ \mathrm{codim}\: p = 1}} \mathcal{O}_{X, p} = \mathcal{O}_X(U),$$
    where the inclusion follows from the fact that $U$ is biequidimensional (see remark \ref{remBiequidim})  
    and the equality follows from \cite[Theorem 11.5]{Ma1}, since $U$ is normal. Thus, from \eqref{eqA0capAJ} and \eqref{eqLocalIntersection},
    $$\A_0(X, \mathcal{O}_X)\cap \A_J(X, \mathcal{O}_X) = \bigcap_{(p_0, \ldots)\in S(X, J)} \mathcal{O}_{X, p_0}\subset \bigcap_{U\text{ is biequid. affine}}\mathcal{O}_X(U) = H^0(X, \mathcal{O}_X),$$
    where the inclusion and the second equality follow from the strong biequidimensionality of $X$ (in particular, $X = \bigcup_{U\text{ is biequid. affine}}U $).
    Since the opposite inclusion is clear, the assertion in the case $I = (0)$ and $J\neq \varnothing$ follows. 

    Now consider the case $0\in I$ and $I\setminus 0 \neq \varnothing$. It follows that $J\neq \varnothing$ by assumptions. 
    Let $\Delta = (\eta, \ldots, p_m)\in S(X, I)$. Let $U = \mathrm{Spec}\:R$ be a biequidimensional affine neighborhood of $p_m$ of dimension $\dim X$. Then $\Delta$ can be considered as a flag of prime ideals in $S(R, I)$. Let us prove the equality
    \begin{equation}\label{eqCdeltacapCgamma}
    C_\Delta R \cap \prod_{\substack{\Gamma\in S(R, J)\\ \Delta\setminus 0 \subset \Gamma}}C_\Gamma R = C_{\Delta\setminus 0}R,
    \end{equation}
    where the intersection is taken in $\prod_{\Gamma'\in S(R, 0\cup J)}C_{\Gamma'} R$. This intersection is well-defined by theorem \ref{LocalFactorEmbedding}. Let $\mathfrak{q}$ be an arbitrary prime ideal in $R$ of height $1$ such that $\mathfrak{q}\geq \Delta\setminus 0$. By lemma \ref{aSCDeltaR} and theorem \ref{LocalFactorEmbedding}, 
    \begin{eqnarray*}
    S^{-1}_\mathfrak{q}C_{\Delta\setminus 0}R \cap \prod_{\substack{\Gamma\in S(R, J)\\ \Delta\setminus 0 \subset \Gamma}}\mathfrak{q}^nS^{-1}_\mathfrak{q}C_\Gamma R &=&
    S^{-1}_\mathfrak{q}C_{\Delta\setminus 0}R \cap \prod_{\substack{\Gamma\in S(R, J)\\ \Delta\setminus 0 \subset \Gamma}}\mathfrak{q}^{(n)}S^{-1}_\mathfrak{q}C_\Gamma R \\
    &=& \mathfrak{q}^{(n)}S^{-1}_\mathfrak{q}C_{\Delta\setminus 0}R \\
    &=& \mathfrak{q}^nS^{-1}_\mathfrak{q}C_{\Delta\setminus 0}R.
    \end{eqnarray*}
    Since $R$ is normal, $\mathfrak{q}R_\mathfrak{q} = (t)$ for some $t\in R_\mathfrak{q}$, and thus $C_\Delta R = \bigcup_{n}t^{-n}S^{-1}_\mathfrak{q}C_{\Delta\setminus 0} R$. This means that if 
    $$x\in C_\Delta R \cap \prod_{\substack{\Gamma\in S(R, J)\\ \Delta\setminus 0 \subset \Gamma}}C_\Gamma R \subset C_\Delta R \cap \prod_{\substack{\Gamma\in S(R, J)\\ \Delta\setminus 0 \subset \Gamma}} S^{-1}_\mathfrak{q}C_\Gamma R,$$
    then there exists $n\geq 0$ such that
    $$x\in t^{-n}S^{-1}_\mathfrak{q}C_{\Delta\setminus 0}R \cap \prod_{\substack{\Gamma\in S(R, J)\\ \Delta\setminus 0 \subset \Gamma}}S^{-1}_\mathfrak{q}C_\Gamma R.$$
    In other words,
    $$t^nx \in S^{-1}_\mathfrak{q}C_{\Delta\setminus 0}R \cap \prod_{\substack{\Gamma\in S(R, J)\\ \Delta\setminus 0 \subset \Gamma}}t^nS^{-1}_\mathfrak{q}C_\Gamma R = t^nS^{-1}_\mathfrak{q}C_{\Delta\setminus 0}R,$$
    and we obtain that $x \in S^{-1}_\mathfrak{q}C_{\Delta\setminus 0}R$. Therefore 
    $$C_\Delta R\cap \prod_{\substack{\Gamma\in S(R, J)\\ \Delta\setminus 0\subset \Gamma}}C_\Gamma R \subset \bigcap_{\substack{\mathrm{ht}\: \mathfrak{q = 1} \\ \mathfrak{q}\geq \Delta\setminus 0}}R_\mathfrak{q}\cdot C_{\Delta\setminus 0}R.$$
    Note that $R_\mathfrak{q}\cdot C_{\Delta\setminus 0}R = C_\Delta R$ for a prime ideal $\mathfrak{q}$ of height $1$ such that $\mathfrak{q}\not\geq \Delta\setminus 0$, because if $\Delta\setminus 0 = (\mathfrak{r}, \ldots)$, then $R_\mathfrak{q}\cdot R_\mathfrak{r} = \mathrm{Frac}(R)$. Thus 
    $$\bigcap_{\substack{\mathrm{ht}\: \mathfrak{q = 1} \\ \mathfrak{q}\geq \Delta\setminus 0}}R_\mathfrak{q}\cdot C_{\Delta\setminus 0}R = \bigcap_{\mathrm{ht}\: \mathfrak{q = 1} }R_\mathfrak{q}\cdot C_{\Delta\setminus 0}R = C_{\Delta \setminus 0} R,$$
    where the second equality follows from proposition \ref{IntersectionofHeight1}. Hence we obtain an inclusion
    $$C_\Delta R\cap \prod_{\substack{\Gamma\in S(R, J)\\ \Delta\setminus 0\subset \Gamma}}C_\Gamma R  \subset C_{\Delta \setminus 0} R.$$
    Since the reverse inclusion is clear, we obtain the desired equality. 
    
    We have
    \begin{eqnarray*}
        \A_I(X, \mathcal{O}_X)\cap \A_J(X, \mathcal{O}_X) &\subset&  \prod_{\Delta\in S(X, I)}\A_\Delta(X, \mathcal{O}_X)\cap \prod_{\Gamma\in S(X, J)}\A_\Gamma(X, \mathcal{O}_X) \\
        &\subset& \prod_{\Delta\in S(X, I)}\left(\A_\Delta(X, \mathcal{O}_X)\cap \prod_{\substack{\Gamma\in S(X, J)\\ \Delta\setminus \eta\subset \Gamma}}\A_\Gamma (X, \mathcal{O}_X)\right)\\
        &=& \prod_{\Delta\in S(X, I)}\A_{\Delta\setminus \eta}(X, \mathcal{O}_X),
    \end{eqnarray*}
    where the equality follows from \eqref{eqCdeltacapCgamma}. Thus
    $$\A_I(X, \mathcal{O}_X)\cap \A_J(X, \mathcal{O}_X) \subset \A_I(X, \mathcal{O}_X)\cap \prod_{\Delta\in S(X, I)}\A_{\Delta\setminus \eta}(X, \mathcal{O}_X) = \A_{I\setminus 0}(X, \mathcal{O}_X),$$
    where the equality follows from lemma \ref{ProponInter}. Since the opposite inclusion is clear, we obtain the equality
    \begin{equation}\label{eqAIcapAJ}
    \A_I(X, \mathcal{O}_X)\cap \A_J(X, \mathcal{O}_X) = \A_{I\setminus 0}(X, \mathcal{O}_X).
    \end{equation}

    Now consider the case of an arbitrary flat quasicoherent sheaf $\mathcal{F}$. Since $X$ is irreducible, every affine open subscheme $U$ is irreducible. It is also normal, excellent, and Noetherian. Thus if $U$ is biequidimensional of dimension $\dim X$, then we have (see \eqref{eqAIcapAJ})
    $$\A_I(U, \mathcal{O}_U)\cap \A_J(U, \mathcal{O}_U) = \A_{I\setminus 0}(U, \mathcal{O}_U).$$
    Since $\mathcal{F}$ is flat, $\mathcal{F}(U)$ is a flat $\mathcal{O}_U(U)$-module. Hence
    $$\big(\A_I(U, \mathcal{O}_U)\otimes_{\mathcal{O}_U(U)}\mathcal{F}(U)\big)\cap \big(\A_J(U, \mathcal{O}_U)\otimes_{\mathcal{O}_U(U)}\mathcal{F}(U)\big) = \A_{I\setminus 0}(U, \mathcal{O}_U)\otimes_{\mathcal{O}_U(U)}\mathcal{F}(U),$$
    and from proposition \ref{AffineOxtimesF} we obtain
    \begin{equation}\label{eqUAIcapAJ}
    \A_I(U, \mathcal{F}|_U)\cap \A_J(U, \mathcal{F}|_U) = \A_{I\setminus 0}(U, \mathcal{F}|_U).
    \end{equation}
    Since $X$ is Noetherian and strongly biequidimensional, there is a finite cover $X = \bigcup_{i = 1}^n U_i$, where $U_i$ is a biequidimensional affine open subscheme of dimension $\dim X$. Thus $S(X, I) = \bigcup_{i = 1}^nS(U_i, I)$. From this and from equation \eqref{eqUAIcapAJ} we obtain the desired equality by \cite[Lemma 3.3.1]{Hu}.
\end{proof}

\section{Intersections of adelic groups on algebraic varieties}\label{sect5}

In this section we examine limits of restrictions to closed subschemes of cohomology groups and adeles. As a consequence, we obtain a theorem on intersections of adelic groups on a normal projective surface. We also calculate cohomology groups of a curtailed adelic complex. Hence we prove a theorem on intersections of adelic groups on a three-dimensional variety.

\subsection{Limits of cohomology groups}
Note that if $Z\hookrightarrow X$ is a closed embedding of Noetherian schemes, then there exists a restriction map on cohomology groups $H^j(X, \mathcal{F})\rightarrow H^j(Z, \mathcal{F}|_Z)$ for any quasicoherent sheaf $\mathcal{F}$ on $X$ and any $j\geq 0$. Recall that $\mathscr{E}$ is an index category of locally equidimensional closed subschemes such that irreducible components have the same codimension in $X$, see definition \ref{defIndCatE}.

\begin{thm}\label{Hjlim}
    Let $X$ be an irreducible Cohen--Macaulay projective scheme over a field $\Bbbk$, let $\mathcal{F}$ be a locally free sheaf on $X$, and let $\mathscr{C}$ be some index category of closed subschemes of $X$. Then there is a natural map
    $$H^k(X, \mathcal{F})\overset{\iota}{\longrightarrow} \lim_{Z\in \mathscr{C}}H^k(Z, \mathcal{F}|_Z).$$
    Then $\iota$ is an isomorphism, if one of the following is satisfied:
    \vspace{0.1cm}

    $1)$ $0\leq k < \dim X - 1$ and $\mathscr{C} = \mathscr{E}$,
    \vspace{0.1cm}

    $2)$ $k = 0$ and $\mathscr{C} = \mathscr{P}_{i, j}$ for $0\leq i < j \leq \dim X$,
    \vspace{0.1cm}

    $3)$ $X$ is regular, $k = 0$ and $\mathscr{C} = \mathscr{R}^{pow}_{i, j}$ for $0\leq i < j \leq \dim X$.
\end{thm}
\begin{proof}
    Let $\mathcal{L}$ be a very ample invertible sheaf on $X$ such that $H^k(X, \mathcal{F}\otimes_{\mathcal{O}_X}\mathcal{L}^{\otimes -n}) = 0$ for any $n> 0$ and $0\leq k < \dim X$. It exists by Serre duality and Serre's vanishing theorem.
    Denote $\mathcal{F}\otimes_{\mathcal{O}_X}\mathcal{L}^{\otimes n}$ by $\mathcal{F}(n)$ for $n\in \Z$. By a {\it hypersurface section}, we mean a zero scheme of a non-zero section of the sheaf $\mathcal{L}^{\otimes n}$ for some $n > 0$. A hypersurface section is a Cohen--Macaulay projective scheme by \cite[Ch. IV, Prop. 14]{Se}. Moreover, it is locally equidimensional and every irreducible component has codimension $1$ in $X$.

    Define an {\it admissible} hypersurface section as follows. First, this is a hypersurface section. In case $1$, we do not require anything.
    In case $2$, we additionally require that this hypersurface section is irreducible, and in case $3$, we additionally require that this hypersurface section is a regular variety. 
    
    Note that for any finite set $x_1, \ldots, x_m$ of closed point of $X$ there exists an admissible hypersurface section $D$ containing $x_1, \ldots, x_m$. The case $1$ is clear. In case $2$, this follows from \cite[Lemma 3.5]{GK} for the infinite field $\Bbbk$ and from \cite[Theorem 5.7]{GK} for the finite field $\Bbbk$, and in case $3$, this follows from \cite[Theorem 3.6]{GK} for the infinite field $\Bbbk$ and from \cite[Corollary 4.7]{GK}, \cite[Theorem 5.7]{GK} for the finite field $\Bbbk$. 

    Let $D_1, \ldots, D_s$ be hypersurface sections such that $\mathrm{codim}_XD_1\cap\ldots\cap D_l = l$ for $1\leq l\leq s$. Denote $D = D_s$ and $Y = D_1\cap\ldots\cap D_{s - 1}$, if $s > 1$, and $Y = X$, if $s = 1$. Assume that $D$ is the zero scheme of a non-zero section of the sheaf $\mathcal{L}^{\otimes m}$.
    First, let us prove the equality
    \begin{equation}\label{eqHkzero}
    H^k\big(Y\cap D, \mathcal{F}(-n)|_{Y\cap D}\big) = 0, \quad 0\leq k < \dim X - s
    \end{equation}
    for any $n > 0$.  Assume by induction that $H^k(Y, \mathcal{F}(-n)|_Y) = 0$ for $0\leq k < \dim X - s + 1$ and any $n > 0$. If $s = 1$, i.e. $Y = X$, then this follows from our assumptions on the sheaf $\mathcal{L}$. Fix $n > 0$. Then there is an exact sequence of sheaves
    $$0\longrightarrow \mathcal{F}(-n - m)|_Y\longrightarrow \mathcal{F}(-n)|_Y\longrightarrow \mathcal{F}(-n)|_{Y\cap D}\longrightarrow 0.$$
    Since $H^k(Y, \mathcal{F}(-n)|_Y) = H^k(Y, \mathcal{F}(-n-m)|_Y) = 0$ by the induction hypothesis, we obtain from the long cohomology sequence that $H^k\big(Y\cap D, \mathcal{F}(-n)|_{Y\cap D}\big) = 0$ for $0\leq k < \dim X - s$.

    Now let us prove that 
    \begin{eqnarray}
        H^k(Y, \mathcal{F}|_Y)&\overset{\sim}{\longrightarrow}& H^k(Y\cap D, \mathcal{F}|_{Y\cap D}), \quad 0\leq k < \dim X - s, \label{eqIsoHk} \\
        H^{\dim X - s}(Y, \mathcal{F}|_Y)&\longhookrightarrow& H^{\dim X - s}(Y\cap D, \mathcal{F}|_{Y\cap D}). \label{eqEmbedHdimXs}
    \end{eqnarray}
    There is an exact sequence of sheaves 
    $$0\longrightarrow \mathcal{F}(- m)|_Y\longrightarrow \mathcal{F}|_Y\longrightarrow \mathcal{F}|_{Y\cap D}\longrightarrow 0,$$
    and since $H^k(Y, \mathcal{F}(-m)|_Y) = 0$ for $0\leq k < \dim X - s + 1$ by assumptions on the sheaf $\mathcal{L}$, if $s = 1$, and by \eqref{eqHkzero}, if $s > 1$, we obtain \eqref{eqIsoHk} and \eqref{eqEmbedHdimXs} from the long cohomology sequence.

    Let us prove the assertion in case $1$. Let $(a_Z)_{Z\in \mathscr{E}}\in \lim_{Z\in \mathscr{E}}H^k(Z, \mathcal{F}|_Z)$, where $0\leq k < \dim X - 1$. Let $D$ be any hypersurface section. Then $H^k(X, \mathcal{F})\cong H^k(D, \mathcal{F}|_D)$. Let $a\in H^k(X, \mathcal{F})$ be an element such that $a|_D = a_D$. If $E$ is another hypersurface section such that $\mathrm{codim}_XD\cap E = 2$, then 
    $$(a|_E - a_E)|_{D\cap E} = a|_{D\cap E} - a_E|_{D\cap E} = a_D|_{D\cap E} - a_E|_{D\cap E} = a_{D\cap E} - a_{D\cap E} = 0.$$
     Since $H^k(E, \mathcal{F}|_E) \longhookrightarrow H^k(D\cap E, \mathcal{F}|_{D\cap E})$ by \eqref{eqIsoHk} and \eqref{eqEmbedHdimXs} for $s = 2$, we obtain that $a|_E = a_E$. If $E$ is an arbitrary hypersurface section, then there exists a hypersurface section $E'$ such that $\mathrm{codim}_X E\cap E' = \mathrm{codim}_X D\cap E' = 2$. Then $a|_{E'} = a_{E'}$ and, by the previous argument for $E$ and $E'$, we obtain $a|_E = a_E$.
     
     If $Z\in \mathscr{E}$, then there exists a hypersurface section $D'$ such that $Z\longhookrightarrow D'$. Therefore
     $$a|_Z = a_{D'}|_Z = a_Z.$$
     Hence we have constructed a map
     $$\lim_{Z\in \mathscr{E}}H^k(Z, \mathcal{F}|_Z)\longhookrightarrow H^k(X, \mathcal{F}),$$
     which is the inverse to $\iota$. The injectivity of this map follows from \eqref{eqIsoHk} and \eqref{eqEmbedHdimXs} for $s = 1$. Thus the assertion in case $1$ is proven.

     To prove the remaining cases, we will proceed by induction on $\dim X$. If $\dim X = 0$, there is nothing to prove. If $i = 0$, then $X\in \mathscr{C}$ by assumptions, so $\iota$ is an isomorphism in this case. If $\dim X = 1$, then $i = 0$ and $j = 1$, and $\iota$ is an isomorphism by the previous observation. Now assume that $\dim X > 1$ and $i > 0$. It follows that $j > 1$.
     
     Let $(a_Z)_{Z\in \mathscr{C}}\in \lim_{Z\in \mathscr{C}}H^0(Z, \mathcal{F}|_Z)$. Let $D$ be an admissible hypersurface section. Denote $\mathscr{C}_D = \mathscr{P}_{i - 1, j - 1}(D)$ in case $2$ and $\mathscr{C}_D = \mathscr{R}^{pow}_{i - 1, j - 1}(D)$ in case $3$. By induction hypothesis,
     $$H^0(D, \mathcal{F}|_D)\overset{\sim}{\longrightarrow}\lim_{Y\in \mathscr{C}_D}H^0(Y, \mathcal{F}|_Y).$$
     Let $Y\in \mathscr{C}_D$. If $Y'$ is a power thickening of $Y_{red}$ in $X$, i.e. $Y'\in \mathscr{C}$, such that $Y\longhookrightarrow Y'$, then there is a restriction map $H^0(Y', \mathcal{F}|_{Y'})\longrightarrow H^0(Y, \mathcal{F}|_Y).$
     Therefore, by means of these maps, we obtain a well-defined map
     \begin{equation}\label{eqaD}
     \lim_{Z\in \mathscr{C}}H^0(Z, \mathcal{F}|_Z)\longrightarrow \lim_{Y\in \mathscr{C}_D}H^0(Y, \mathcal{F}|_Y)\cong H^0(D, \mathcal{F}|_D).
     \end{equation}
    Let $D_{[m]}$ be the $m$-th power thickening of $D$. Then $D_{[m]}$ is also a hypersurface section. By \eqref{eqIsoHk},
    \begin{equation}\label{eqisoPowThick}
    H^0(X, \mathcal{F})\cong H^0(D, \mathcal{F}|_D), \quad H^0(X, \mathcal{F})\cong H^0(D_{[m]}, \mathcal{F}|_{D_{[m]}}),
    \end{equation}
     from what follows that the map \eqref{eqaD} can be continued to a well-defined map
     \begin{equation}\label{eqCD}
     \lim_{Z\in \mathscr{C}}H^0(Z, \mathcal{F}|_Z) \longrightarrow \lim_{Z\in \mathscr{D}}H^0(Z, \mathcal{F}|_Z),
     \end{equation}
     where $\mathrm{Ob}(\mathscr{D}) =\mathrm{Ob}(\mathscr{C})\cup \{\text{power thickenings of admissible hypersurface sections}\}$ and morphisms in $\mathscr{D}$ are closed embeddings. By construction, the map \eqref{eqCD} is injective. Hence it follows that \eqref{eqCD} is an isomorphism, since the left inverse to \eqref{eqCD} is given by the restriction map induced by the inclusion $\mathscr{C}\hookrightarrow \mathscr{D}$.

     Let $a\in H^0(X, \mathcal{F})$ be such that $a|_D = a_D$. This $a$ exists and is unique by \eqref{eqIsoHk}. Then $a|_{D_{[m]}} = a_{D_{[m]}}$ for any $m\geq 1$ by \eqref{eqisoPowThick}.
     If $E$ is another admissible hypersurface section such that $\mathrm{codim}_X D\cap E = 2$ and $D\cap E$ is irreducible, then there exist admissible hypersurface sections $E = D_1, \ldots, D_{j - 1}$ such that $\mathrm{codim}_X D\cap D_1\cap\ldots\cap D_k = k + 1$ for $1\leq k\leq j - 1$. 
     If $j < \dim X$, assume that $D\cap D_1\cap\ldots\cap D_k$ is irreducible for $1\leq k\leq j - 1$. 
     Such hypersurface sections exist by \cite[Lemma 3.5]{GK} for the infinite field $\Bbbk$ and by \cite[Theorem 5.7]{GK} for the finite field $\Bbbk$. Let $Y = D\cap D_1\cap\ldots \cap D_{j - 1}$. Let $Y'$ be the $m$-th power thickening of $Y_{red}$ in $X$ such that $Y\longhookrightarrow Y'$. By construction, if $j < \dim X$, then
     $Y'\in \mathscr{D}$, and if $j = \dim X$, then $Y'$ is a disjoint union of power thickenings of closed points, i.e. elements of $\mathscr{D}$. Therefore in both cases, $a_{D_{[m]}}|_{Y'} = a_{E_{[m]}}|_{Y'}$.
     Then, since $a_D = a_{D_{[m]}}|_D$, $a_E = a_{E_{[m]}}|_E$, and $Y'\longhookrightarrow D_{[m]}, E_{[m]}$, we have
     $$(a|_E - a_E)|_Y = (a|_D)|_Y - a_E|_Y = (a|_{D_{[m]}})|_Y - a_{E_{[m]}}|_Y = a_{D_{[m]}}|_Y - a_{E_{[m]}}|_Y = (a_{D_{[m]}}|_{Y'} - a_{E_{[m]}}|_{Y'})|_Y = 0.$$
     From \eqref{eqIsoHk} and \eqref{eqEmbedHdimXs} we obtain that $a|_E = a_E$. If $E$ is an arbitrary admissible hypersurface section, then there exists an admissible hypersurface section $E'$ such that $\mathrm{codim}_X E\cap E' = \mathrm{codim}_X D\cap E' = 2$ and both intersections are irreducible (for example, one can take a suitable admissible hypersurface section passing through two closed points on $D$ and $E$). Then $a|_{E'} = a_{E'}$ and, by the previous argument for $E$ and $E'$, we obtain $a|_E = a_E$.

     Let $Z\in \mathscr{C}$. We need to show that $a|_Z = a_Z$. We have
     $$H^0(Z, \mathcal{F}|_Z)\longhookrightarrow \prod_{x \text{ is closed in } Z}(\mathcal{F}|_Z)_x \longhookrightarrow \prod_{x\text{ is closed in } Z} \varprojlim_n\: (\mathcal{F}|_Z)_x/\mathfrak{m}_x^n(\mathcal{F}|_Z)_x,$$
     where $\mathfrak{m}_x$ is the maximal ideal of the ring $\mathcal{O}_{Z, x}$. Note that the second arrow is injective by Krull's intersection theorem. Observe that $H^0(\tilde{x}, \mathcal{F}|_{\tilde{x}}) = (\mathcal{F}|_Z)_x/\mathfrak{m}_x^n(\mathcal{F}|_Z)_x,$
     where $\tilde{x}$ is the $n$-th power thickening of $\overline{\{x\}}$ in $Z$. Thus it is suffcient to prove that $a|_{\tilde{x}} = a_Z|_{\tilde{x}}$
     for any closed point $x$ of $Z$ and any power thickening $\tilde{x}$ of $\overline{\{x\}}$ in $Z$. Let $Y\in \mathscr{C}$ of codimension $i$ in $X$ such that $Z\longhookrightarrow Y$. Let $D$ be an admissible hypersurface section such that $x\in D$. Since $\mathrm{codim}_XD\cap Y \leq i + 1$, there exists $Z'\in\mathscr{C}$ of codimension $j$ in $X$ such that $Z'\longhookrightarrow D_{[m]}\cap Y_{[m]}$ for some $m > 0$ and $\tilde{x}\longhookrightarrow Z'$. Then, since $a_Z = a_{Y_{[m]}}|_Z$, $a_{Z'} = a_{Y_{[m]}}|_{Z'}$, and $a_{D_{[m]}} = a|_{D_{[m]}}$, we obtain 
     $$a|_{\tilde{x}} - a_Z|_{\tilde{x}} = a|_{\tilde{x}} - a_{Y_{[m]}}|_{\tilde{x}} = a|_{\tilde{x}} - a_{Z'}|_{\tilde{x}} = a_{D_{[m]}}|_{\tilde{x}} - a_{Z'}|_{\tilde{x}} = (a_{D_{[m]}}|_{Z'})|_{\tilde{x}} - a_{Z'}|_{\tilde{x}} = a_{Z'}|_{\tilde{x}} - a_{Z'}|_{\tilde{x}} = 0.$$
     Hence $a|_Z = a_Z$. Thus we have constructed a well-defined map $\lim_{Z\in \mathscr{C}}H^0(Z, \mathcal{F}|_Z)\longrightarrow H^0(X, \mathcal{F})$, which can be easily seen to be inverse to $\iota$.
\end{proof}

\subsection{Limits of adelic groups}
\begin{lemma}\label{LimPowSub}
    Let $X$ be a finite-dimensional catenary locally equidimensional Noetherian scheme, let $\mathcal{F}$ be a coherent sheaf on $X$, and let $I = (i_0,\ldots)\subset \{0, 1, \ldots, \dim X\}$. Then the natural map
    $$\A_I(X, \mathcal{F}) \longrightarrow \lim_{Z\in \mathscr{P}_{i_0}} \A_{I|_Z}(Z, \mathcal{F}|_Z)$$
    is an isomorphism. 
\end{lemma}
\begin{proof}
By definition, $\A_I(X, \mathcal{F}) = \A(S(X, I), \mathcal{F})$. Since $\mathcal{F}$ is coherent,
    \begin{equation}\label{eq20}
    \A(S(X, I), \mathcal{F}) =
    \prod_{\substack{p\in X\\\mathrm{ codim}\:p = i_0}}\varprojlim_{l}\A(\leftidx{_p}S(X, I), \mathcal{F}_p/\mathfrak{m}_p^l\mathcal{F}_p),
    \end{equation}
    where the equality holds, since each flag in $S(X, I)$ starts with a point of codimension $i_0$.
    
    Fix some $p\in X$ of codimension $i_0$. Let $Z = \overline{\{p\}}$, and let $Z_{[l]}$ be the $l$-th power thickening of $Z$. 
    By lemma \ref{LemmaIandIZ}, we have  $\leftidx{_p}S(X, I) = S\big(Z_{[l]}, (I\setminus i_0)|_{Z_{[l]}}\big)$. By \cite[Prop. 3.1.2]{Hu} and by lemma \ref{IandI0}, there is an equality 
    $$\A(\leftidx{_p}S(X, I), \mathcal{F}_p/\mathfrak{m}_p^l\mathcal{F}_p) = \A_{I|_{Z_{[l]}}}(Z_{[l]}, \mathcal{F}|_{Z_{[l]}}).$$
    From this and from \eqref{eq20} the assertion follows.
\end{proof}

\begin{rem}
    Note that for a general quasicoherent sheaf lemma \ref{LimPowSub} does not hold, since filtered colimits do not commute with limits in general.
\end{rem}

\begin{rem}
    Observe that if $i < i_0$, then $\A_I(X, \mathcal{F})\neq \lim\limits_{Z\in \mathscr{P}_i}\A_{I|_Z}(Z, \mathcal{F}|_Z)$, since the restrictions of the terms on the right-hand side to an element of $\mathscr{P}_{i_0}$ might be different.
\end{rem}

\begin{lemma}\label{jk}
    Let $X$ be an irreducible finite-dimensional catenary Noetherian scheme, let $\mathcal{F}$ be a coherent sheaf on $X$, and let $I = (i_0, \ldots)\subset \{0, 1, \ldots, \dim X\}$. Let $m > 1$ and let $j_1, j_2, \ldots, j_m\in \{0, 1, \ldots, \dim X\}$ such that $j_1 < j_2<\ldots < j_m\leq i_0$. Let $k\in \{1, 2, \ldots, m\}$. Then the natural map
    $$\lim_{Z\in \mathscr{P}_{j_1,\ldots, j_k, \ldots, j_m}}\A_{I|_Z}(Z, \mathcal{F}|_Z)\longrightarrow \lim_{Z\in\mathscr{P}_{j_1, \ldots, \widehat{j_k}, \ldots, j_m}}\A_{I|_Z}(Z, \mathcal{F}|_Z),$$
    is injective, where $\:\widehat{}\:$ means that we skip the corresponding index.
\end{lemma}
\begin{proof}
    If $k\geq 2$, then for any $Y\in \mathscr{P}_{j_k}$ there exists $Z\in \mathscr{P}_{j_1}$ such that $Y\hookrightarrow Z$. From this observation the statement easily follows. If $k = 1$, then for any $Y\in \mathscr{P}_{j_1}$ and any non-zero $a_Y\in \A_{I|_Y}(Y, \mathcal{F}|_Y)$ we can find $Z\in \mathscr{P}_{j_2 - j_1}(Y)$ such that $a_Y|_Z\neq 0$ using lemma \ref{LimPowSub}. Hence we can deduce the lemma in this case. 
\end{proof}

\begin{prop}\label{LimDiag}
    Let $X$ be an irreducible scheme of finite type over a field $\Bbbk$ with $\dim X \geq 2$ and let $\mathcal{F}$ be a coherent sheaf on $X$. Let $I = (i_0, \ldots)\subset \{2, 3,\ldots, \dim X\}$ and let $i, j\in \{1, 2, \ldots, \dim X\}$ such that $i < j \leq i_0$. Then the natural map
    $$\A_I(X, \mathcal{F}) \longrightarrow \lim_{Z \in \mathscr{P}_{i, j}}\A_{I|_Z}(Z, \mathcal{F}|_Z)$$
    is an isomorphism.
\end{prop}
\begin{proof}
Let us construct the inverse to the morphism 
$$\A_I(X, \mathcal{F}) \longrightarrow \lim_{Z \in \mathscr{P}_{i, j}}\A_{I|_Z}(Z, \mathcal{F}|_Z).$$
If $j = i_0$, then the inverse map is given by a composition
\begin{equation}\label{eqComposition}
\lim_{Z \in \mathscr{P}_{i, i_0}}\A_{I|_Z}(Z, \mathcal{F}|_Z)\longrightarrow \lim_{Z \in \mathscr{P}_{i_0}}\A_{I|_Z}(Z, \mathcal{F}|_Z)\overset{\sim}{\longrightarrow} \A_I(X, \mathcal{F}),
\end{equation}
where the first map is induced by inclusion $\mathscr{P}_{i_0}\subset \mathscr{P}_{i, i_0}$ and the second map is isomorphism by lemma \ref{LimPowSub}. Indeed, since the composition
$$\A_I(X, \mathcal{F}) \longrightarrow \lim_{Z \in \mathscr{P}_{i, j}}\A_{I|_Z}(Z, \mathcal{F}|_Z) \longrightarrow \A_I(X, \mathcal{F})$$
is the identity map by construction and since the first map in \eqref{eqComposition} is injective by lemma \ref{jk}, we obtain the desired isomorphism.

Assume that $j < i_0$. Let $(a_Z)_{Z\in \mathscr{P}_{i, j}} \in \lim_{Z \in \mathscr{P}_{i, j}}\A_{I|_Z}(Z, \mathcal{F}|_Z)$. Let $Y\in \mathscr{P}_{i_0}$. Assume that $Z_1, Z_2\in \mathscr{P}_j$ such that $Y\longhookrightarrow Z_1\cap Z_2$. Let $U$ be an open affine subscheme of $X$ such that $U\cap Y\neq \varnothing$. Then $U$ is irreducible and dense in $X$, and $Y = \overline{U\cap Y}$, $Z_1 = \overline{U\cap Z_1}$, and $Z_2 = \overline{U\cap Z_2}$. By assumptions, $\dim U \geq 2$ and $\mathrm{codim}_U(U\cap (Z_1\cup Z_2)) = j \geq i + 1 \geq 2$. We can assume that $U$ is a subscheme in some projective space. Then there exists an irreducible closed subscheme $H\subset U$ such that $U\cap (Z_1\cup Z_2)\longhookrightarrow H$,  $\mathrm{codim}_H(U\cap (Z_1\cup Z_2)) = \mathrm{codim}_U(U\cap (Z_1\cup Z_2)) - 1 = j - 1$ by \cite[Lemma 3.5]{GK} for the infinite field $\Bbbk$ and by \cite[Theorem 5.7]{GK} for the finite field $\Bbbk$. We can continue this procedure of taking irreducible closed subschemes of smaller dimension containing $U\cap (Z_1\cup Z_2)$ inductively, and in the end we obtain an irreducible closed subscheme $H_0$ of $U$ such that $U\cap (Z_1\cup Z_2)\longhookrightarrow H_0$, $\mathrm{codim}_{H_0}(U\cap (Z_1\cup Z_2)) = j - i$. Let $Z_0\in \mathscr{P}_i$ be a power thickening of the integral subscheme $\overline{(H_0)_{red}}$ of $X$ such that $Z_1\cup Z_2\longhookrightarrow Z_0$. 
    Define 
    $$a_Y = a_{Z_1}|_Y\in \A_{I|_{Y}}(Y, \mathcal{F}|_{Y}).$$
    This is well-defined, since $a_{Z_1}|_Y - a_{Z_2}|_Y = a_{Z_0}|_Y - a_{Z_0}|_Y = 0.$
    Since for any $Y\in \mathscr{P}_{i_0}$ there exists $Z\in \mathscr{P}_j$ such that $Y\longhookrightarrow Z$, we obtain a well-defined map
    $$\lim_{Z \in \mathscr{P}_{i, j}}\A_{I|_Z}(Z, \mathcal{F}|_Z)\longrightarrow \lim_{Z \in \mathscr{P}_{i, j, i_0}}\A_{I|_Z}(Z, \mathcal{F}|_Z).$$
    Moreover, this map is injective by the explicit construction. Since the composition
    $$\A_I(X, \mathcal{F})\longrightarrow \lim_{Z \in \mathscr{P}_{i, j}}\A_{I|_Z}(Z, \mathcal{F}|_Z)\longhookrightarrow \lim_{Z \in \mathscr{P}_{i, j, i_0}}\A_{I|_Z}(Z, \mathcal{F}|_Z)\longhookrightarrow \lim_{Z \in \mathscr{P}_{ i_0}}\A_{I|_Z}(Z, \mathcal{F}|_Z) =\A_I(X, \mathcal{F})$$
    is the identity map, where the second inclusion follows from lemma \ref{jk} and the equality follows from lemma \ref{LimPowSub}, we obtain the desired isomorphism.
\end{proof}

\subsection{Intersections of adelic groups on a normal projective surface}

\begin{prop}\label{IntersectiondimX}
    Let $X$ be a Cohen--Macaulay projective variety over a field $\Bbbk$. Let $\mathcal{F}$ be a locally free sheaf on $X$. Then 
    $$\A_{\dim X - 1}(X, \mathcal{F})\cap \A_{\dim X}(X, \mathcal{F}) = H^0(X, \mathcal{F}).$$
\end{prop}
\begin{proof}
    Since $\A_{(\dim X -1)|_Z}(Z, \mathcal{F}|_Z) = 0$ and $\A_{(\dim X -1, \dim X)|_Z}(Z, \mathcal{F}|_Z) = 0$ for any $Z\in \mathscr{P}_{\dim X}$, we obtain 
    \begin{eqnarray*}
        \A_{\dim X - 1}(X, \mathcal{F})\cap \A_{\dim X}(X, \mathcal{F}) =  \ker\big(\A_{\dim X - 1}(X, \mathcal{F})\oplus \A_{\dim X}(X, \mathcal{F})\longrightarrow \A_{(\dim X - 1, \dim X)}(X, \mathcal{F})\big) \\
        = \lim_{Z\in \mathscr{P}_{\dim X -1, \dim X}}\ker\big(\A_{(\dim X -1)|_Z}(Z, \mathcal{F}|_Z)\oplus \A_{(\dim X)|_Z}(Z, \mathcal{F}|_Z)\longrightarrow \A_{(\dim X - 1, \dim X)|_Z}(Z, \mathcal{F}|_Z)\big)
    \end{eqnarray*}
    by lemma \ref{LimPowSub} and proposition \ref{LimDiag}. 
    Observe that 
    $$ \A_{(\dim X -1)|_Z}(Z, \mathcal{F}|_Z)\oplus \A_{(\dim X)|_Z}(Z, \mathcal{F}|_Z)\longrightarrow \A_{(\dim X - 1, \dim X)|_Z}(Z, \mathcal{F}|_Z)$$
    is the complex of reduces adeles $\A_{\mathrm{red}}^{\sbullet}(Z, \mathcal{F}|_Z)$ if $Z\in\mathscr{P}_{\dim X - 1}$ or $Z\in \mathscr{P}_{\dim X}$. Thus
    \begin{eqnarray*}
    \ker\big(\A_{(\dim X -1)|_Z}(Z, \mathcal{F}|_Z)\oplus \A_{(\dim X)|_Z}(Z, \mathcal{F}|_Z)\longrightarrow \A_{(\dim X - 1, \dim X)|_Z}(Z, \mathcal{F}|_Z)\big) = H^0(Z, \mathcal{F}|_Z)
    \end{eqnarray*}
    by theorem \ref{CohomologyRedAdelicComp}. Therefore
    $$\A_{\dim X - 1}(X, \mathcal{F})\cap \A_{\dim X}(X, \mathcal{F}) = \lim_{Z\in \mathscr{P}_{\dim X -1, \dim X}}H^0(Z, \mathcal{F}|_Z) = H^0(X, \mathcal{F}),$$
    where the last equality follows from theorem \ref{Hjlim}.
\end{proof}

\begin{thm}\label{ThmDimX2}
    Let $X$ be a normal projective surface over a field $\Bbbk$. Let $\mathcal{F}$ be a locally free sheaf on $X$ and let $I, J\subset \{0, 1, 2\}$. Then there is an equality
    $$\A_I(X, \mathcal{F})\cap \A_J(X, \mathcal{F}) = \A_{I\cap J}(X, \mathcal{F}).$$
\end{thm}
\begin{proof}
If $0\in I\cap J$, then 
    $$\A_I(X, \mathcal{F})\cap \A_J(X, \mathcal{F}) = \A_{I\setminus 0}(X, \mathcal{F}_\eta)\cap \A_{J\setminus 0}(X, \mathcal{F}_\eta) = \varinjlim_{D}\:\A_{I\setminus 0}\big(X, \mathcal{F}(D)\big)\cap \A_{J\setminus 0}\big(X, \mathcal{F}(D)\big),$$
    where the first equality follows from lemma \ref{IandI0} and the direct limit is taken over Cartier divisors on $X$. Since $\mathcal{F}(D)$ is a locally free sheaf, the assertion is reduced to the case $I\setminus 0$ and $J\setminus 0$. 
    
    If $0\in I\setminus J$, then 
    $$\A_I(X, \mathcal{F})\cap \A_J(X, \mathcal{F}) = \A_I(X, \mathcal{F})\cap \A_{(1, 2)}(X, \mathcal{F})\cap \A_J(X, \mathcal{F}) = \A_{I\setminus 0}(X, \mathcal{F})\cap \A_J(X, \mathcal{F}),$$
    where the last equality follows from theorem \ref{TheoremIcapJI0}, so the assertion is reduced to the case $I\setminus 0$ and $J$. The case $0\in J\setminus I$ is analogous.

    Therefore, we can assume that $0\notin I, J$. We can also assume that $(1, 2)\neq I, J$, otherwise the assertion is obvious. The only remaining case is $I = (1)$, $J =(2)$, which follows from proposition \ref{IntersectiondimX}, since a normal surface is Cohen--Macaulay by Serre's criterion for normality.
\end{proof}

\subsection{Cohomology groups of a curtailed adelic complex}\label{sect6.1}

\begin{lemma}\label{LimE}
    Let $X$ be a finite-dimensional catenary locally equidimensional Noetherian scheme, let $\mathcal{F}$ be a coherent sheaf on $X$, and let $I \subset \{1, \ldots, \dim X\}$. Then the natural map
    $$\A_I(X, \mathcal{F})\longrightarrow \lim_{Z\in \mathscr{E}}\A_{I|_Z}(Z, \mathcal{F}|_Z)$$
    is an isomorphism. 
\end{lemma}
\begin{proof}
    There is a commutative diagram
    \begin{center}
        \begin{tikzcd}[column sep=tiny]
            & \lim\limits_{Z\in \mathscr{E}}\A_{I|_Z}(Z, \mathcal{F}|_Z) \arrow{rd} & \\
            \A_I(X, \mathcal{F}) \arrow{rr}{\sim} \arrow{ru}& & \lim\limits_{Z\in \mathscr{P}_{i_0}}\A_{I|_Z}(Z, \mathcal{F}|_Z),
        \end{tikzcd}
    \end{center}
    where the right arrow is induced by the inclusion $\mathscr{P}_{i_0}\subset \mathscr{E}$ and the bottom arrow is an isomorphism by lemma \ref{LimPowSub}. The right arrow is injective, since every element in $\A_{I|_Z}(Z, \mathcal{F}|_Z)$ for $Z\in \mathscr{E}$ with $\mathrm{codim}\: Z \leq i_0$ is uniquely determined by its restrictions to elements of $\mathscr{P}_{i_0-\mathrm{codim}\: Z}(Z)$, and $\A_{I|_Z}(Z, \mathcal{F}|_Z) = 0$ for $Z\in \mathscr{E}$ with $\mathrm{codim}\: Z > i_0$. Thus the left arrow is an isomorphism.
\end{proof}

\begin{prop}\label{Complwithout0}
    Let $X$ be an irreducible Cohen--Macaulay projective scheme over a countable field $\Bbbk$ and let $\mathcal{F}$ be a locally free sheaf on $X$. Consider a curtailed adelic complex
    $$\mathcal{C}(\mathcal{F})^{\sbullet}\colon \A_1(X, \mathcal{F})\oplus \ldots\oplus\A_{\dim X}(X, \mathcal{F})\longrightarrow \ldots\longrightarrow \A_{(1,2,\ldots, \dim X)}(X, \mathcal{F}),$$
    where 
    $$\mathcal{C}(\mathcal{F})^m =  \bigoplus_{\substack{I\subset \{1, \ldots, \dim X\}\\ |I| = m + 1}}\A_I(X, \mathcal{F}).$$
    Then $H^k\big(\mathcal{C}(\mathcal{F})^{\sbullet}\big) \cong H^k(X, \mathcal{F})$ for $0\leq k < \dim X - 1$. In particular, $H^k\big(\mathcal{C}(\mathcal{F}(n))^{\sbullet}\big) = 0$ for $0 < k < \dim X - 1$ and $n\gg 0$.
\end{prop}
\begin{proof}
Let $\mathcal{C}(\mathcal{F})^{\sbullet}|_Z$ be the restriction of the complex $\mathcal{C}(\mathcal{F})^{\sbullet}$ to $Z\in \mathscr{E}$. Then $\mathcal{C}(\mathcal{F})^{\sbullet}|_Z = \A_{\mathrm{red}}^{\sbullet}(Z, \mathcal{F}|_Z)$, since $\dim Z < \dim X$. Thus $H^k\big(\mathcal{C}(\mathcal{F})^{\sbullet}|_Z\big) \cong H^k(Z, \mathcal{F}|_Z)$ by theorem \ref{CohomologyRedAdelicComp}.

Let us denote
$$\mathcal{Z}^{\sbullet} = \ker\big(\mathcal{C}(\mathcal{F})^{\sbullet}\overset{d}{\longrightarrow} \mathcal{C}(\mathcal{F})^{{\sbullet} + 1}\big), \quad \mathcal{B}^{\sbullet} = \mathrm{im}\big(\mathcal{C}(\mathcal{F})^{{\sbullet} - 1}\overset{d}{\longrightarrow} \mathcal{C}(\mathcal{F})^{{\sbullet}}\big),$$
$$\mathcal{Z}^{\sbullet}_Z = \ker\big(\mathcal{C}(\mathcal{F})^{\sbullet}|_Z\overset{d}{\longrightarrow} \mathcal{C}(\mathcal{F})^{{\sbullet} + 1}|_Z\big), \quad \mathcal{B}^{\sbullet}_Z = \mathrm{im}\big(\mathcal{C}(\mathcal{F})^{{\sbullet} - 1}|_Z\overset{d}{\longrightarrow} \mathcal{C}(\mathcal{F})^{{\sbullet}}|_Z\big),$$
where $Z\in \mathscr{E}$. There are exact sequences 
\begin{equation}\label{eqBZZZHZ}
0\longrightarrow \mathcal{B}_Z^{\sbullet}\longrightarrow \mathcal{Z}_Z^{\sbullet}\longrightarrow H^{\sbullet}(Z, \mathcal{F}|_Z)\longrightarrow 0,
\end{equation}
\begin{equation}\label{eqZZCZBZ}
0\longrightarrow \mathcal{Z}_Z^{\sbullet}\longrightarrow \mathcal{C}(\mathcal{F})^{\sbullet}|_Z\longrightarrow \mathcal{B}_Z^{{\sbullet} + 1}\longrightarrow 0.
\end{equation}

Fix some very ample invertible sheaf $\mathcal{L}$ on $X$. Define a {\it hypersurface section} as a zero scheme of a non-zero section of the sheaf $\mathcal{L}^{\otimes m}$ for some $m > 0$. Let $\mathscr{H}$ be an index category, where objects are hypersurface sections and morphisms are closed embeddings. Note that $\mathscr{H}$ is countable, since $\Bbbk$ is countable.

Observe that $\mathscr{H}\subset \mathscr{E}$. Note that for any $Z\in \mathscr{E}$ there exists $H\in \mathscr{H}$ such that $Z\longhookrightarrow H$ and for any $H_1, H_2\in \mathscr{H}$ there exists $H_3\in \mathscr{H}$ such that $H_1, H_2\longhookrightarrow H_3$. This means that $\mathscr{H}$ is cofinal in $\mathscr{E}$. Therefore 
\begin{equation}\label{eqEqualityofLim}
\lim_{Z\in \mathscr{E}}F(Z) = \lim_{Z\in \mathscr{H}}F(Z)
\end{equation}
for any diagram $F\colon \mathscr{E}^{op}\longrightarrow {\bf Ab}$. Note that $\mathscr{H}$ is filtered.

Now take $\lim_{Z\in \mathscr{H}}$ of sequences \eqref{eqBZZZHZ}. We obtain an exact sequence
$$0\longrightarrow \lim_{Z\in \mathscr{H}} \mathcal{B}_Z^{\sbullet} \longrightarrow \lim_{Z\in \mathscr{H}} \mathcal{Z}_Z^{\sbullet} \longrightarrow \lim_{Z\in \mathscr{H}} H^{\sbullet}(Z, \mathcal{F}|_Z)\longrightarrow \sideset{}{^1}\lim_{Z\in \mathscr{H}} \mathcal{B}_Z^{\sbullet} \longrightarrow \sideset{}{^1}\lim_{Z\in \mathscr{H}} \mathcal{Z}_Z^{\sbullet} \longrightarrow \sideset{}{^1}\lim_{Z\in \mathscr{H}} H^{\sbullet}(Z, \mathcal{F}|_Z),$$
where $\sideset{}{^1}\lim_{Z\in \mathscr{H}}$ is the first right derived functor of the functor $\lim_{Z\in \mathscr{H}}$. If $Z, Z'\in \mathscr{H}$ such that $Z\longhookrightarrow Z'$, then the restriction map $\mathcal{C}(\mathcal{F})^{\sbullet}|_{Z'}\longrightarrow \mathcal{C}(\mathcal{F})^{\sbullet}|_Z$ is surjective, since $\mathcal{C}(\mathcal{F})^{\sbullet}|_{Z'}$ and $\mathcal{C}(\mathcal{F})^{\sbullet}|_Z$ are complexes of reduced adeles. It follows that $\mathcal{B}_{Z'}^{\sbullet}\longrightarrow \mathcal{B}_Z^{\sbullet}$ is surjective. Therefore $\{\mathcal{B}^{\sbullet}_Z\}_{Z\in \mathscr{H}}$ is a Mittag--Leffler system. By \cite[Theorem 3.1]{Ro}, $\sideset{}{^1}\lim_{Z\in \mathscr{H}} \mathcal{B}_Z^{\sbullet} = 0$, since $\mathscr{H}$ is countable. We also have $\lim^1_{Z\in \mathscr{H}}H^{\sbullet}(Z, \mathcal{F}|_Z) = 0$ by \cite[Proposition 1.1]{Je}, since $H^k(Z, \mathcal{F}|_Z)$ is finite-dimensional for any $k\geq 0$. Thus we obtain that $\sideset{}{^1}\lim_{Z\in \mathscr{H}} \mathcal{Z}_Z^{\sbullet} = 0$ and the sequence 
\begin{equation}\label{eqLimBZZZHZ}
0\longrightarrow \lim_{Z\in \mathscr{H}} \mathcal{B}_Z^{\sbullet} \longrightarrow \lim_{Z\in \mathscr{H}} \mathcal{Z}_Z^{\sbullet} \longrightarrow \lim_{Z\in \mathscr{H}} H^{\sbullet}(Z, \mathcal{F}|_Z)\longrightarrow 0
\end{equation}
is exact.

Now take $\lim_{Z\in \mathscr{H}}$ of sequences \eqref{eqZZCZBZ}. Since $\sideset{}{^1}\lim_{Z\in \mathscr{H}} \mathcal{Z}_Z^{\sbullet} = 0$, we obtain an exact sequence 
$$0\longrightarrow \lim_{Z\in \mathscr{H}}\mathcal{Z}_Z^{\sbullet}\longrightarrow \lim_{Z\in \mathscr{H}}\mathcal{C}(\mathcal{F})^{\sbullet}|_Z\longrightarrow \lim_{Z\in \mathscr{H}}\mathcal{B}_Z^{{\sbullet} + 1}\longrightarrow 0.$$
Observe that $\mathcal{C}(\mathcal{F})^{\sbullet} = \lim_{Z\in \mathscr{H}}\mathcal{C}(\mathcal{F})^{\sbullet}|_Z$ by lemma \ref{LimE} and equality \eqref{eqEqualityofLim}. Since limits commute with kernels, we obtain that $\mathcal{Z}^{\sbullet} = \lim_{Z\in \mathscr{H}}\mathcal{Z}_Z^{\sbullet}$. Thus we obtain 
$$\mathcal{B}^{{\sbullet} + 1} = 
\lim_{Z\in \mathscr{H}}\mathcal{B}_Z^{{\sbullet} + 1}.$$
Therefore we obtain from \eqref{eqLimBZZZHZ}
$$H^{\sbullet}\big(\mathcal{C}(\mathcal{F})^{\sbullet}\big) \cong \lim_{Z\in \mathscr{H}} H^{\sbullet}(Z, \mathcal{F}|_Z),$$
since $H^k\big(\mathcal{C}(\mathcal{F})^{\sbullet}\big) = \mathcal{Z}^k/\mathcal{B}^k$. It follows from equality \eqref{eqEqualityofLim} and theorem \ref{Hjlim} that
$$H^k\big(\mathcal{C}(\mathcal{F})^{\sbullet}\big) \cong H^k(X, \mathcal{F}), \quad 0\leq k < \dim X - 1.$$
The second part of the assertion follows from Serre's vanishing theorem.
\end{proof}

\subsection{Intersections of adelic groups on a regular three-dimensional projective variety}

\begin{lemma}\label{FinExactSeq}
    Let $X$ be a projective scheme over a field $\Bbbk$ and let $\mathcal{F}$ be a locally free sheaf on $X$. Then there exists a positive integer $l_0$ such that for any integer $l\geq l_0$ there exists an exact sequence
    $$0\longrightarrow \mathcal{F}\longrightarrow \mathcal{O}_X(n)^{\oplus s}\longrightarrow\mathcal{O}_X(m)^{\oplus r}$$
    for some integers $s, r > 0$ and $n, m\geq l$.
\end{lemma}
\begin{proof}
    Note that there is a similar right exact sequence for $\mathcal{F}^\vee = \mathcal{H}om(\mathcal{F}, \mathcal{O}_X)$ with arbitrary large negative powers of $\mathcal{O}_X(1)$. Hence the dual sequence is left exact, and it remains to note that $\mathcal{F}^{\vee\vee}\cong \mathcal{F}$.
\end{proof}

\begin{prop}\label{ImplforThick}
    Let $X$ be a 
    subvariety in a 
    projective variety $Y$ and let $\tilde{X}$ be a thickening of $X$ in $Y$ locally defined by primary ideals. Let $\mathcal{N}$ be a nilpotent ideal sheaf on $\tilde{X}$ defining $X$. Let $I, J\subset \{0, 1, \ldots, \dim X\}$. Assume that $\mathcal{N}^j/\mathcal{N}^{j + 1}$ is a locally free sheaf on $X$ for any $j\geq 0$. Assume also that for any locally free sheaf $\mathcal{G}$ on $X$ there is an equality 
    $$\A_I(X, \mathcal{G})\cap \A_J(X, \mathcal{G}) = \A_{I\cap J}(X, \mathcal{G}).$$
    Then for any locally free sheaf $\mathcal{F}$ on $\tilde{X}$ there is an equality
    $$\A_I(\tilde{X}, \mathcal{F})\cap \A_J(\tilde{X}, \mathcal{F}) = \A_{I\cap J}(\tilde{X}, \mathcal{F}).$$
\end{prop}
\begin{proof}
Let $l$ be some large positive integer, which we specify later, and let $n, m\geq l$ be as in lemma \ref{FinExactSeq}, so there exists an exact sequence
\begin{equation}\label{eq21}
0\longrightarrow \mathcal{F}\longrightarrow \mathcal{O}_{\tilde{X}}(n)^{\oplus s}\longrightarrow\mathcal{O}_{\tilde{X}}(m)^{\oplus r}.
\end{equation}
Note that $\A_{I\cap J}(\tilde{X}, -)\longrightarrow \A_I(\tilde{X}, -)\cap \A_J(\tilde{X}, -)$ is a transformation of left exact functors on the category of locally free sheaves on $\tilde{X}$. Thus, by the five lemma and by \eqref{eq21}, we can reduce the problem to the case $\mathcal{F} = \mathcal{O}_{\tilde{X}}(n)$ for $n\geq l$.

There is a filtration $\mathcal{O}_{\tilde{X}}\supset \mathcal{N}\supset \mathcal{N}^2\supset\ldots \supset \mathcal{N}^k =0$
such that $\mathcal{N}^j/\mathcal{N}^{j + 1}$ is locally free on $X$ for any $j\geq 0$. Let $l$, which we mentioned above, be a positive integer such that
\begin{equation}\label{eqH1=0}
H^1\big(\tilde{X}, \mathcal{N}^j\otimes_{\mathcal{O}_{\tilde{X}}}\mathcal{O}_{\tilde{X}}(n)\big) = 0
\end{equation}
for any $0\leq j \leq k$ and any $n\geq l$. Such an $l$ exists by Serre's vanishing theorem. Denote $\mathcal{N}^j\otimes_{\mathcal{O}_{\tilde{X}}}\mathcal{O}_{\tilde{X}}(n)$ by $\mathcal{N}^j(n)$. There is an exact sequence of abelian groups
\begin{equation}\label{eq22}
0\longrightarrow \A_{I\cap J}(\tilde{X}, \mathcal{N}^{j + 1}(n)) \longrightarrow \A_{I\cap J}(\tilde{X}, \mathcal{N}^{j}(n))\longrightarrow \A_{I\cap J}\left(\tilde{X}, \left(\mathcal{N}^j/\mathcal{N}^{j + 1}\right)(n)\right)\longrightarrow 0
\end{equation}
obtained from the corresponding exact sequence of sheaves.
Since the functor $\A_{I\cap J}(\tilde{X}, -)$ is left exact, only the exactness in the last term of \eqref{eq22} is not obvious. If $I\cap J\neq \varnothing $, then this follows from the exactness of $\A_{I\cap J}(\tilde{X}, -)$. If $I\cap J = \varnothing$, then $\A_{I\cap J}(\tilde{X}, -) = H^0(\tilde{X}, -)$, and the exactness follows from \eqref{eqH1=0}.

Note that 
\begin{eqnarray}\A_I\big(\tilde{X}, \left(\mathcal{N}^j/\mathcal{N}^{j + 1}\right)(n)\big) = \A_I\big(X, \left(\mathcal{N}^j/\mathcal{N}^{j + 1}\right)(n)\big), \label{eqAIX}\\
\A_J\big(\tilde{X}, \left(\mathcal{N}^j/\mathcal{N}^{j + 1}\right)(n)\big)= \A_J\big(X, \left(\mathcal{N}^j/\mathcal{N}^{j + 1}\right)(n)\big),\label{eqAJX}
\end{eqnarray}
so the intersection $\A_I\big(\tilde{X}, \left(\mathcal{N}^j/\mathcal{N}^{j + 1}\right)(n)\big)\cap \A_J\big(\tilde{X}, \left(\mathcal{N}^j/\mathcal{N}^{j + 1}\big)(n)\right)$ is well-defined. Since $\mathcal{N}^j(n)$, $0\leq j\leq k$, are subsheaves of the sheaf $\mathcal{O}_{\tilde{X}}(n)$, the groups $\A_I\big(\tilde{X}, \mathcal{N}^j(n))\cap \A_J(\tilde{X}, \mathcal{N}^j(n)\big)$ are well-defined by theorem \ref{MainThmembeddingAdeles}. Thus we obtain a commutative diagram
\begin{center}
     \begin{tikzcd}[column sep=small, font = \normalfont]
         0\arrow{r} &  \A_{I\cap J}(\tilde{X}, \mathcal{G}_{j + 1}) \arrow{r}\arrow{d}  & \A_{I\cap J}(\tilde{X}, \mathcal{G}_j) \arrow{r}\arrow{d} & \A_{I\cap J}(\tilde{X}, \mathcal{G}_j/\mathcal{G}_{j + 1}) \arrow{r} \arrow{d} & 0 \\
          0\arrow{r}&  (\A_{I}\cap \A_{J})(\tilde{X}, \mathcal{G}_{j + 1}) \arrow{r}  & (\A_{I}\cap \A_{J})(\tilde{X}, \mathcal{G}_{j})\arrow{r} & (\A_{I}\cap \A_{J})(\tilde{X}, \mathcal{G}_j/\mathcal{G}_{j + 1}),&  
     \end{tikzcd}
 \end{center}
where $\mathcal{G}_j = \mathcal{N}^j(n)$, $(\A_{I}\cap \A_{J})(\tilde{X}, -) = \A_{I}(\tilde{X}, -)\cap \A_{J}(\tilde{X}, -)$, and the top exact sequence equals \eqref{eq22}. Note that the left and the middle vertical arrows are injective by theorem \eqref{MainThmembeddingAdeles}.
By \eqref{eqAIX}, \eqref{eqAJX} and by assumptions of the proposition, we obtain that the right vertical arrow is an isomorphism. Thus
$$\A_{I}\big(\tilde{X}, \mathcal{N}^j(n)\big)\cap \A_{J}\big(\tilde{X}, \mathcal{N}^j(n)\big) = \A_{I\cap J}\big(\tilde{X}, \mathcal{N}^j(n)\big) + \A_{I}\big(\tilde{X}, \mathcal{N}^{j + 1}(n)\big)\cap \A_{J}\big(\tilde{X}, \mathcal{N}^{j + 1}(n)\big).$$
Therefore
\begin{align*}
 \A_{I}\big(\tilde{X}, \mathcal{O}_{\tilde{X}}(n)\big)\cap \A_{J}\big(\tilde{X}, \mathcal{O}_{\tilde{X}}(n)\big) &= \A_{I\cap J}\big(\tilde{X}, \mathcal{O}_{\tilde{X}}(n)\big) + \A_{I\cap J}\big(\tilde{X}, \mathcal{N}(n)\big) + \ldots \\
 &+ \A_{I\cap J}\big(\tilde{X}, \mathcal{N}^{k - 1}(n)\big)
+ \A_{I}\big(\tilde{X}, \mathcal{N}^k(n)\big)\cap \A_{J}\big(\tilde{X}, \mathcal{N}^k(n)\big)\\
&= \A_{I\cap J}\big(\tilde{X}, \mathcal{O}_{\tilde{X}}(n)\big) ,
 \end{align*}
 where the last equality follows from the fact that $\mathcal{N}^k  =0$.
\end{proof}

The next lemma is a particular case, when the conditions of proposition \ref{ImplforThick} are satisfied.
\begin{lemma}\label{PartCase}
    Let $X$ be a subvariety of codimension $1$ of a regular variety $Y$. Let $X_{[n]}$ be the $n$-th power thickening of $X$ in $Y$ for some $n\geq 1$. Then $X_{[n]}$ is locally defined by primary ideals in $Y$ and $\mathcal{N}^j/\mathcal{N}^{j + 1}$ is a locally free sheaf on $X$ for any $j\geq 0$, where $\mathcal{N}$ is the nilpotent ideal sheaf on $X_{[n]}$ defining $X$.
\end{lemma}
\begin{proof}
    By \cite[Theorem A]{Mu}, we see that $\mathscr{S}_1 =\mathscr{P}_1$, since $X$ is locally given by the prime ideal of height $1$ in a regular ring. Thus the $n$-th power thickening $X_{[n]}$ of $X$ is an $n$-th symbolic power, from what follows $X_{[n]}$ is locally defined by primary ideals. 

    Let $\mathcal{J}_X$ and $\mathcal{J}_{X_{[n]}} = \mathcal{J}_X^n$ are ideal sheaves on $Y$ defining $X$ and $X_{[n]}$ respectively. Then $\mathcal{N} = \mathcal{J}_X/\mathcal{J}_X^n$. Hence $\mathcal{N}^j = \mathcal{J}_X^j/\mathcal{J}_X^n$ for $0\leq j\leq n$ and $\mathcal{N}^j = 0$ for $j > n$. Therefore $\mathcal{N}^j/\mathcal{N}^{j + 1} = \mathcal{J}_X^j/\mathcal{J}_X^{j + 1}$ for $j < n$ and $\mathcal{N}^j/\mathcal{N}^{j + 1} = 0$ for $j\geq n$. Since $Y$ is regular and $X$ is a divisor on $Y$, $X$ is locally regularly embedded in $Y$. Thus $\mathcal{J}_X^j/\mathcal{J}_X^{j + 1}$ is a locally free sheaf on $X$ for any $j\geq 0$.
\end{proof}

\begin{prop}\label{A1capAj}
    Let $X$ be a regular projective variety over a field $\Bbbk$ and let $\mathcal{F}$ be a locally free sheaf on $X$. Let $j$ be an integer with $1 < j \leq \dim X$. Then
    $$\A_1(X, \mathcal{F})\cap \A_j(X, \mathcal{F}) = H^0(X, \mathcal{F}).$$
\end{prop}
\begin{proof}
Note that 
$$\A_1(X, \mathcal{F}) = \lim_{Z\in \mathscr{S}_{1, j}}\A_{(1)|_Z}(Z, \mathcal{F}|_Z), \quad \A_j(X, \mathcal{F}) = \lim_{Z\in \mathscr{S}_{1, j}}\A_{(j)|_Z}(Z, \mathcal{F}|_Z),$$
$$\A_{(1, j)}(X, \mathcal{F}) = \lim_{Z\in \mathscr{S}_{1, j}}\A_{(1, j)|_Z}(Z, \mathcal{F}|_Z)$$
by lemma \ref{LimPowSub}, proposition \ref{LimDiag}, and lemma \ref{LemmaEqualityLimitSymbPow}. Therefore 
\begin{eqnarray*}
\A_1(X, \mathcal{F})\cap \A_j(X, \mathcal{F}) &=& 
\ker\big(\A_1(X, \mathcal{F})\oplus \A_j(X, \mathcal{F})\longrightarrow \A_{(1, j)}(X, \mathcal{F})\big)\\
&=&
\lim_{Z\in \mathscr{S}_{1, j}}\ker\big(\A_{(1)|_Z}(Z, \mathcal{F}|_Z)\oplus \A_{(j)|_Z}(Z, \mathcal{F}|_Z)\longrightarrow \A_{(1, j)|_Z}(Z, \mathcal{F}|_Z)\big).
\end{eqnarray*}
Let us denote $K_Z = \ker\big(\A_{(1)|_Z}(Z, \mathcal{F}|_Z)\oplus \A_{(j)|_Z}(Z, \mathcal{F}|_Z)\longrightarrow \A_{(1, j)|_Z}(Z, \mathcal{F}|_Z)\big).$
Observe that 
$${K_Z= \begin{cases}
    \A_0(Z, \mathcal{F}|_Z)\cap \A_{j - 1}(Z, \mathcal{F}|_Z),\: Z\in \mathscr{S}_1,\\
    \A_0(Z, \mathcal{F}|_Z),\: Z\in \mathscr{S}_j.
\end{cases}}$$
Note that the right-hand side of the equality is well-defined by theorem \ref{MainThmembeddingAdeles}. Since $\mathscr{R}^{sym}_{1, j}\subset \mathscr{S}_{1, j}$, there is a restriction map
$$\lim\limits_{Z\in \mathscr{S}_{1, j}}K_Z\overset{\rho}{\longrightarrow} \lim\limits_{Z\in \mathscr{R}_{1, j}^{sym}}K_Z.$$
Assume $(a_Z)_{Z\in \mathscr{S}_{1, j}}\in \ker\rho$ is non-zero. In particular, $a_Z = 0$ for any $Z\in \mathscr{R}^{sym}_{1, j}$. If there exists $Z'\in \mathscr{S}_j$ such that $a_{Z'}\neq 0$, then there exists $Z\in \mathscr{S}_1$ such that $Z'\longhookrightarrow Z$, which means that $a_Z\neq 0$, because $a_Z|_{Z'} = a_{Z'}\neq 0$. Hence we can assume that $a_Z\neq 0$ for some $Z\in \mathscr{S}_1$. 
Note that $\mathcal{F}|_Z(U)\longhookrightarrow (\mathcal{F}|_Z)_{\eta}$ by lemma \ref{lemmaEmbSheafPrimThik}, where $U\subset Z$ is an open set and $\eta$ is the generic point of $Z$.
Then 
$$\A_0(Z, \mathcal{F}|_Z)\cap \A_{j - 1}(Z, \mathcal{F}|_Z) = \big\{a \in (\mathcal{F}|_Z)_{\eta} \mid \forall\: p\in Z, \: \mathrm{codim}_Zp = {j - 1}\colon a\in (\mathcal{F}|_Z)_p\big\}.$$
The proof of this equality is similar to the proof of equations \eqref{eqA0capAJ} and \eqref{eqLocalIntersection}, since $\mathcal{F}|_Z$ is locally free. 
Let $p$ be a point of codimension $j - 1$ in $Z$ and let $Y = \overline{\{p\}}$. 
Then
$$\bigcap_{n\geq 1}\mathfrak{m}_p^n(\mathcal{F}|_Z)_p = 0$$
by \cite[Theorem 8.10]{Ma1}, where $\mathfrak{m}_p$ is the maximal ideal in the ring $\mathcal{O}_{Z, p}$. Since $a_Z\neq 0$, there exists some $n\geq 1$ such that $a_Z\notin \mathfrak{m}_p^n(\mathcal{F}|_Z)_p$. This means that $a_Z|_{Y'}\neq 0,$
where $Y'$ is $n$-th power thickening of $Y$ in $Z$. Let $Y_{(m)}$ be the $m$-th symbolic thickening of $Y$ in $X$ for $m \geq 1$. By proposition \ref{PropCofinalSymbPow} and remark \ref{RemarkSymPow}, there exist $m$ and some symbolic thickening $\tilde{Z}$ of $Z_{red}$ in $X$ such that
$$Y'\longhookrightarrow Y_{(m)} \longhookrightarrow \tilde{Z}.$$
Since $a_{Y_{(m)}}|_{Y'} = a_{\tilde{Z}}|_{Y'} = a_{Z}|_{Y'}\neq 0,$
we obtain 
\begin{equation}\label{eqaYneq0}
a_{Y_{(m)}} \neq 0.
\end{equation}

Fix some very ample invertible sheaf $\mathcal{L}$ on $X$. Define a {\it hypersurface section} as a zero scheme of a non-zero section of the sheaf $\mathcal{L}^{\otimes m}$ for some $m > 0$.
Let $H$ be a hypersurface section of $X$ such that $H$ is a regular variety, $Z\cap H\neq \varnothing$, and $\dim Z\cap H = \dim Z - 1$. Such an $H$ exists by \cite[Theorem 3.6]{GK} for the infinite field $\Bbbk$ and by \cite[Corollary 4.7]{GK}, \cite[Theorem 5.7]{GK} for the finite field $\Bbbk$. Since $\dim Z\cap H = \dim Z - 1$, there exists a point $p\in Z\cap H$ of codimension $j - 1$ in $Z$. Let $Y_{(m)}$ be the $m$-th symbolic thickening of $\overline{\{p\}}$ in $X$ for some large $m$. Let $H_{(k)}$ be the $k$-th symbolic thickening of $H$ in $X$ such that $Y_{(m)}\longhookrightarrow H_{(k)}$, which exists by remark \ref{RemarkSymPow}. Then $a_{Y_{(m)}} = a_{H_{(k)}}|_{Y_{(m)}} = 0,$
since $H_{(k)}\in \mathscr{R}^{sym}_{1, j}$ and $a_{H_{(k)}} = 0$ by our assumptions. On the other hand, $a_{Y_{(m)}} \neq 0$ by \eqref{eqaYneq0}.
Thus we get a contradiction, from what follows $\ker\rho = 0$.

Let $Z\in \mathscr{R}^{sym}_{1, j}$ of codimension $1$ in $X$. Then
$$K_Z = \A_0(Z, \mathcal{F}|_Z)\cap \A_{j - 1}(Z, \mathcal{F}|_Z) = H^0(Z, \mathcal{F}|_Z),$$
where the second equality follows from theorem \ref{TheoremIcapJI0}, proposition \ref{ImplforThick}, and lemma \ref{PartCase}. If $Z'\in \mathscr{R}^{sym}_{1, j}$ of codimension $j$ in $X$ such that $Z'\longhookrightarrow Z$, then there is a commutative diagram
\begin{center}
    \begin{tikzcd}
        H^0(Z, \mathcal{F}|_Z) \arrow[-, double equal sign distance]{r}\arrow{d} & K_Z\arrow{d}\\
        H^0(Z', \mathcal{F}|_{Z'})\arrow[hook]{r} & K_{Z'},
    \end{tikzcd}
\end{center}
where the bottom horizontal arrow is injective by lemma \ref{lemmaEmbSheafPrimThik}, since $K_{Z'} = \A_0(Z', \mathcal{F}|_{Z'})$. Therefore
$$\lim\limits_{Z\in \mathscr{R}_{1, j}^{sym}}K_Z = \lim\limits_{Z\in \mathscr{R}_{1, j}^{sym}}H^0(Z, \mathcal{F}|_Z).$$
Thus we get 
$$\A_1(X, \mathcal{F})\cap \A_j(X, \mathcal{F}) = \lim\limits_{Z\in \mathscr{S}_{1, j}}K_Z\overset{\rho}{\longhookrightarrow}\lim\limits_{Z\in \mathscr{R}_{1, j}^{sym}}K_Z = \lim\limits_{Z\in \mathscr{R}_{1, j}^{sym}}H^0(Z, \mathcal{F}|_Z) = H^0(X, \mathcal{F}),$$
where the last equality follows from lemma \ref{EqualityRpowRsym} and theorem \ref{Hjlim}. Hence the assertion follows.
\end{proof}

\begin{rem}
    One can show analogously as in proposition \ref{A1capAj} that there is an equality $\A_1(X, \mathcal{F})\cap \A_J(X, \mathcal{F}) = H^0(X, \mathcal{F}),$
    where $X$ is a regular projective variety, $\mathcal{F}$ is a locally free sheaf on $X$, and $J\subset \{2, \ldots, \dim X\}$.
\end{rem}

\begin{thm}\label{ThmDimX3}
    Let $X$ be a regular three-dimensional projective variety over a countable field $\Bbbk$, let $\mathcal{F}$ be a locally free sheaf on $X$, and let $I, J\subset \{0, 1, 2, 3\}$. Then there is an equality
    $$\A_I(X, \mathcal{F})\cap\A_J(X, \mathcal{F}) = \A_{I\cap J}(X, \mathcal{F}).$$
\end{thm}
\begin{proof}
First, consider the case $0\in I, J$. Then
$$\A_I(X, \mathcal{F})\cap \A_J(X, \mathcal{F}) = \A_{I\setminus 0}(X, \mathcal{F}_\eta)\cap \A_{J\setminus 0}(X, \mathcal{F}_\eta) = \varinjlim_D\:\A_{I\setminus 0}\big(X, \mathcal{F}(D)\big)\cap \A_{J\setminus 0}\big(X, \mathcal{F}(D)\big)),$$
where the first equality follows from lemma \ref{IandI0} and the direct limit is taken over the direct system of Cartier divisors on $X$. Thus the assertion can be reduced to the case $I\setminus 0$ and $J\setminus 0$.

Second, consider the case $0\in I\setminus J$. The case $0\in J\setminus I$ is analogous. It follows that $J\subset \{1, 2, 3\}$. By theorem \ref{TheoremIcapJI0}, 
$$\A_I(X, \mathcal{F})\cap \A_J(X, \mathcal{F}) = \A_I(X, \mathcal{F})\cap \A_{(1, 2, 3)}(X, \mathcal{F}) \cap \A_J(X, \mathcal{F}) = \A_{I\setminus 0}(X, \mathcal{F})\cap \A_J(X, \mathcal{F}),$$
so the assertion is reduced to the case $I\setminus 0$ and $J$.

By the above observations, it is sufficient to consider the case $I, J\subset \{1, 2, 3\}$. We can assume that $I\neq \{1, 2, 3\}$ and $J\neq \{1, 2, 3\}$, otherwise the statement is clear. We can also assume that $I\not\subset J$ and $J\not\subset I$.

Consider the case $|I| = |J| = 1$. Then $I\cap J = \varnothing$ by our assumptions. The equality 
$$\A_i(X, \mathcal{F})\cap \A_j(X, \mathcal{F}) = H^0(X, \mathcal{F}),$$
where $I = (i)$ and $J = (j)$, follows from proposition \ref{A1capAj} and 
proposition \ref{IntersectiondimX}.

Now consider the case $|I| = |J| = 2$. The only possibilities up to permutation of $I$ and $J$ are
 $$I = (1, 2), \: J = (1, 3), \quad I = (1, 2), \: J = (2, 3), \quad I = (1, 3), \: J = (2, 3).$$
Since the proofs of the assertion will be similar in these cases, we can only consider the case $I = (1, 2), \: J = (2, 3)$. Let $n > 0$ be such that $H^1\big(X, \mathcal{F}(n)\big) = 0$. Let $A_{K} = \A_{K}\big(X, \mathcal{F}(n)\big)$
for $K\subset \{1, 2, 3\}$.
Then the complex
\begin{equation}\label{eqCurtAdelicCompdim3}
A_1\oplus A_2\oplus A_3 \overset{d}{\longrightarrow} A_{(1, 2)}\oplus A_{(1, 3)}\oplus  A_{(2, 3)}\overset{d}{\longrightarrow} A_{(1,2 , 3)}
\end{equation}
is exact in the middle term by proposition \ref{Complwithout0}. Since $\mathcal{F}\longhookrightarrow \mathcal{F}(n)$, there is an embedding $\A_{K}(X, \mathcal{F})\longhookrightarrow A_{K}$ for any $K\subset \{1, 2, 3\}$. Let $a\in \A_{(1, 2)}(X, \mathcal{F})\cap \A_{(2, 3)}(X, \mathcal{F})$. Consider the element $\tilde{a} = (a, 0, -a)\in A_{(1, 2)}\oplus A_{(1, 3)}\oplus  A_{(2, 3)}.$
Then $d\tilde{a} = 0$. Thus, from the exactness in the middle term of the complex \eqref{eqCurtAdelicCompdim3}, there exists an element $(r, t, s)\in A_1\oplus A_2\oplus A_3$
such that 
\begin{center}
        $$\begin{cases}
            t - r = a,\\
            s - r = 0,\\
            s - t = - a,
        \end{cases}
        \implies
        \begin{cases}
            r =s \in A_1\cap A_3,\\
            a = t - r \in A_2 + A_1\cap A_3.
        \end{cases}
        $$
    \end{center}
    Since $A_2\longhookrightarrow \A_{(0, 2)}(X, \mathcal{F})$ and
    $$A_1\cap A_3 = \A_1\big(X, \mathcal{F}(n)\big)\cap \A_3\big(X, \mathcal{F}(n)\big) = H^0\big(X, \mathcal{F}(n)\big)\subset \A_0(X, \mathcal{F}),$$
    where the second equality follows from the previous case, we obtain that $A_2 + A_1\cap A_3 \subset \A_{(0, 2)}(X, \mathcal{F}).$
    Hence
    $$a\in \A_{(1, 2)}(X, \mathcal{F})\cap \A_{(2, 3)}(X, \mathcal{F})\cap \A_{(0, 2)}(X, \mathcal{F}) = \A_{(1, 2)}(X, \mathcal{F})\cap\A_2(X, \mathcal{F}) = \A_2(X, \mathcal{F}).$$
    Therefore we obtain the desired equality
    $$\A_{(1, 2)}(X, \mathcal{F})\cap \A_{(2, 3)}(X, \mathcal{F}) = \A_2(X, \mathcal{F}).$$
    
    Now consider the remaining case $|I| = 2$, $|J| = 1$ (the case $|I| = 1$, $|J| = 2$ is similar). By our assumptions, $I\cap J = \varnothing$. The only possibilities are
 $$I = (1, 2), \: J = (3), \quad I = (1, 3), \: J = (2), \quad I = (2, 3), \: J = (1).$$
 As in the previous case, the proofs of these cases are analogous. Thus we can consider only the case $I = (1, 2), \: J = (3)$. Then
 \begin{align*}
 \A_{(1, 2)}(X, \mathcal{F})\cap \A_3(X, \mathcal{F}) &= \A_{(1, 2)}(X, \mathcal{F})\cap \A_{(2, 3)}(X, \mathcal{F})\cap \A_3(X, \mathcal{F}) = \A_{2}(X, \mathcal{F})\cap \A_3(X, \mathcal{F})\\
 &= H^0(X, \mathcal{F}),
 \end{align*}
 where the second and the third equalities follow from the previous cases.
\end{proof}

\end{document}